\documentclass[11pt]{article}
\usepackage{amssymb, authblk}
\usepackage{amsmath,bbm}
\usepackage{fullpage}
\usepackage{amsthm, hyperref}
\usepackage{graphicx, mathabx}
\usepackage{verbatim}
\usepackage{algorithm}
\usepackage{algpseudocode}
\usepackage[dvipsnames]{xcolor}
\usepackage{tikz}
\usetikzlibrary{arrows.meta}
\usetikzlibrary{decorations.pathreplacing}

\usepackage{natbib}

\usepackage{mathtools}

\DeclarePairedDelimiter\floor{\lfloor}{\rfloor}

\algnewcommand\algorithmicinput{\textbf{INPUT:}}
\algnewcommand\INPUT{\item[\algorithmicinput]}
\algnewcommand\algorithmicoutput{\textbf{OUTPUT:}}
\algnewcommand\OUTPUT{\item[\algorithmicoutput]}

\usepackage{hyperref}[]
\hypersetup{
    colorlinks=true,
    linkcolor=blue,
    filecolor=magenta,      
    urlcolor=cyan,
      citecolor=blue,
    }

\usepackage{cleveref}

\newtheorem{theorem}{Theorem}
\newtheorem{lemma}[theorem]{Lemma}
\newtheorem{proposition}[theorem]{Proposition}
\newtheorem{corollary}[theorem]{Corollary}

\newtheorem{definition}{Definition}
\newtheorem{remark}{Remark}
\newtheorem{assumption}{Assumption}

\allowdisplaybreaks

\DeclareMathOperator*{\argmin}{arg\,min}

\DeclareMathOperator*{\esssup}{ess\,sup}

\title{Optimal network online change point localisation}
\author[1]{Yi Yu}
\author[2]{Oscar Hernan Madrid Padilla}
\author[3]{Daren Wang}
\author[4]{Alessandro Rinaldo}
\affil[1]{Department of Statistics, University of Warwick}
\affil[2]{Department of Statistics, University California, Los Angeles}
\affil[3]{Department of Statistics, University of Notre Dame}
\affil[4]{Department of Statistics \& Data Science, Carnegie Mellon University}
\date{\today}

\begin{document}

\maketitle
\begin{abstract}
	  	
We study the problem of online network change point detection.  In this setting, a collection of independent Bernoulli networks is collected sequentially, and the underlying distributions change when a change point occurs.  The goal is to detect the change point as quickly as possible, if it exists, subject to a constraint on the number or probability of false alarms.  In this paper, on the detection delay, we establish a minimax lower bound and two upper bounds based on NP-hard algorithms and polynomial-time algorithms, i.e.
	\[
		\mbox{detection delay}  \begin{cases}
  			\gtrsim \log(1/\alpha) \frac{\max\{r^2/n, \, 1\}}{\kappa_0^2 n \rho}, \\
  			\lesssim  \log(\Delta/\alpha) \frac{\max\{r^2/n, \, \log(r)\}}{\kappa_0^2 n \rho}, & \mbox{with NP-hard algorithms}, \\
  			\lesssim  \log(\Delta/\alpha) \frac{r}{\kappa_0^2 n \rho}, & \mbox{with polynomial-time algorithms},
  		\end{cases}
	\]
	where $\kappa_0, n, \rho, r$ and $\alpha$ are the normalised jump size, network size, entrywise sparsity, rank sparsity and the overall Type-I error upper bound.  All the model parameters are allowed to vary as $\Delta$,  the location of the change point, diverges. 

The polynomial-time algorithms are novel procedures that we propose in this paper, designed for quick detection under two different forms of Type-I error control.  The first is based on controlling the overall probability of a false alarm  when there are no change points, and the second is based on specifying a lower bound on the expected time of the first false alarm.  Extensive experiments show that, under different scenarios and the aforementioned forms of Type-I error control, our proposed approaches outperform  state-of-the-art methods.  


\end{abstract}

\section{Introduction}\label{sec-introduction}

In this paper we are concerned with online change point detection in dynamic networks.  To be specific, we observe a sequence of independent adjacency matrices $\{A(t), \, t = 1, 2, \ldots\}$, with $\mathbb{E}\{A(t)\} = \Theta(t)$, for $t \in \mathbb{N}_+$.  If there exists $t^* \geq 2$, such that $\Theta(t^*) \neq \Theta(t^*-1)$, then we call $t^*$ a change point.  Our aim is to detect the existence of such change points as soon as they occur.  On the other hand, if there is no change point, then we would like to avoid false alarms.  To the best of our knowledge, this problem has not been theoretically studied in the existing statistical literature. 

The problem we described above is an abstractification of various real-life problems.  For instance, in cybersecurity, one monitors the internet or a system and wishes to detect malicious activity as early as it starts.  In finance, regulatory authorities oversee the markets and aim to stop unlawful activities at an early stage.  In epidemiology, public health sectors follow the spreading of a contagious disease in a community and target at knowing the spreading pattern changes as they happen.  

As a concrete example, we consider the Massachusetts Institute of Technology (MIT) cellphone data set \citep{eagle2006reality}.  The data set consists of human interactions  measured by the cellphone activity of the participants.  There were 96 participants that included students and faculty members at the MIT.  The data were taken from 14-Sept-2004 to 5-May-2005.  

We construct two experiments to evaluate our proposed methods and our competitors.  In the first example, we use the data from 14-Sept-2004 to 15-Feb-2005, which cover the MIT winter recess starting on 22-Dec-2004 and ending on 3-Jan-2005.  In our second example, we use the data from 1-Jan-2005 to 5-May-2005, which cover the spring recess starting on 26-Mar-2005 and ending on 3-Apr-2005.  In \Cref{fig-sec-intro}, we plot the interaction networks for a few representative dates.  A white dot means the corresponding row and column individuals interacted on the specific date, while a red dot means the lack of interaction.  For these two examples, our proposed method detects change points at 27-Dec-2004 and 31-Mar-2005, respectively.  Our competitors' change point estimators are around 30-Jan-2005 and 6-Apr-2005, respectively.  Our method is clearly the best at detecting the winter and spring recess periods.   Numerical details are explained in \Cref{sec:cellphone_data}.  

\begin{figure}[t!]
	\begin{center}
		\includegraphics[width=0.24\textwidth]{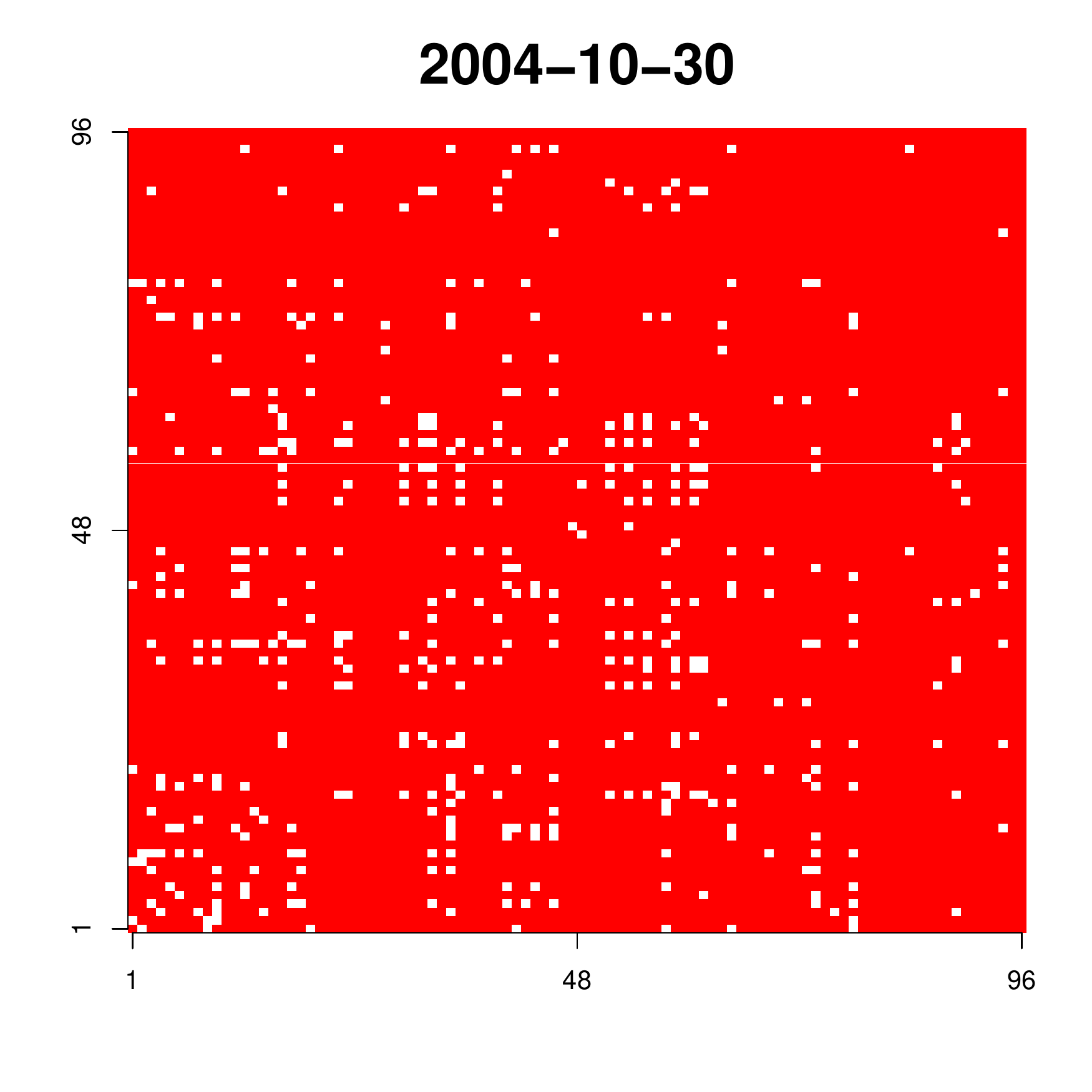} 
		\includegraphics[width=0.24\textwidth]{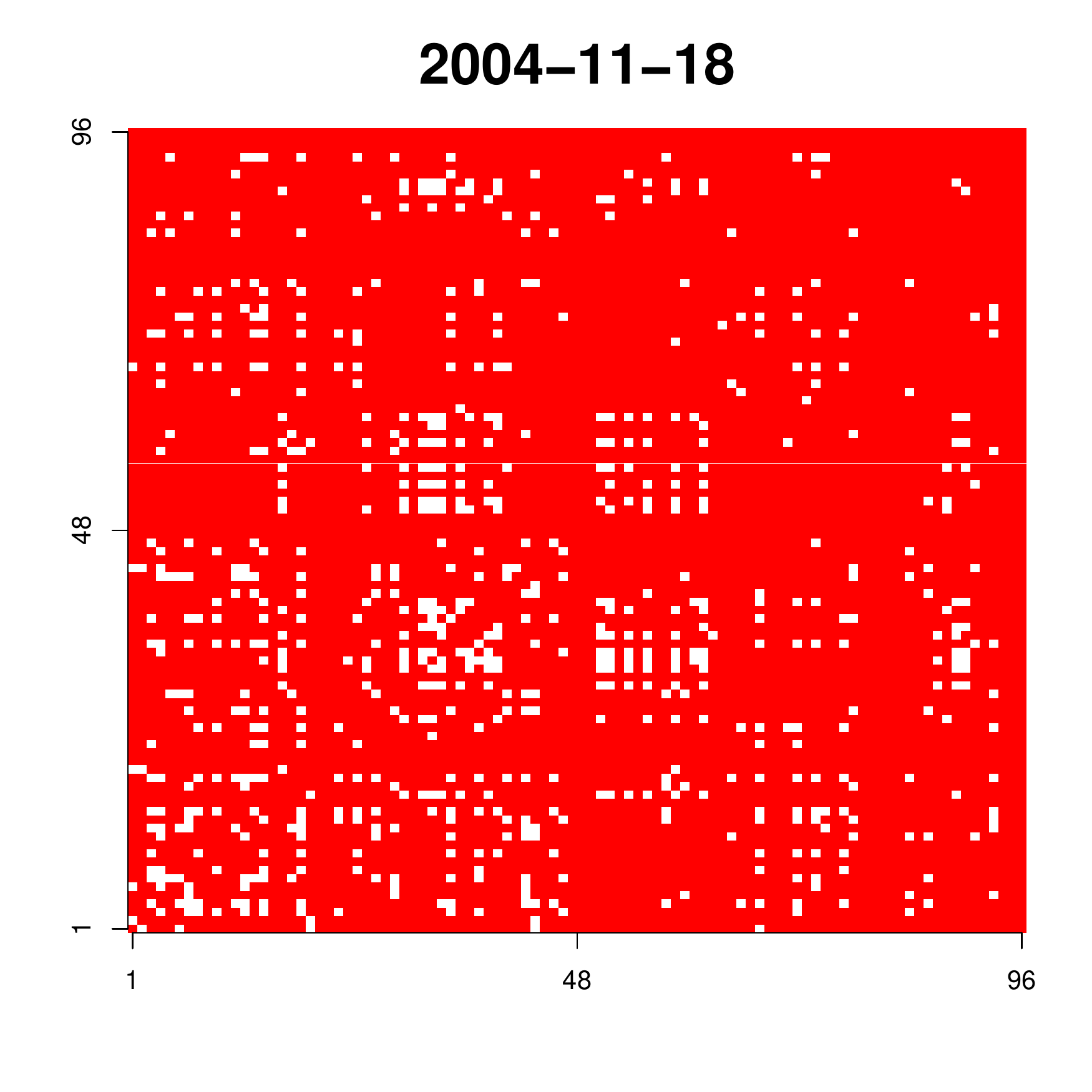} 	
		\includegraphics[width=0.24\textwidth]{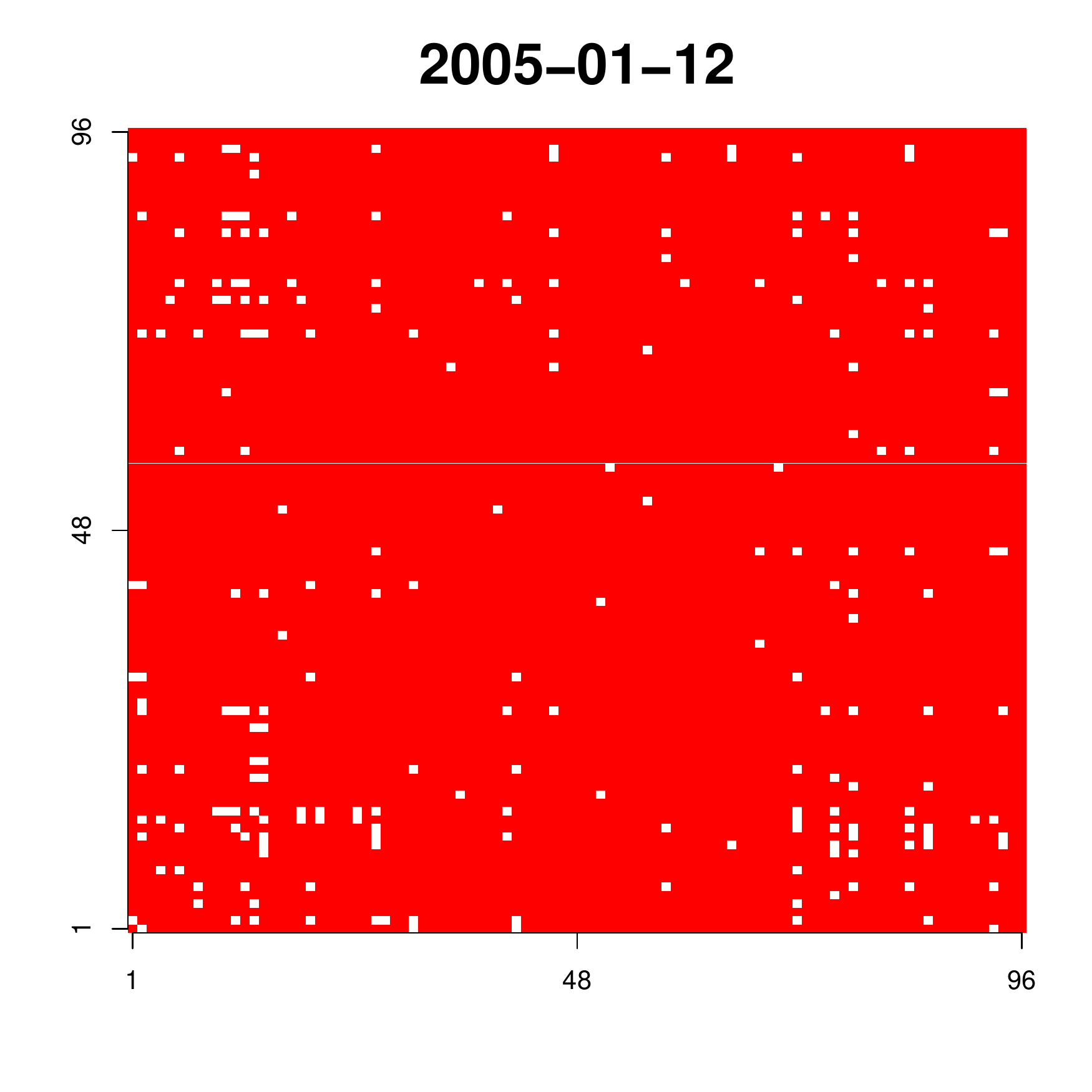} 
		\includegraphics[width=0.24\textwidth]{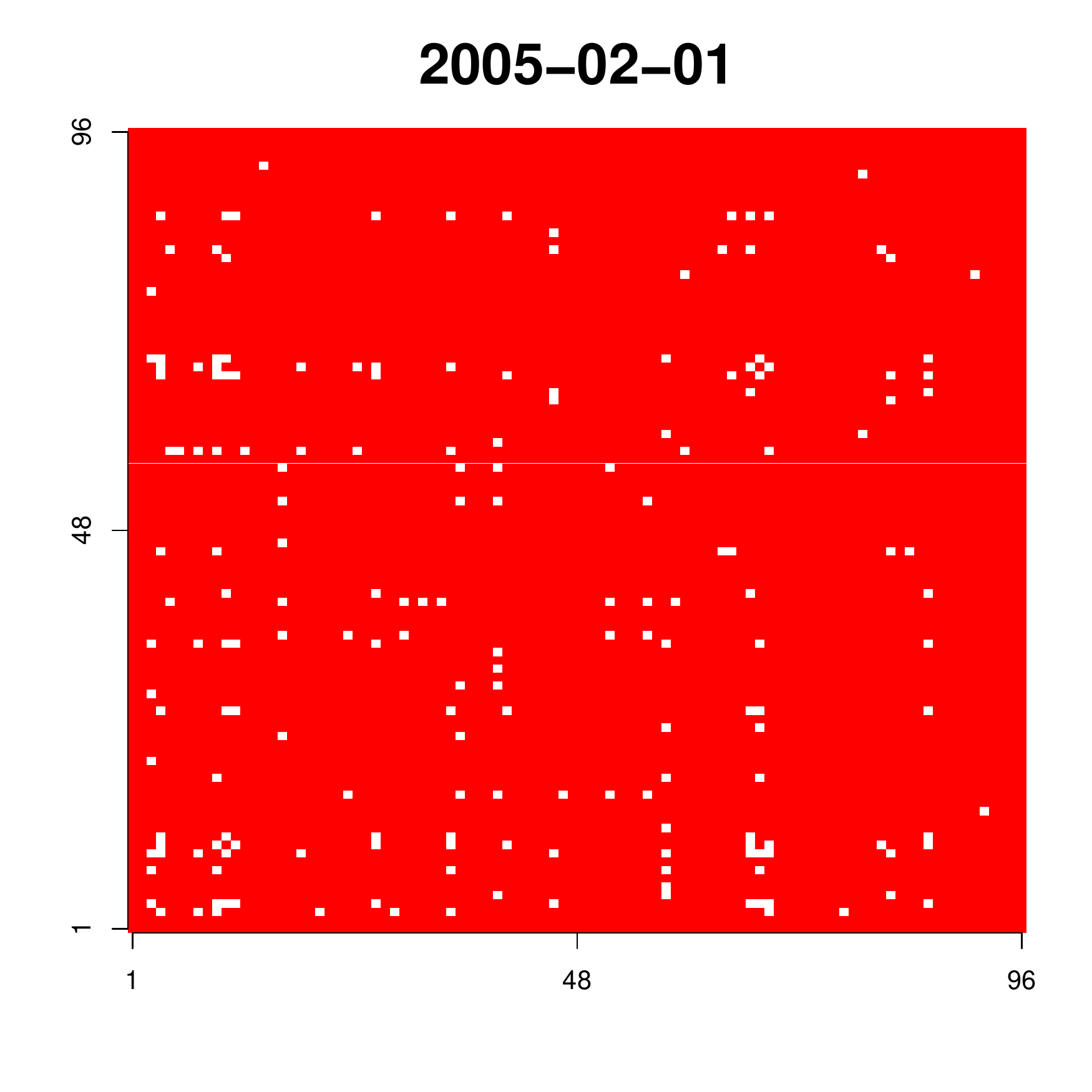} 		
		\includegraphics[width=0.24\textwidth]{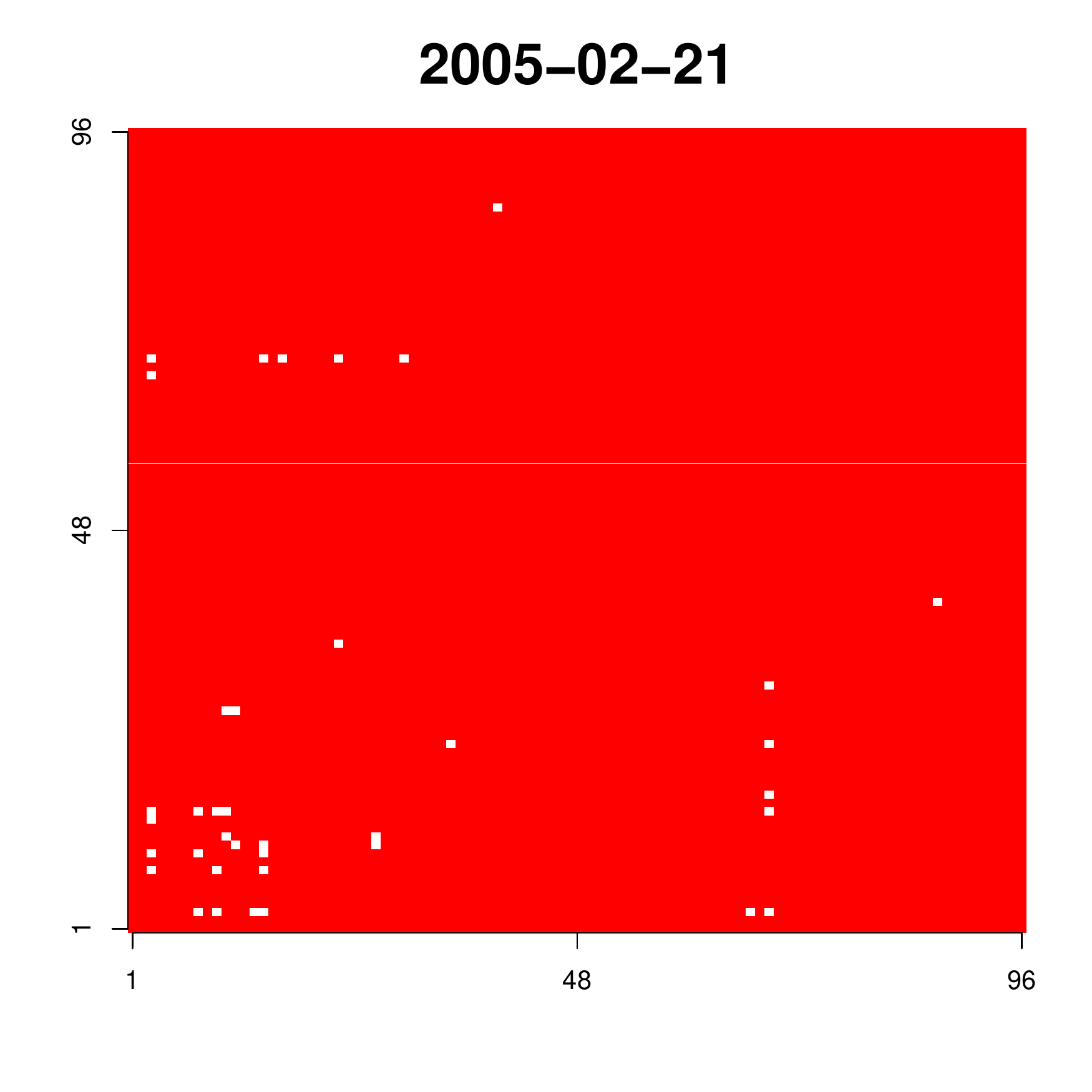} 
		\includegraphics[width=0.24\textwidth]{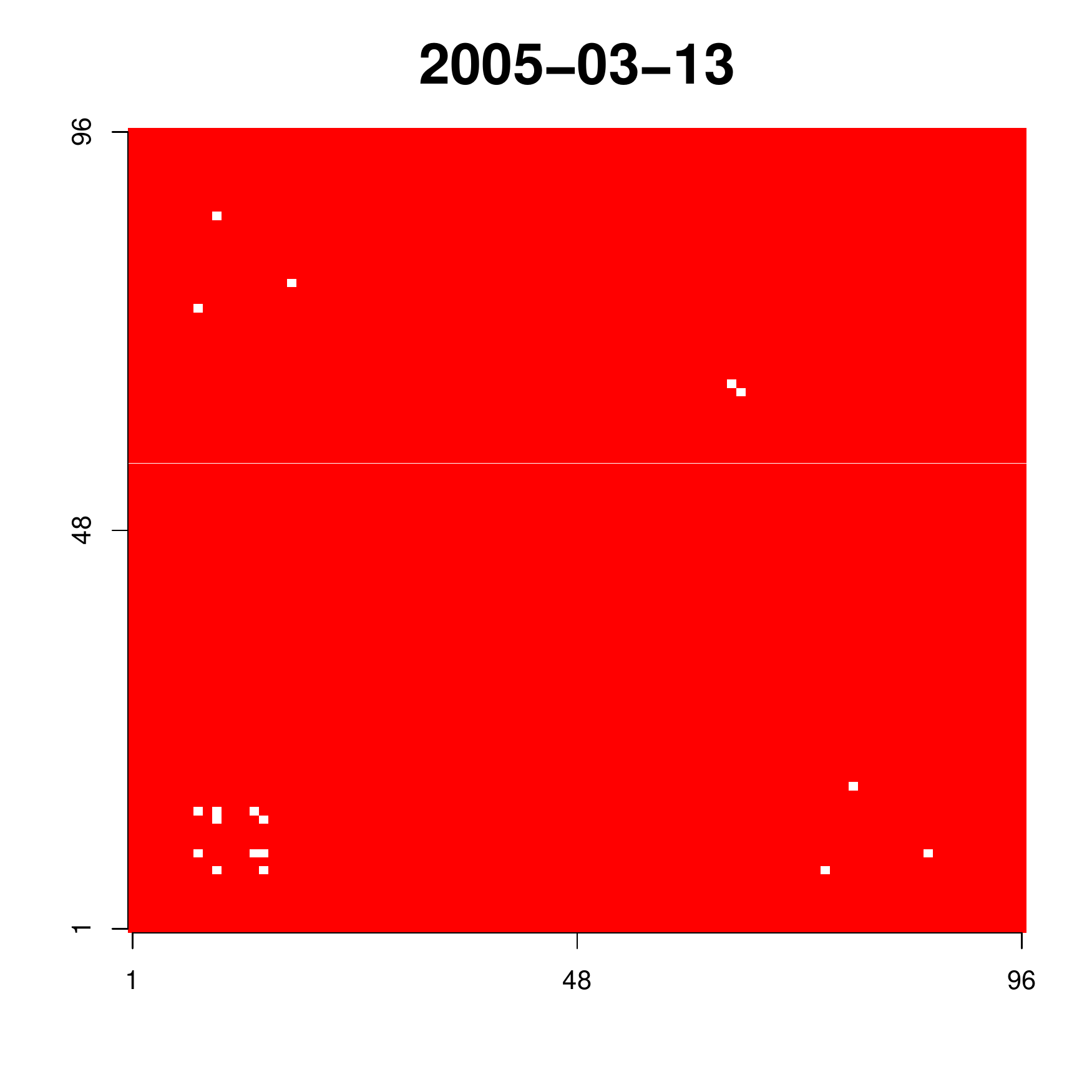} 	
		\includegraphics[width=0.24\textwidth]{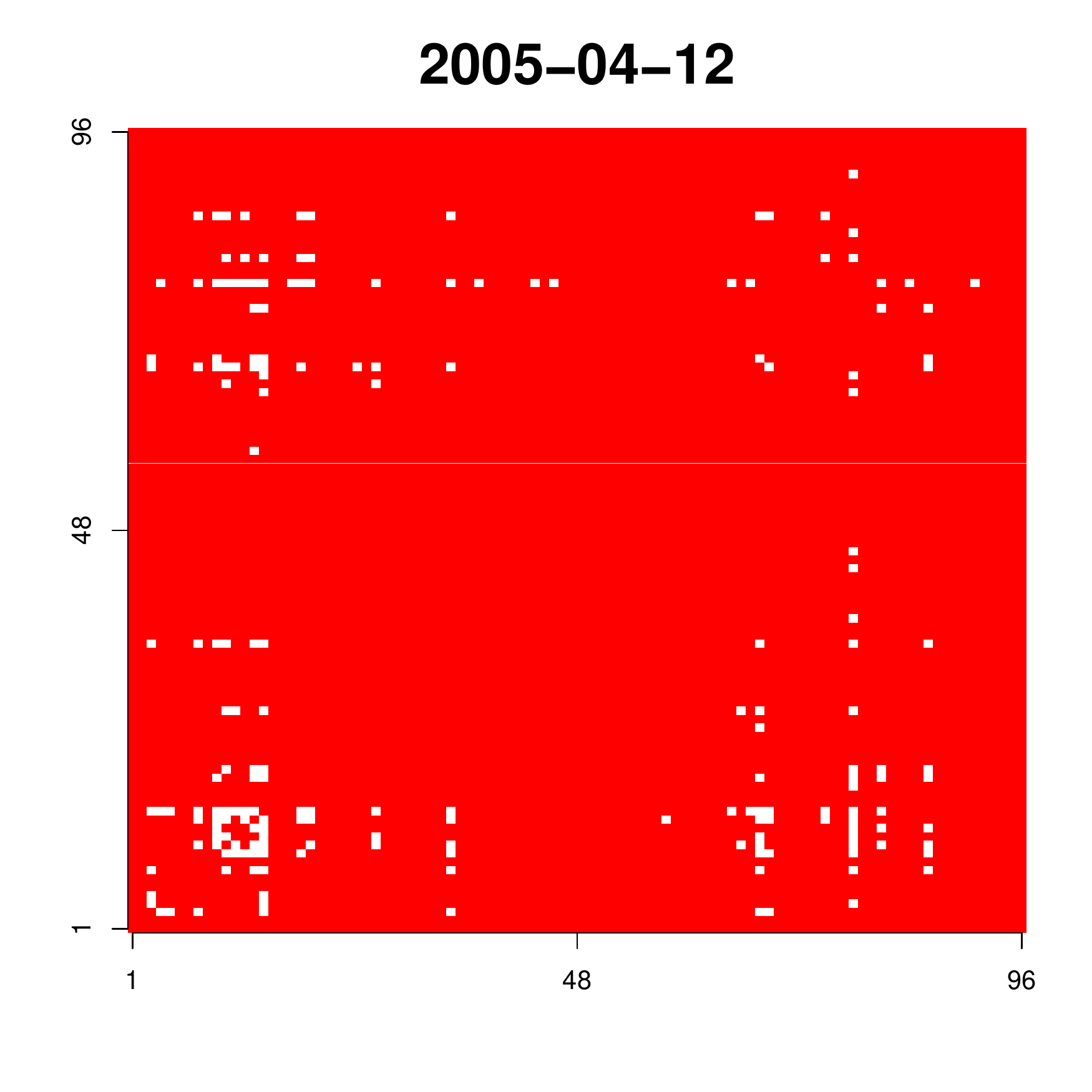} 
		\includegraphics[width=0.24\textwidth]{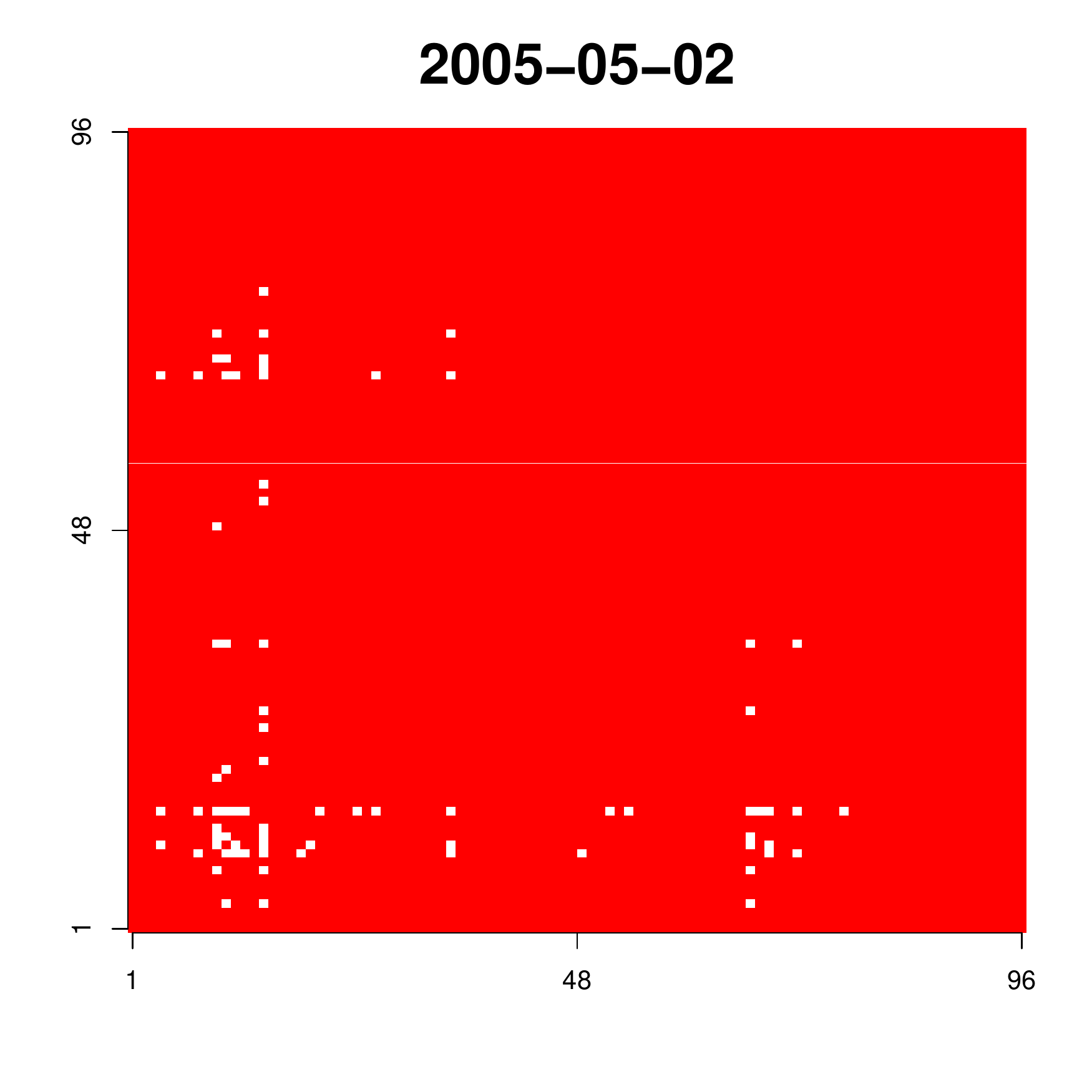} 
		\caption{\label{fig-sec-intro}  Interaction networks based on the MIT cellphone data sets.  A white dot means the corresponding row and column individuals interacted on the specific date.  A red dot means the lack of interaction.  Details are explained in \Cref{sec:cellphone_data}.}
	\end{center}
\end{figure}

Due to the aforementioned real-world applications, change point detection problems have been intensively studied in the literature, not necessarily in the dynamic networks context though.  In terms of online change point detection, i.e.,  making sequential decisions about the existence of change points while collecting data, \cite{lorden1971procedures}, \cite{moustakides1986optimal}, \cite{ritov1990decision}, \cite{lai1981asymptotic}, \cite{lai1998information}, \cite{lai2001sequential}, \cite{chu1996monitoring}, \cite{aue2004delay}, \cite{kirch2008bootstrapping}, \cite{madrid2019sequential} and \cite{yu2020note} studied univariate sequences; \cite{he2018sequential} focused on a sequence of random graphs; \cite{chen2019sequential} and \cite{dette2019likelihood} allowed for more general scenarios, including nonparametric models; \cite{chen2020high} and \cite{keshavarz2018sequential} studied high-dimensional Gaussian vectors.  In terms of offline change point detection, i.e., after collecting a sequence of data, one seeks change points retrospectively, a wide range of models have been studied.  The closely related one is \cite{wang2018optimal}, where a sequence of adjacency matrices were considered.  More discussions with existing literature will be provided as we unfold our results.

\subsection{List of contributions}

The contributions of this paper are summarised below.  
\begin{itemize}
\item To the best of our knowledge, it is the first time that online change point detection is formally analysed in a sequence of adjacency matrices, allowing all model parameters to vary as functions of $\Delta$.

\item We establish minimax lower bounds on the detection delay.  To the best of our knowledge, in statistical networks literature, a lower bound involving the rank parameter has only been established in estimation problems \citep[e.g.][]{gao2015rate}, but not in the context of testing, not to mention change point detection. Our  lower bound matches, up to a logarithmic factor, an upper bound derived based on an NP-hard algorithm.

\item In addition, we propose a computationally-efficient network online change point detection method, which comes with two variants corresponding to two different Type-I error controlling strategies.  Extensive numerical results are provided to evaluate the performance of our proposed methods against state-of-the-art competitors.  We also  discuss tuning parameter selection 
 aspects of our approaches.
\end{itemize}

Throughout this paper, we will adopt the following notation.  For any matrix $M \in \mathbb{R}^{m_1 \times m_2}$, let $\|M\|_{\mathrm{F}}$ and $\|M\|_{\infty} = \max_{i = 1}^{m_1}\max_{j = 1}^{m_2}|M_{ij}|$ be the Frobenius norm and the entry-wise supremum norm of $M$, respectively.  For any two matrices $A, B \in \mathbb{R}^{m_1 \times m_2}$, let $(A, B) = \mathrm{tr}(A^{\top}B)$ be the Frobenius inner product of two matrices.

\section{Methods}\label{sec-methods}

Since our data are a sequence of adjacency matrices, our first task is to formally define the networks at every time point.  

\begin{definition}[Inhomogeneous Bernoulli networks]\label{def-inhomo-ber-net}
A network with node set $\{1, \ldots, n\}$  is an inhomogeneous Bernoulli network if its adjacency matrix $A \in \mathbb{R}^{n\times n}$ satisfies
		\[
		A_{ij} = A_{ji} = \begin{cases}
			1, & \mbox{nodes $i$ and $j$ are connected by an edge},\\
			0, & \mbox{otherwise};
		\end{cases}
		\]
		and $\{A_{ij}, i < j\}$ are independent Bernoulli random variables with $\mathbb{E}(A_{ij}) = \Theta_{ij}$.  We refer to the matrix $\Theta$ as the graphon matrix. 
\end{definition}

This general definition includes popular network models as special cases, but it is not restricted to any specific model.  Note that we allow every random variable that corresponds to an integer pair $(i, j)$, $1 \leq i < j \leq n$, to have its own mean, i.e.~$\mathbb{E}(A_{ij}) = \Theta_{ij}$.  

\begin{remark}
Despite the flexibility it enjoys, \Cref{def-inhomo-ber-net} is also subjected to a number of restrictions.  First, each random variable is assumed to be Bernoulli, which is a sub-Gaussian random variable.  However, our framework can be extended to handle Poisson random variables.  The algorithms we propose in the sequel follow naturally, and the theoretical results can be adjusted by using sub-Exponential concentration inequalities.  Second, the adjacency matrices are assumed to be symmetric in \Cref{def-inhomo-ber-net}.  All the results in this paper can be extended straightforwardly to asymmetric cases, corresponding to directed networks.  Finally, the model does not allow for dependence between entries of the adjacency matrix.
\end{remark}

In order to detect the change points of $\{\Theta(t), \, t = 1, 2, \ldots\}$, we adopt the network CUSUM statistic, which originated in the univariate CUSUM statistics from \cite{Page1954} and was first formally stated in \cite{wang2018optimal}.

\begin{definition}\label{def-cusum-net}
Given a sequence of matrices $\{A(t)\}_{t = 1, 2, \ldots} \subset \mathbb{R}^{n \times n}$, we define the corresponding online CUSUM statistics as
	\[
		\widehat{A}_{s, t} = \sqrt{\frac{t-s}{st}} \sum_{l = 1}^s A(l) - \sqrt{\frac{s}{(t-s)t}}\sum_{l = s+1}^t A(l),
	\]
	for all integer pairs $(s, t)$, $t \geq 2$ and $s \in [1, t)$.  
\end{definition}

With the notation introduced in \Cref{def-cusum-net}, we have for any integer pair $(s, t)$, $1 \leq s < t$,
	\[
		\mathbb{E}(\widehat{A}_{s, t}) = \widehat{\Theta}_{s, t} = \sqrt{\frac{t-s}{st}} \sum_{l = 1}^s \Theta(l) - \sqrt{\frac{s}{(t-s)t}}\sum_{l = s+1}^t \Theta(l).
	\]

\Cref{alg-online-net} is our main procedure, with a subroutine detailed in \Cref{alg-usvt} and a variant in \Cref{alg-online-net-variant}.  Both Algorithms~\ref{alg-online-net} and \ref{alg-online-net-variant} are written in a way that they will not stop if no change point is detected.  In practice, they can be terminated either by users or if there are no more new data.

To motivate our algorithms, we first investigate the statistics in  \Cref{def-cusum-net}.  These are linear combinations of all adjacency matrices up to time point $t \geq 2$.  To be specific,  for $1 \leq s < t$, the corresponding statistic
is a difference between the sample means before and after time point $s$.  In the univariate online change point detection problem \citep[e.g.][]{yu2020note}, one scans through all possible integer pairs $(s, t)$, $1 \leq s < t$.  In \Cref{alg-online-net}, we propose a more efficient algorithm to avoid scanning through all $s \in [1, t)$.

For every time point $t \geq 2$, in \Cref{alg-online-net}, we only consider $s \in \mathcal{S}(t)$ for candidates of change points, where $\mathcal{S}(t) = \{t - 2^{j}, j = 0, 1, \ldots, \lfloor \log(t)/\log(2)\rfloor - 1\}$ is a set of geometric scale grid points.  Once the criteria

	\begin{equation}\label{eq-criterion}
		\|\widetilde{B}_{s, t}\|_{\mathrm{F}} > C\log^{1/2}(t/\alpha) \quad \mbox{and} \quad (\widehat{A}_{s, t}, \widetilde{B}_{s, t}/\|\widetilde{B}_{s, t}\|_{\mathrm{F}}) > b_t
	\end{equation}
	are met, we declare that there exists a change point at $t$.  

The criteria in \eqref{eq-criterion} are constructed based on two independent samples $\{A(t)\}$ and $\{B(t)\}$.  In practice, one can achieve this by splitting data into odd and even indices subsamples.  For each integer pair $(s, t)$, the quantity $\widetilde{B}_{s, t}$ is a function of the CUSUM statistic $\widehat{B}_{s, t}$, obtained by the subroutine \Cref{alg-usvt}. In fact, $\widetilde{B}_{s, t}$   is a universal singular value thresholding (USVT) estimator.  The USVT algorithm was proposed in \cite{Chatterjee2015}, for the purpose of estimating low-rank high-dimensional sparse matrices.  

The criteria \eqref{eq-criterion} have two components.  We first need to check that the USVT estimator $\|\widetilde{B}_{s, t}\|_{\mathrm{F}}$ is large enough.  In theory, this is required to prompt near optimal detection delay.  Intuitively speaking, change points would only occur, if $\|\widetilde{B}_{s, t}\|_{\mathrm{F}}$ is large.  Provided that this criterion holds, we then check that the matrix inner product $(\widehat{A}_{s, t}, \widetilde{B}_{s, t}/\|\widetilde{B}_{s, t}\|_{\mathrm{F}})$ is large enough.  The data splitting is summoned due to the fact that for any Bernoulli random variable $X$, $X^2 = X$.  Therefore, in order to detect the change in terms of the Frobenius norm, data splitting helps to estimate the squared means of Bernoulli random variables.  

 In view of the whole procedure, there is a sequence of tuning parameters.  The tuning parameters 
\[
\{\tau_{j, s, u}, \, j = 1, 2, \, u = 2, 3, \ldots, \, s \in \mathcal{S}(u)\}
\]
are used in the subroutine \Cref{alg-usvt}.  As suggested in \cite{xu2018rates}, the parameters $\tau_{1, \cdot, \cdot}$ serve as cutoffs of the upper bound on sample fluctuations; and the parameters $\tau_{2, \cdot, \cdot}$ are chosen to be the entry-wise maximum norms of the matrices of interest.  The tuning parameter $\alpha \in (0, 1)$ is the tolerance of Type-I errors and acts as an upper bound on the probability of returning at least one false alarm.  The  thresholds $\{b_t\}$ are upper bounds on the inner products when there is no change point.  More detailed discussions and guidance on tuning parameter selection are provided in Sections~\ref{sec-theory} and \ref{sec-numeric}.

\begin{algorithm}[!ht]
	\begin{algorithmic}
		\INPUT Symmetric matrix $A \in \mathbb{R}^{n \times n}$, $\tau_1, \tau_2 > 0$.
		\State $(\lambda_i, v_i) \leftarrow $ the $i$th eigen-pair of $A$, with $|\lambda_1| \geq \cdots \geq |\lambda_n|$;
		\State $A' \leftarrow \sum_{i:\, |\lambda_i| \geq \tau_1} \lambda_i v_i v_i^{\top}$;
		\State $A'' \leftarrow $ a matrix with $(i, j)$th entry $(A'')_{ij}$ satisfying
			\[
				(A'')_{ij} = \begin{cases}
 						(A')_{ij}, & |(A')_{ij}| \leq \tau_2,\\
 						\mathrm{sign}((A')_{ij}) \tau_2, & |(A')_{ij}| > \tau_2.
 					\end{cases}
			\]
		\OUTPUT $A''$.
		\caption{Universal Singular Value Thresholding.  $\mathrm{USVT}(A, \tau_1, \tau_2)$} \label{alg-usvt}
	\end{algorithmic}
\end{algorithm} 

\begin{algorithm}[!ht]
	\begin{algorithmic}
		\INPUT $\{A(u), B(u)\}_{u=1, 2, \ldots} \subset \mathbb{R}^{n \times n}$, $\{b_u, \tau_{1, s, u}, \tau_{2, s, u}, u = 2, 3, \ldots, s = 1, 2, \ldots, u\} \subset \mathbb{R}, \alpha \in (0, 1)$.
		\State $t \leftarrow 1$;
		\State $\mathrm{FLAG} \leftarrow 0$;
		\While{$\mathrm{FLAG} = 0$}  
			\State{$t \leftarrow t + 1$;}
			\State{$J \leftarrow \lfloor \log(t)/\log(2)\rfloor$;}
			\State{$j \leftarrow 0$;}
			\While{$j < J$ and $\mathrm{FLAG} = 0$}
				\State{$s_j \leftarrow t-2^j$;}
				\State{$\widetilde{B}_{s_j, t} \leftarrow \mathrm{USVT}(\widehat{B}_{s_j, t}, \tau_{1, s_j, t}, \tau_{2, s_j, t})$;}
				\State{$\mathrm{FLAG} = \mathbbm{1} \left\{(\widehat{A}_{t-s_j, t}, \widetilde{B}_{t-s_j, t}/\|\widetilde{B}_{t-s_j, t}\|_{\mathrm{F}}) > b_t\right\}\mathbbm{1} \left\{\|\widetilde{B}_{t-s_j, t}\|_{\mathrm{F}} > C\log^{1/2}(t/\alpha)\right\}$;}			
				\State{$j \leftarrow j + 1$;}
			\EndWhile
		\EndWhile
		\OUTPUT $t$.
		\caption{Network online change point detection} \label{alg-online-net}
	\end{algorithmic}
\end{algorithm} 

In addition, we also present a variant of \Cref{alg-online-net} in \Cref{alg-online-net-variant}.  Note that the main difference between these two algorithms is that the tuning parameter $\alpha \in (0, 1)$ is replaced by $\gamma \in \mathbb{N}$ in \Cref{alg-online-net-variant}.  Inputs are changed correspondingly.  These two algorithms represent two popular ways of controlling Type-I errors.  The tuning parameter $\gamma$ is in fact a lower bound on the average run length.  In other words, a  choice of $\gamma$ implies that, when there is no change point, the expected time of the first false alarm is at least $\gamma$.

There is no algorithmic differences between Algorithms~\ref{alg-online-net} and \ref{alg-online-net-variant}.  Their theoretical differences will be explained in \Cref{sec-theory}, and the tuning parameter selection differences will be discussed in \Cref{sec-numeric}.

\begin{algorithm}[!ht]
	\begin{algorithmic}
		\INPUT $\{A(u), B(u)\}_{u=1, 2, \ldots} \subset \mathbb{R}^{n \times n}$, $\{b_{u}, \tau_{1, s, u}, \tau_{2, s, u}, u = 2, 3, \ldots, s = 1, 2, \ldots, u\} \subset \mathbb{R}, \gamma \in \mathbb{N}$.
		\State $t \leftarrow 1$;
		\State $\mathrm{FLAG} \leftarrow 0$;
		\While{$\mathrm{FLAG} = 0$}  
			\State{$t \leftarrow t + 1$;}
			\State{$J \leftarrow \lfloor \log(t)/\log(2)\rfloor$;}
			\State{$j \leftarrow 0$;}
			\While{$j < J$ and $\mathrm{FLAG} = 0$}
				\State{$s_j \leftarrow t-2^{j-1}$;}
				\State{$\widetilde{B}_{s_j, t} \leftarrow \mathrm{USVT}(\widehat{B}_{s_j, t}, \tau_{1, s_j, t}, \tau_{2, s_j, t})$;}
				\State{$\mathrm{FLAG} = \mathbbm{1} \left\{(\widehat{A}_{s_j, t}, \widetilde{B}_{s_j, t}/\|\widetilde{B}_{s_j, t}\|_{\mathrm{F}}) > b_t\right\}\mathbbm{1} \left\{\|\widetilde{B}_{s_j, t}\|_{\mathrm{F}} > C\log^{1/2}(\gamma)\right\}$;}
				\State{$j \leftarrow j + 1$;}
			\EndWhile
		\EndWhile
		\OUTPUT $t$.
		\caption{Network online change point detection -- a variant} \label{alg-online-net-variant}
	\end{algorithmic}
\end{algorithm}

\section{Theory}\label{sec-theory}

This section consists of all the theoretical results we develop in this paper, with all the technical details in the Appendix.  This section is organised as follows.  All the assumptions are stated and discussed in \Cref{sec-assumptions}.  The theoretical guarantees of Algorithms~\ref{alg-online-net} and \ref{alg-online-net-variant} are provided in \Cref{sec-main-results}.  To investigate the fundamental limits, we established a minimax lower bound on the detection delay in \Cref{sec-lb}, with an NP-hard procedure which is nearly minimax optimal studied in \Cref{sec-np-hard}.  To conclude, we provide some additional discussions through comparisons with existing work in \Cref{sec-np-hard}.

\subsection{Assumptions}\label{sec-assumptions}

Before arriving at our main results, we start by introducing some assumptions.  The following three assumptions introduce the sparsity parameter, describe the one change point and no change point scenarios, respectively.

\begin{assumption}\label{assump-1-net}
Assume that $\{A(1), A(2), \ldots\} \subset \mathbb{R}^{n \times n}$ is a sequence of inhomogeneous Bernoulli networks satisfying $\mathbb{E}\{A(i)\} = \Theta(i) \in \mathbb{R}^{n \times n}$, $i = 1, 2, \ldots$, and 
	\[
		\sup_{i = 1, 2, \ldots}\|\Theta(i)\|_{\infty} = \rho,
	\]
	where $\rho n \geq \log(n)$.
\end{assumption}

\begin{assumption}[One change point scenario] \label{assump-2-net}
Assume that there exists $\Delta \in \mathbb{N}^*$ such that
	\[
		\Theta(1) = \cdots = \Theta(\Delta) = \Theta_1 \quad \mbox{and} \quad \Theta(\Delta + 1) = \Theta(\Delta + 2) = \cdots = \Theta_2.
	\]	
	In addition, let 
	\[
		\kappa_0 = \frac{\kappa}{n\rho} = \frac{\|\Theta_1 - \Theta_2\|_{\mathrm{F}}}{n\rho} > 0 \quad \mbox{and} \quad r = \mathrm{rank}(\Theta_1 - \Theta_2).
	\]
\end{assumption}

\begin{assumption}[No change point scenario]\label{assump-2-0-net}
Assume that 
	\[
		\Theta(1) = \Theta(2) = \cdots = \Theta.
	\]	
\end{assumption}

In the one change point scenario, in view of Assumptions~\ref{assump-1-net} and \ref{assump-2-net}, we see that the change point detection problem is characterised by the following parameters: the network size $n$, the entry-wise sparsity parameter $\rho$, the size of the uncontaminated sample $\Delta$, the normalised jump size $\kappa_0$, and the low-rank parameter~$r$.

The jump size $\kappa$ is defined to be the Frobenius norm of the difference between two consecutive but distinct graphon matrices.  The choice of the Frobenius norm is tailored to the context of dynamic networks.  Arguably, the most popular statistical network model is the stochastic block models \citep{HollandEtal1983}.  If both $\Theta_1$ and $\Theta_2$ are graphon matrices of two stochastic block models, then the Frobenius norm can explicitly reflect the magnitude of the change, as compared to other matrix norms including the operator norm and the supremum norm.  For instance, if  the community structure stays unchanged, but the between community probability changes from $p_1$ to $p_2$ in a community of size $n/2$, then $\kappa = n|p_1 - p_2|/2$.  If the between and within community probabilities, $p_1$ and $p_2$, remain the same, but the community structure changes from a balanced 2-community network to a balanced 3-community network, then $\kappa = \sqrt{13/18}n|p_1 - p_2|$.  Since $\kappa \in (0, n\rho]$, the normalised jump size $\kappa_0$ is scale free and satisfies that $\kappa_0 \in (0, 1]$.

Without further restrictions, the low-rank parameter $r$ is allowed to be $r \in \{1, \ldots, n\}$.  Note that, the introduction of the parameter $r$ is on the difference matrix and we allow for arbitrary structure of each graphon \emph{per se}.

\subsection{Main results}\label{sec-main-results}

Recall that our missions are as follows.  When there is a change point, we wish to declare the existence of the change point as soon as it appears.  The distance between the change point estimator and the change point is called detection delay, which is to be minimised.  On the other hand, it is also vital to control the false alarms.  When there is no change point, we either control the probability of declaring change points, or the expected time of the first false alarm.  These two different ways of controlling  false alarms are in fact Algorithms~\ref{alg-online-net} and \ref{alg-online-net-variant}.  Their theoretical guarantees are provided in \Cref{thm-net-delay-varying-N} and \Cref{thm-selection}, respectively.  In the results presented in this section, we assume the existence of two independent sequences of adjacency matrices.  In practice, this can be done by splitting the data sequence into odd and even index sequences.  

\begin{theorem}\label{thm-net-delay-varying-N}
For any $\alpha \in (0, 1)$, assume that the data $\{A(t), \, B(t)\}_{t = 1, 2, \ldots}$ are two independent sequences of adjacency matrices satisfying \Cref{assump-1-net}.  Let $\widehat{t}$ be the output of \Cref{alg-online-net} with inputs $\{A(t), \, B(t)\}_{t = 1, 2, \ldots}$, 
	\begin{align*}
		& b_u = C_1\sqrt{\rho\log\left(\frac{u}{\alpha}\right)},  \,\, \tau_{1, s, u} = C\sqrt{n\rho} + \sqrt{2\log\left\{ \frac{u(u+1) \log(u)}{ \alpha \log(2)}\right\}} \,\, \mbox{and} \,\, \tau_{2, s, u} = \sqrt{\frac{(u-s)s}{u}} \rho.
	\end{align*}
	\begin{itemize}
	\item [(i)] (No change point.) If $\{A(t), \, B(t)\}_{t = 1, 2, \ldots}$ in addition satisfy \Cref{assump-2-0-net}, the it holds that 
		\[
			\mathbb{P}\left\{\underset{m \in \mathbb{N}}{\bigcap }   \{ \widehat{t} > m  \}   \right\} > 1 - \alpha.
		\]	
	\item [(ii)] (One change point.) If $\{A(t), \, B(t)\}_{t = 1, 2, \ldots}$ in addition satisfy \Cref{assump-2-net}, and there exists a large enough absolute constant $C_{\mathrm{SNR}} > 0$ such that
		\begin{equation}\label{eq-snr}
			\Delta \kappa_0^2 n \rho > C_{\mathrm{SNR}} r \log(\Delta/\alpha),
		\end{equation}
		then
		\[
			\mathbb{P}\left\{0 < \widehat{t} - \Delta < \frac{C_drn\rho \log(\Delta/\alpha)}{\kappa^2} = \frac{C_d r\log(\Delta/\alpha)}{\kappa_0^2n\rho}\right\} > 1 - \alpha,
		\]
		where $C, C_1, C_d > 0$ are absolute constants.
	\end{itemize}
\end{theorem}

 We can see from \Cref{thm-net-delay-varying-N}(i) that when there is no change point, with probability at least~$1 - \alpha$, \Cref{alg-online-net} will not raise any false alarm.  On the other hand, if there is a change point, then it follows from \Cref{thm-net-delay-varying-N}(ii) that the detection delay is at most of order
	\[
		\frac{r\log(\Delta/\alpha)}{\kappa_0^2 n \rho},
	\]
	with probability at least $1 - \alpha$.  
	
In fact, the condition \eqref{eq-snr} can be regarded as a sort of signal-to-noise ratio condition and is a mild constraint.  We list a few special cases here.  
	\begin{itemize}
	\item (Small sample size.) If $\kappa_0, \rho, r, \alpha \asymp 1$, then as long as $\Delta \gtrsim \log^2(\Delta)/n$, \eqref{eq-snr} holds.  If the network size $n$ is large, then this shows that \Cref{alg-online-net} can detect change points with a very small number of uncontaminated samples.
	\item (Large rank matrices.) If $\kappa_0 \asymp 1$, $\rho \asymp \log^2(n)/n$, $\alpha \asymp 1$ and $\Delta \asymp n$, then the rank parameter $r$ is allowed to be $r \asymp n$.  This means that provided the size of uncontaminated sample is comparable with the size of networks, then it is not necessary to have a low-rank assumption imposed on the difference of the graphons.
	\item (Small jump size.)  If $\rho, r, \Delta,\alpha \asymp 1$, then, provided that $\kappa_0 \gtrsim n^{-1/2}$,  condition \eqref{eq-snr} holds.  This means that the normalised jump size can decrease to zero  if the size of the network diverges.
	\end{itemize}

Finally, we remark on the choices of the tuning parameters.  As we have mentioned in \Cref{sec-methods},  the tuning parameters $\{\tau_{1, s, u}\}$ are the cutoffs due to the low-rank parameter and $\{\tau_{2, s, u}\}$ are to bound the entry-wise maximum norms.  The theoretical choices of these two sets of parameters can be found in \Cref{lem-usvt-error-control}, with the aim of ensuring the good performances of the USVT estimators.  The tuning parameter $\alpha$ is completely determined by practitioners, reflecting the tolerance of Type-I errors.  The sequence $\{b_u\}$ reflects an upper bound on the statistics' fluctuations when there is no change point. The rate of  $\{b_u\}$  is determined in \Cref{lem-s2}.

\begin{corollary} \label{thm-selection}
For $\gamma \geq 2$, assume that the data $\{A(t), \, B(t)\}_{t = 1, 2, \ldots}$ are two independent sequences of adjacency matrices satisfying \Cref{assump-1-net}.  Let $\widehat{t}$ be the output of \Cref{alg-online-net-variant} with inputs $\{A(t), \, B(t)\}_{t = 1, 2, \ldots}$, 
	\begin{align*}
		& b_u = C_1\sqrt{\rho\log\left(\gamma\right)}, \,\, \tau_{1, s, u} = C\sqrt{n\rho} + \sqrt{2\log\left\{\frac{2(\gamma+1)\gamma \log(\gamma+1)}{\log(2)}\right\}}  \,\, \mbox{and}  \,\, \tau_{2, s, u} = \sqrt{\frac{(u-s)s}{u}} \rho.
	\end{align*}
	\begin{itemize}
	\item [(i)] (No change point.) If $\{A(t), \, B(t)\}_{t = 1, 2, \ldots}$ in addition satisfy \Cref{assump-2-0-net}, then
		\[
			\mathbb{E}(\widehat{t}) \geq \gamma,
		\]
		where $\mathbb{E}(\widehat{t})$, under \Cref{assump-2-0-net}, is called the average run length.
	\item [(ii)] (One change point.) If $\{A(t), \, B(t)\}_{t = 1, 2, \ldots}$ in addition satisfy \Cref{assump-2-net}, and it holds that
		\begin{equation}\label{eq-m-gamma-cond}
			\gamma \geq \Delta \quad \mbox{and} \quad  \Delta \kappa_0^2 n \rho > C_{\mathrm{SNR}} r \log(\gamma),
		\end{equation}
		where $C_{\mathrm{SNR}} > 0$ is an absolute constant, then 
		\begin{equation}\label{eq-detect-delay-cor}
			\mathbb{P}\left\{0 < \widehat{t} - \Delta < \frac{C_drn\rho \log(\gamma)}{\kappa^2} = \frac{C_d r\log(\gamma)}{\kappa_0^2n\rho}\right\} > 1 - \gamma^{-1},
		\end{equation}
		where $C, C_1, C_d > 0$ are absolute constants.
	\end{itemize}	
\end{corollary}

\Cref{thm-selection} is \Cref{thm-net-delay-varying-N}'s counterpart based on \Cref{alg-online-net-variant}.  Comparisons of \Cref{thm-net-delay-varying-N}(i) and \Cref{thm-selection}(i) show that Algorithms~\ref{alg-online-net} and \ref{alg-online-net-variant} have different strategies in controlling the false alarms.  \Cref{thm-net-delay-varying-N}(i) shows that the Type-I error across the whole time horizon is upper bounded by $\alpha$ if \Cref{alg-online-net} is deployed. In contrast,  \Cref{thm-selection}(i) ensures that if \Cref{alg-online-net-variant} is used then the expected time of the first false alarm is at least $\gamma$.

Both of these two ways to control the false alarms are widely used in the literature.    We show in \Cref{thm-net-delay-varying-N} and \Cref{thm-selection} that, if $\gamma \geq \Delta$ and 
	\begin{equation}\label{eq-gamma-alpha}
		\gamma \asymp \Delta/\alpha,
	\end{equation}
	then these two methods provide the same order of detection delay.    If $\gamma < \Delta$, then the same  localisation  rate of \Cref{thm-selection} holds for  $\max\{\widehat{t} -\Delta,0\}$ instead of  $\widehat{t}-\Delta$.

Finally, we summarise the differences between Algorithms~\ref{alg-online-net} and \ref{alg-online-net-variant}.  Since both $\alpha$ and $\gamma$ reflect the preferences on Type-I error  control, these two tuning parameters can be specified by the users.  Although we have specified theoretical guidance on all the other tuning parameters, in practice, they still involve either unknown quantities or unspecified constants.  In order to tune these parameters, \Cref{alg-online-net-variant} might be handier.  One may have access to historical data under the pre-change-point distribution, and tune all the tuning parameters such that the average run length is $\gamma$.  This is less natural under the strategy of \Cref{alg-online-net}, unless the whole time course has a pre-specified endpoint, since the Type-I error is across the whole time course.   This is in fact how we tune the tuning parameters in \Cref{sec-numeric} for \Cref{alg-online-net}.

\subsection{A lower bound}\label{sec-lb}

In \Cref{thm-net-delay-varying-N}, we show that we are able to detect change points with the order of the detection delay upper bounded by 
	\begin{equation} \label{eqn:delay}
		\frac{r\log(\Delta/\alpha)}{\kappa_0^2 n \rho}.
	\end{equation}
	In this subsection, we will investigate the optimality of this upper bound.

\begin{proposition} \label{thm-lb}
Assume that $\{A(t)\}_{t = 1, 2, \ldots}$ is a sequence of independent adjacency matrices satisfying Assumptions~\ref{assump-1-net} and \ref{assump-2-net}. 
  Denote the joint distribution of $\{A(t)\}_{t = 1, 2, \ldots}$ as $P_{\kappa, \Delta}$.  Consider the class of estimators $\mathcal{D}$ defined as
	\[
		\mathcal{D} = \left\{T: \, T \mbox{ is a stopping time and satisfies } \mathbb{P}_{\infty} (T < \infty) \leq \alpha\right\},
	\]
	where $\mathbb{P}_{\infty}$ indicates $\Delta = \infty$.  Then for sufficiently small $\alpha \in (0, 1)$, there exists an absolute constant $c > 0$ such that 
	we have that
	\[
		\inf_{\widehat{t} \in \mathcal{D}}\sup_{P_{\kappa, \Delta}} \mathbb{E}_{P}\left\{(\widehat{t} - \Delta)_+\right\} \geq \frac{c\log(1/\alpha)}{\kappa^2_0 n \rho} \max\left\{1, \, r^2/n\right\}.
	\]
\end{proposition}

The change point estimators are all stopping time random variables satisfying that the overall Type-I error is controlled by $\alpha \in (0, 1)$.  The rate of the detection delay is lower bounded by 
	\[
		\frac{\log(1/\alpha)}{\kappa_0^2 n \rho}\max\left\{1, \, r^2/n\right\}.
	\]
	This means in the low-rank regime $r \lesssim \sqrt{n}$, we have the lower bound $\log(1/\alpha)\left(\kappa_0^2 n\rho\right)^{-1}$; in the large-rank regime $r \gtrsim \sqrt{n}$, the lower bound is of the order $\log(1/\alpha)r^2\left(\kappa_0^2 n^2\rho\right)^{-1}$.  In view of \Cref{thm-lb}, we see that \eqref{eqn:delay} is nearly-optimal, saving for a logarithmic factor, only in the extreme regimes, i.e.~$r \asymp n$ or $r \asymp 1$.  
	
It is then interesting to investigate the gap between the lower and upper bounds.  Recall that in the graphon estimation problems, \cite{gao2015rate} has shown that, the minimax rate of the mean squared error for estimating a rank-$r$  graphon $\Theta  \in \mathbb{R}^{n \times n}$ is 
	\[
		\inf_{\widehat{\Theta}} \sup_{\Theta} \mathbb{E}\left\{\frac{1}{n^2} \left\|\widehat{\Theta} - \Theta\right\|_{\mathrm{F}}^2 \right\} \asymp \frac{r^2 + n\log(r)}{n^2},
	\]
	where the upper bound is achieved by an NP-hard algorithm.  In fact, we can also adopt NP-hard procedures to match the lower bound in \Cref{thm-lb}, up to logarithmic factors.  

\subsection{An NP-hard procedure on stochastic block models}\label{sec-np-hard}

In this subsection, we first focus on the stochastic block models \citep{HollandEtal1983}.

\begin{definition}[Sparse Stochastic Block Model]\label{def-ssbm}
	A network is from a sparse stochastic block model with size $n$, sparsity parameter $\rho$, membership matrix $Z \in \mathbb{R}^{n \times r}$ and connectivity matrix $Q \in [0, 1]^{r \times r}$ if the corresponding adjacency matrix  satisfies 
		\[
			\mathbb{E} (A)= \rho ZQZ^{\top}- \mathrm{diag}\bigl(\rho ZQZ^{\top}\bigr) .
		\]
 		
 		The membership matrix $Z$ consists of $n$ rows, each of which has one and only one entry being~1 and has all the entries being 0; moreover, $Z$ is a column full rank matrix, i.e. $\mathrm{rank}(Z) =r $.  The sparsity parameter $\rho \in [0, 1]$ potentially depends on $n$.
\end{definition}

As we have already pointed out, the stochastic block models are special cases of the inhomogeneous Bernoulli networks defined in \Cref{def-inhomo-ber-net}. 

In \cite{gao2015rate}, an NP-hard estimator of stochastic block models' graphons is proposed.  In this subsection, we will replace the USVT estimator defined in \Cref{alg-usvt} and used in \Cref{alg-online-net} by the NP-hard estimator studied in \cite{gao2015rate}.  For completeness, we include the estimator construction below.

\begin{definition}[An NP-hard graphon estimator]\label{def-np-graphon}
For any positive integers $n$ and $r$, $r \leq n$, let $\mathcal{Z}_{n, r} = \{z: \{1, \ldots, n\} \to \{1, \ldots, r\}\}$ be the collection of all possible mappings from $\{1, \ldots, n\}$ to $\{1, \ldots, r\}$.  Given an adjacency matrix $A = (A_{ij}) \in \mathbb{R}^{n \times n}$, any $z \in \mathcal{Z}_{n, r}$ and any $Q = (Q_{ab}) \in \mathbb{R}^{r \times r}$, define the objective function
	\[
		L(Q, z) = \sum_{a, b \in \{1, \ldots, r\}} \sum_{\substack{(i, j) \in z^{-1}(a) \times z^{-1}(b) \\ i \neq j}} (A_{ij} - Q_{ab})^2.
	\]
		For any optimiser of the the objective function 
	\[
		(\widehat{Q}, \hat{z}) \in \argmin_{Q \in \mathbb{R}^{r \times r}, \, z \in \mathcal{Z}_{n, r}} L(Q, z),
	\]
	the estimator is defined as $\widecheck{\Theta} = (\widecheck{\Theta}_{ij})_{i, j = 1}^n \in \mathbb{R}^{n \times n}$, with
	\[
		\widecheck{\Theta}_{ij} = \widecheck{\Theta}_{ji} = \widehat{Q}_{\hat{z}_i \hat{z}_j}, \quad i > j
	\]
	and $\widecheck{\Theta}_{ii} = 0$.  For notational simplicity, we write $\widecheck{\Theta} = \mathrm{NP}(A, r)$.
\end{definition}

The new procedure for change point detection replaces the USVT subroutine in \Cref{alg-online-net} with the estimation detailed in \Cref{def-np-graphon}.  We present the full algorithm below.

\begin{algorithm}[!ht]
	\begin{algorithmic}
		\INPUT $\{A(u), B(u)\}_{u=1, 2, \ldots} \subset \mathbb{R}^{n \times n}$, $\{b_u, u = 2, 3, \ldots\} \subset \mathbb{R}, \alpha \in (0, 1), r_0 \in \mathbb{N}^*$.
		\State $t \leftarrow 1$;
		\State $\mathrm{FLAG} \leftarrow 0$;
		\While{$\mathrm{FLAG} = 0$}  
			\State{$t \leftarrow t + 1$;}
			\State{$J \leftarrow \lfloor \log(t)/\log(2)\rfloor$;}
			\State{$j \leftarrow 0$;}
			\While{$j < J$ and $\mathrm{FLAG} = 0$}
				\State{$s_j \leftarrow t-2^j$;}
				\State{$\widecheck{B}_{s_j, t} \leftarrow \mathrm{NP}(A, r_0)$;}
				\State{$\mathrm{FLAG} = \mathbbm{1} \left\{(\widehat{A}_{t-s_j, t}, \widecheck{B}_{t-s_j, t}/\|\widecheck{B}_{t-s_j, t}\|_{\mathrm{F}}) > b_t\right\}\mathbbm{1} \left\{\|\widecheck{B}_{t-s_j, t}\|_{\mathrm{F}} > C\log^{1/2}(t/\alpha)\right\}$;}			
				\State{$j \leftarrow j + 1$;}
			\EndWhile
		\EndWhile
		\OUTPUT $t$.
		\caption{Network online change point detection - NP-hard} \label{alg-online-net-NP}
	\end{algorithmic}
\end{algorithm} 

\begin{corollary} \label{thm-net-delay-np}
For any $\alpha > 0$, assume that the data $\{A(t), \, B(t)\}_{t = 1, 2, \ldots}$ are two independent sequences of adjacency matrices satisfying \Cref{assump-1-net}.  Let $\widehat{t}$ be the output of \Cref{alg-online-net-NP} with inputs $\{A(t), \, B(t)\}_{t = 1, 2, \ldots}$ and  
	\begin{align*}
		& b_u = C_1\sqrt{\rho\log\left(\frac{u}{\alpha}\right)}.
	\end{align*}
	\begin{itemize}
	\item [(i)] (No change point.) If $\{A(t), \, B(t)\}_{t = 1, 2, \ldots}$in addition satisfy \Cref{assump-2-0-net}, then for any $m \in \mathbb{N}$, 
		\[
			\mathbb{P}\left\{\underset{m \in \mathbb{N}}{\bigcap }   \{ \widehat{t} > m  \}   \right\} > 1 - \alpha.
		\]	
	\item [(ii)] (One change point.) If $\{A(t), \, B(t)\}_{t = 1, 2, \ldots}$ in addition satisfy \Cref{assump-2-net}, the input $r_0 \geq r$, and there exists a large enough absolute constant $C_{\mathrm{SNR}} > 0$ such that
		\[
			\Delta \kappa_0^2 n^2 \rho > C_{\mathrm{SNR}} \{r^2 + n\log(r)\} \log(\Delta/\alpha),
		\]
		then
		\[
			\mathbb{P}\left\{0 < \widehat{t} - \Delta < \frac{C_d \left\{r^2 + n\log(r)\right\}\rho \log(\Delta/\alpha)}{\kappa^2} = \frac{C_d \left\{r^2/n + \log(r)\right\}\log(\Delta/\alpha)}{\kappa_0^2n\rho}\right\} > 1 - \alpha,
		\]
		where $C, C_1, C_d > 0$ are absolute constants.
	\end{itemize}
\end{corollary}

\Cref{thm-net-delay-np} provided us with an upper bound on the detection delay matching the lower bound in \Cref{thm-lb}, saving for logarithmic factors.  However, the detection delay  in \Cref{thm-net-delay-np} is based on an NP-hard procedure in \Cref{alg-online-net-NP}, which has limited practical value.  

Comparing the results in \Cref{thm-net-delay-varying-N} and \Cref{thm-net-delay-np}, we see that in the very extreme regimes, i.e.~$r \asymp 1$ or $r \asymp n$, the detection delays obtained by the proposed  polynomial-time and NP-hard   algorithms achieve the same rates.  Both  estimators are nearly optimal, saving for logarithmic factors.  Between the two extreme cases $r \asymp 1$ and $r \asymp n$, the NP-hard algorithm achieves sharper rates than the  polynomial time algorithm.  This phenomenon is inline with the computational and statistical tradeoffs observed in other high-dimensional statistical problems, e.g.~\cite{zhang2012communication}, \cite{loh2013regularized}, to name but a few.

\subsection{Comparisons with existing work}\label{sec-comp}

With all the theoretical results at hand, we are ready to provide some in-depth comparisons with existing work.  Since we believe that our paper is the first ever providing theoretical results for network (in the sense of random matrices) online change point detection problems, the four papers we select in this subsection are all concerned with different but related problems.

\cite{chen2019sequential} establishes a general framework for online change point detection.  Provided a suitable notion of distance, a $k$-nearest-neighbour-based test statistic is used for testing the existence of the change points in a sequential manner. In  \Cref{sec-numeric}, we consider three different statistics such that  the methods from \cite{chen2019sequential} can be used as our competitors.  As for the theoretical results, \cite{chen2019sequential} focused on the average run length.  In our paper, we provide a range of results including detection delay, average run length and minimax lower bounds.  

In statistics literature, the term ``network'' sometimes refers to the precision matrices in Gaussian graphical models, which are different from what we study in this paper. \cite{keshavarz2018sequential} and \cite{keshavarz2020online} studied online change point detection in Gaussian graphical models.  In addition to the model differences, both \cite{keshavarz2018sequential} and \cite{keshavarz2020online} focused on the limiting distributions of the test statistics under the null and alternative distributions.  We conjecture that a detection delay might be obtainable based on the results thereof, but the results are not explicit yet.  On the other hand, in our paper, especially based on \Cref{thm-net-delay-varying-N}, it is straightforward that the overall probabilities of falsely detecting change points or missing change points are both upper bounded by $\alpha$. 

\cite{wang2018optimal} investigated an offline network change point detection problem, where a sequence of independent adjacency matrices are collected and change point estimators are sought retrospectively.  Despite the difference, there are some interesting comparisons, which to some extent, reflect the connections between online and offline change point detection problems.  
	\begin{itemize}
	\item [(1)]	The detection delay in the online setting can be seen as the counterpart of the localisation error in the offline setting.  The minimax lower bound on the localisation error in \cite{wang2018optimal} is of order $(\kappa_0^2n^2\rho)^{-1}$, while the minimax lower bound on the detection delay in this paper is of order $\log(1/\alpha)(\kappa_0^2 n\rho)^{-1}(r^2/n + 1)$.  The extra $\log(1/\alpha)$ term is rooted in the fact that we need to control the Type-I error in online settings.  The other differences are more interesting -- obviously, the offline rate is better than the online rate.   This is because, the \emph{de facto} smallest sample size for a certain distribution in the offline scenario is $\Delta$, while in the online scenario it is $\min\{\Delta, \, \mbox{detection delay}\}$.
	\item [(2)] In both online and offline settings, we have seen a computational and statistical tradeoff.  Comparing \Cref{thm-net-delay-varying-N} and \Cref{thm-net-delay-np}, we see that NP-hard estimators can detect change points under a weaker condition and provide a smaller detection delay.  In the offline setting, as \cite{wang2018optimal} has conjectured, by replacing the USVT estimator with the NP-hard estimator in \Cref{def-np-graphon}, one can achieve a nearly optimal localisation error under a weaker condition, than the one needed by the USVT estimator.
	\end{itemize}

\section{Numerical experiments}\label{sec-numeric}

\subsection{Simulation studies}
\label{sec:sim}

Recall that in \Cref{sec-methods}, we proposed a network online change point detection method in \Cref{alg-online-net}, with a subroutine in \Cref{alg-usvt} and a variant in \Cref{alg-online-net-variant}.  In this section, we will investigate the numerical performances of our proposed methods.  Since there is no direct competitor available, we will tailor the $k$-nearest neighbours type method proposed in \cite{chen2019sequential}.

In order to make a fair comparison with  \cite{chen2019sequential} we consider its  three different statistics, including: the ``original'' (ORI) which specifies the original edge-count scan statistic, the  weighted edge-count scan statistic (W) and the generalised edge-count scan statistic (G).  The  statistics are computed with internal functions in the R \citep{R} package gStream \citep{gstream}.  

 We consider two different  forms of calibration.  The first is based on the probability of raising a false alarm.  Using 200 Monte Carlo simulations and values of $\alpha \in \{0.01, 0.05\}$, we choose the detection thresholds such that, the probability of raising a false alarm in the interval  $[1, T_{\mathrm{train}}]$ is $\alpha$.  The values of $T_{\mathrm{train}}$ are taken from the set $\{150, 200\}$.  The second one is based on  the average run length $\gamma$.  We consider  values of $\gamma$ in the set $\{150, 200\}$ and  calibrate the  thresholds of the competing methods,  based on 200 Monte Carlo simulations, to have average run length approximately $\gamma$ under the pre-change model.  For the data splitting  required  by Algorithms \ref{alg-online-net} and \ref{alg-online-net-variant}, in all of our experiments, the sequence of adjacency matrices $\{A(u)\}$ consists of the odd indices of the original sequence, and the sequence $\{B(u)\}$ of the even ones. 

To evaluate  the performance of different methods, we proceed as follows.  For each generative model described below, we run $N= 100$ Monte Carlo simulations, where in each trial the data are collected in the interval $[1, T]$, $T= 300$. The change point $\Delta$ occurs at the time point $150$. Each  method provides an estimator $\widehat{t}$, which can be $\infty$ if no change points are detected in $[1, T]$. We define $\widetilde{t} =  \min\{T,\widehat{t}\}$ and   compute the average detection delay
 \[
 \displaystyle \text{Delay}\,=\,   \frac{ \sum_{j=1}^{N}    1_{   \{  \widetilde{t} \geq \Delta    \} }       (   \widetilde{t} -\Delta  )    }{    \sum_{j=1}^{N}    1_{   \{  \widetilde{t} \geq \Delta    \} }   }.
 \]
 We also report the proportion of false alarms
 \[
  \displaystyle \text{PFA}\,=\,   \frac{ \sum_{j=1}^{N}    1_{   \{  \widetilde{t} < \Delta    \} }        }{   N }.
 \]

As for \Cref{alg-online-net}, guided by \Cref{thm-net-delay-varying-N}, we set
\[
\tau_{1, s, u} = 0.2\sqrt{n\hat{\rho}} + \frac{1}{15}\sqrt{2\log\left(\frac{2(u-s)(u-s+1)}{\alpha}\right)} \quad \mbox{and} \quad \tau_{2, s, u} = \sqrt{\frac{(u-s)s}{u}} \hat{\rho},
\]
where $\hat{\rho}$ is an estimator of  $\rho$, calculated as the $0.95$-quantile of the quantities
\begin{equation}\label{eq-rho-sim}
\hat{p}_{i,j} =  \sum_{t=1}^T   A_{ij}(t),\quad i,j \in \{1,\ldots,n\}, \,\,i<j,
\end{equation}
where   the matrices $\{A(t)\}_t^{T}$ are part of the training data.   In addition, we set 
\[
b_u =    C_1\sqrt{\hat{\rho}\log\left(\frac{u}{\alpha}\right)},  
\]
with $C_1$ tuned  to give the desired false alarm rate.

With respect to \Cref{alg-online-net-variant}, guided by \Cref{thm-selection}, we let
	\[
		\tau_{1, s, u} = 0.2 \sqrt{n \hat{\rho}}  +  \frac{\sqrt{2\log(  2\gamma+2    )}}{15} \quad \mbox{and} \quad \tau_{2,s, u} = \sqrt{\frac{(u-s)s}{u}} \hat{\rho},
	\]
	where $\hat{\rho}$ is the $0.95$-quantile of the quantities
 in \eqref{eq-rho-sim}.  In addition, we let
	\[
		b_u = C_1 \sqrt{\hat{\rho}\log( \gamma)}, 
	\]
	with the constant $C_1$  calibrated to such that before the change point the expect time of the first false alarm is $\gamma$.

We consider four different settings.

\paragraph{Scenario 1.} This consists of  a stochastic block model with 3 communities of sizes $\floor{n/3}$, $\floor{n/3}$ and $n- 2\floor{n/3}$.  The network size $n$ takes values in $\{100, 150\}$.  Denoting by  $z_i$  the label of the community associated with node  $i \in \{1, \ldots, n\}$,  the data are generated as 
	\[
		A_{ij} (t) \stackrel{\mbox{ind.}}{\sim} \text{Bernoulli}(\rho B_{z_iz_j}(t)),
	\]
	where $\rho =0.02$ and the matrices $B(t)$ satisfy 
	\[
		B(t) \,=\, \left( \begin{array}{ccc} 
		0.6 & 1.0&  0.6\\
		1.0  &0.6  &0.5\\
		0.6  &0.5  &0.6\\
		\end{array} \right), \quad t \in \{1, \ldots, \Delta\}
	\]
	and
	\[
		B(t) \,=\, \left( \begin{array}{ccc} 
		0.6  &0.5  &0.6\\
		0.5  &0.6&  1.0\\
		0.6 & 1.0  &0.6\\
		\end{array} \right), \quad t \in \{\Delta + 1, \Delta + 2, \ldots, T\}.
	\]

\paragraph{Scenario 2.}  This is also a stochastic block model. We now take the number of communities to be $5$ and the number nodes in each community to be $n/5$ where $n \in \{100, 150\}$.  Again we set $\rho=0.02$ but let 
\[
B(t) \,=\, \left( \begin{array}{ccccc} 
0.9  &0.2  &0.2  &0.2  &0.2\\
0.2  & 0.9 & 0.2 & 0.2 & 0.2\\
0.2&  0.2 & 0.9 & 0.2&  0.2\\
0.2 & 0.2  &0.2  &0.9 & 0.2\\
0.2 & 0.2 & 0.2&  0.2 & 0.9\\
\end{array} \right), \quad t \in \{1, \ldots, \Delta\}
\]
and
\[
B(t) \,=\, \left( \begin{array}{ccccc} 
0.5 & 0.1 & 0.1  &0.1 & 0.1\\
0.1 & 0.5 & 0.1  &0.1 & 0.1\\
0.1 & 0.1  &0.5  &0.1  &0.1\\
0.1  &0.1 & 0.1  &0.5 & 0.1\\
0.1 & 0.1 & 0.1&  0.1 & 0.5\\
\end{array} \right), \quad t \in \{\Delta + 1, \Delta + 2, \ldots, T\}.
\]


\paragraph{Scenario 3.}  We consider a degree corrected block model \citep{karrer2011stochastic} with  3 communities of sizes $\floor{n/3}$, $\floor{n/3}$ and $n- 2\floor{n/3}$, where $n \in \{100, 150\}$.  Let $z_i$ be the community  to which  node $i$ belongs, and  
	define  $v_i =   \sqrt{i/n}$.  The data are then  generated  as
	\[
		A_{ij}(t)  \sim \text{Bernoulli}(  v_i v_j  B_{z_iz_j}(t)  ), 
	\]
where 
\[
B(t) \,=\, \left( \begin{array}{ccc} 
0.9  &0.1  &0.1\\
0.1  &0.9 & 0.1\\
0.1  &0.1&  0.9\\
\end{array} \right), \quad t \in \{1, \ldots, \Delta\}
\]
and
\[
B(t) \,=\, \left( \begin{array}{ccc} 
0.95 &0.15  &0.15\\
0.15 &0.95& 0.15\\
0.15  &0.15&  0.95\\
\end{array} \right), \quad t \in \{\Delta + 1, \Delta + 2, \ldots, T\}.
\]

\begin{figure}[pb!]
	\begin{center}
		\includegraphics[width=0.24\textwidth]{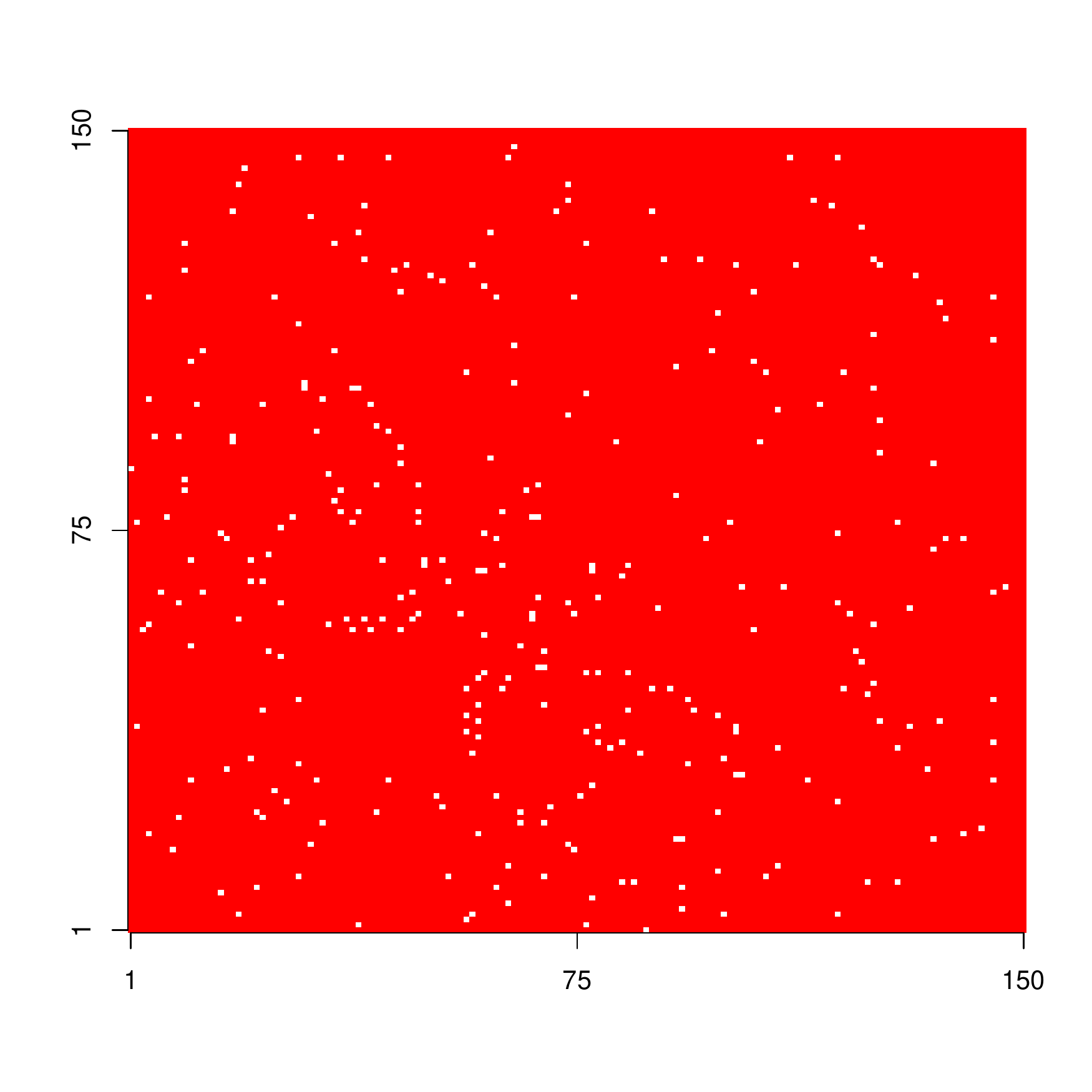} 
		\includegraphics[width=0.24\textwidth]{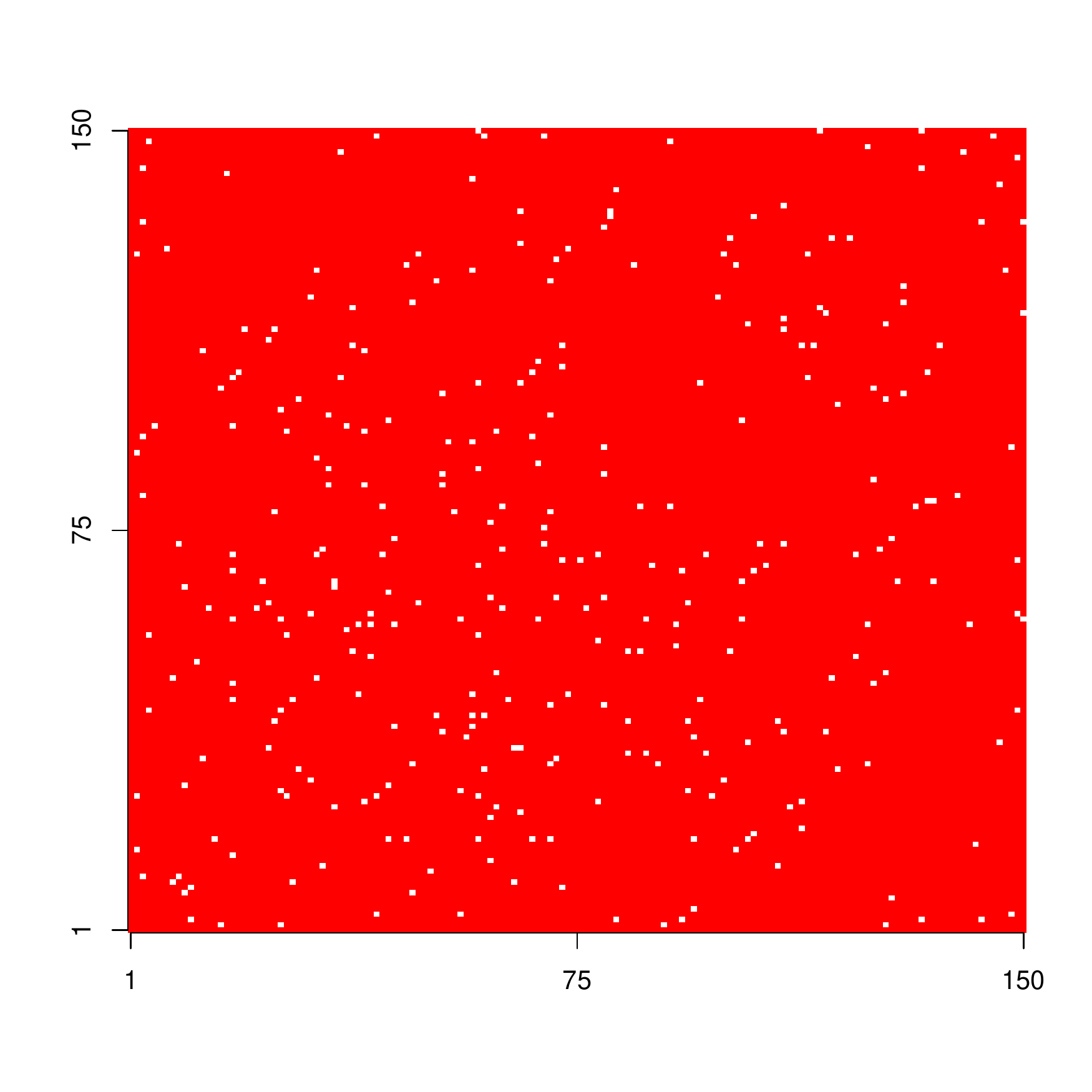}
		\includegraphics[width=0.24\textwidth]{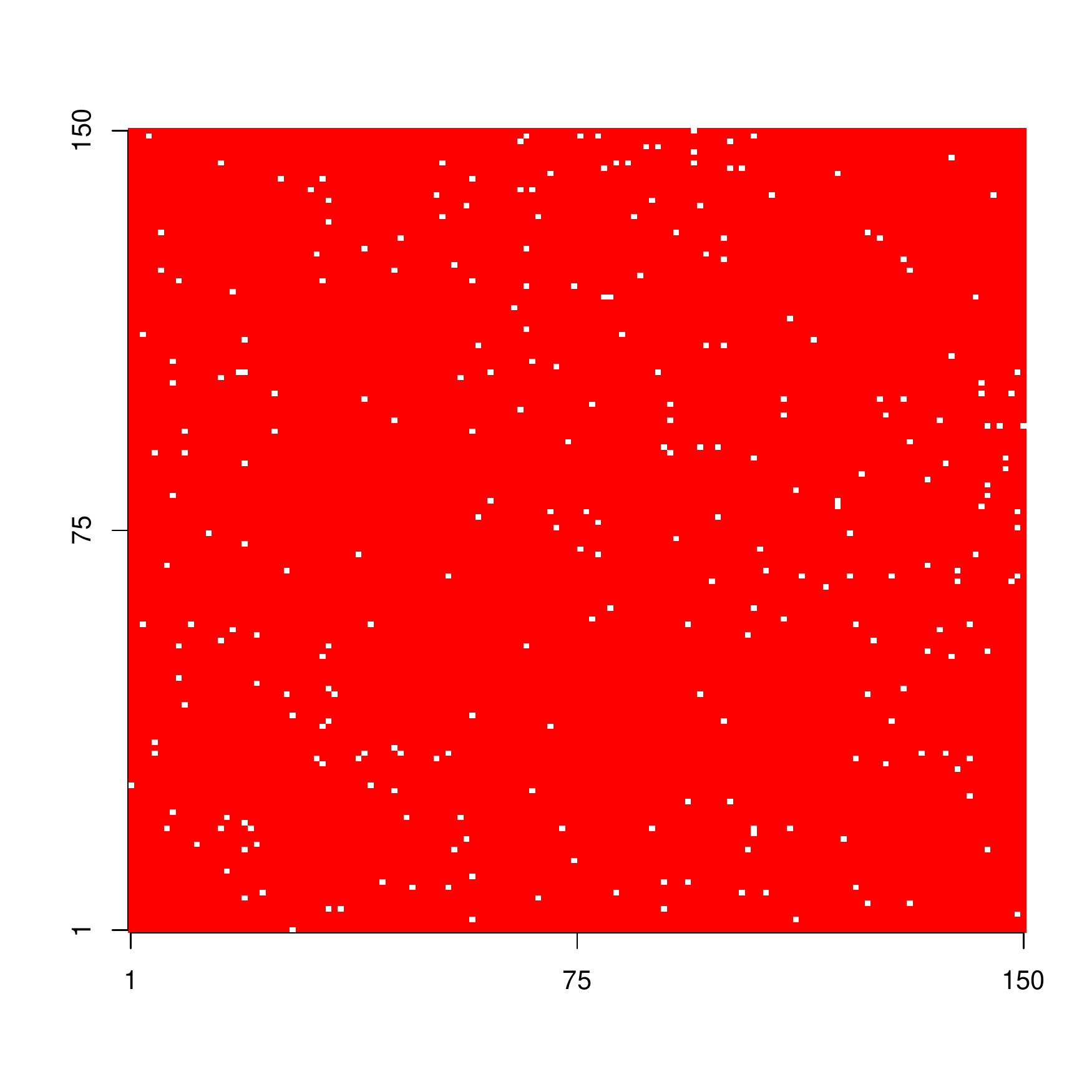}
		\includegraphics[width=0.24\textwidth]{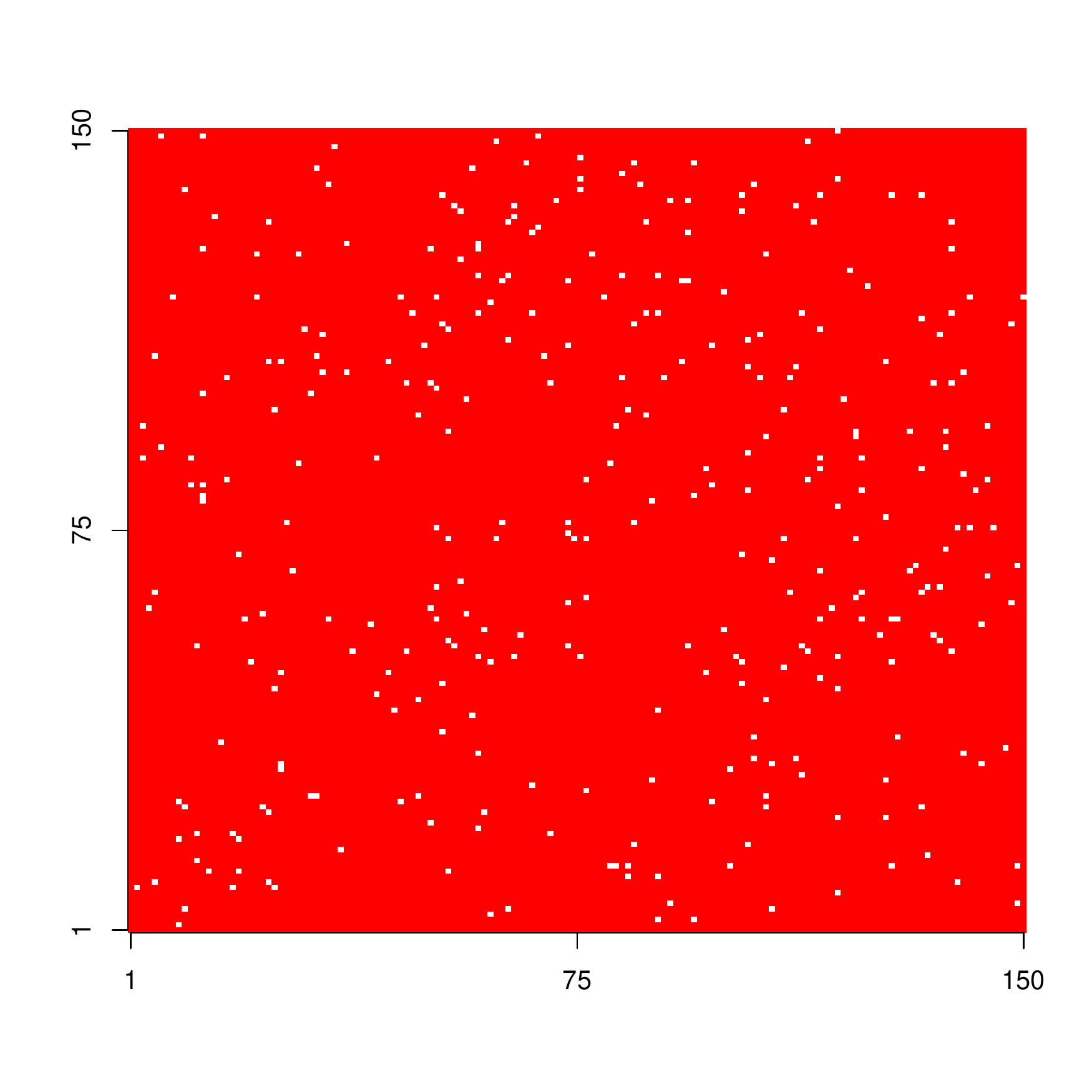}
		\includegraphics[width=0.24\textwidth]{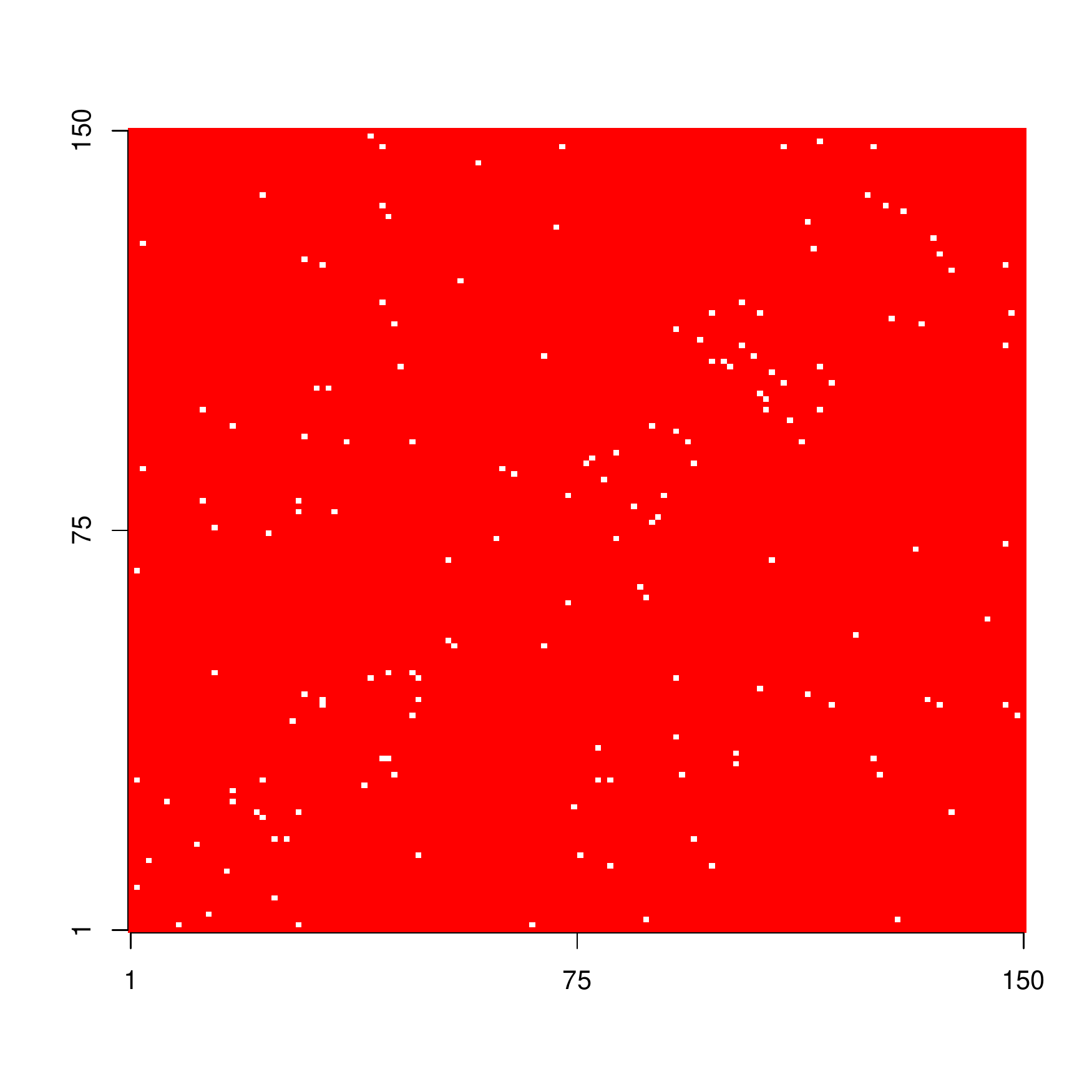} 
		\includegraphics[width=0.24\textwidth]{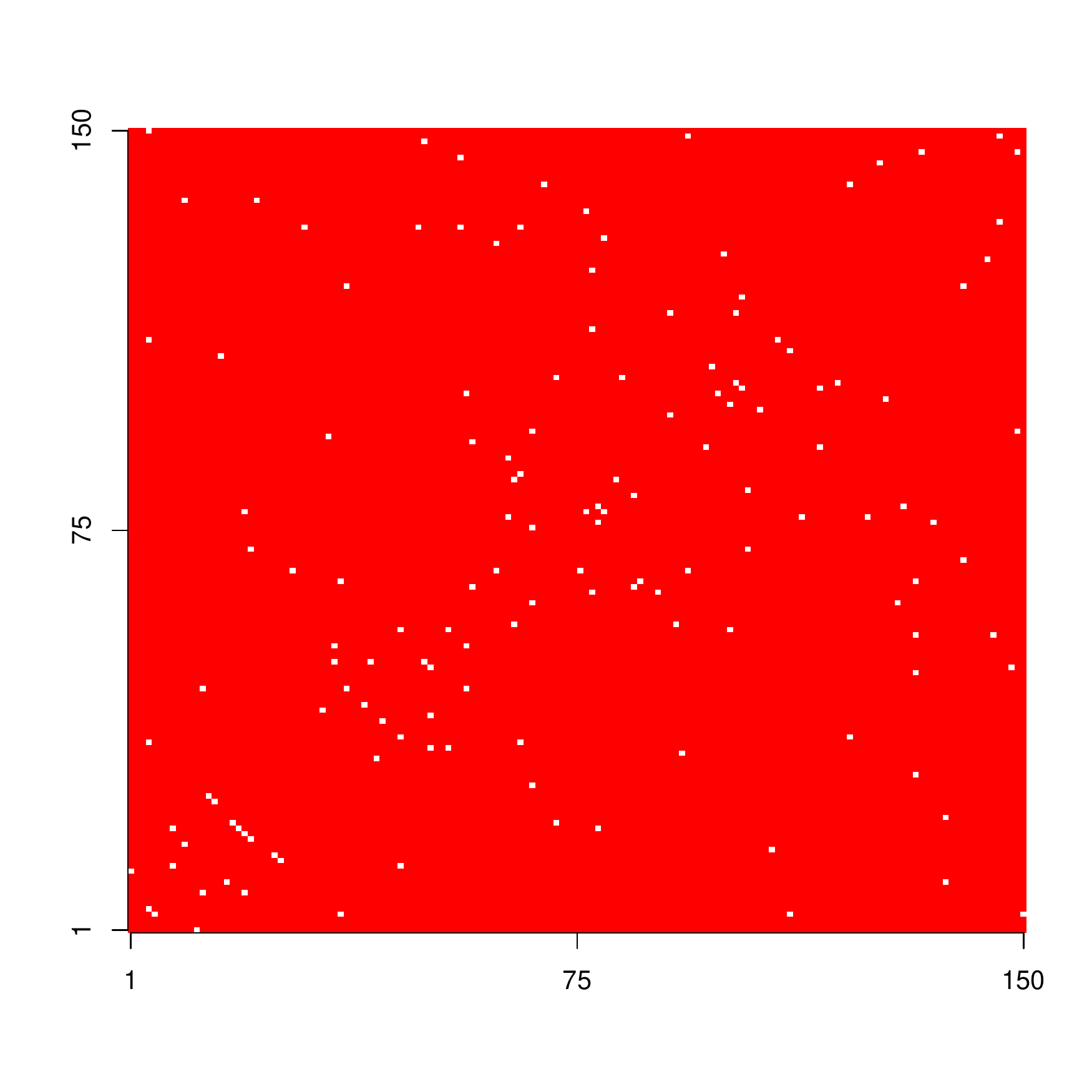}
		\includegraphics[width=0.24\textwidth]{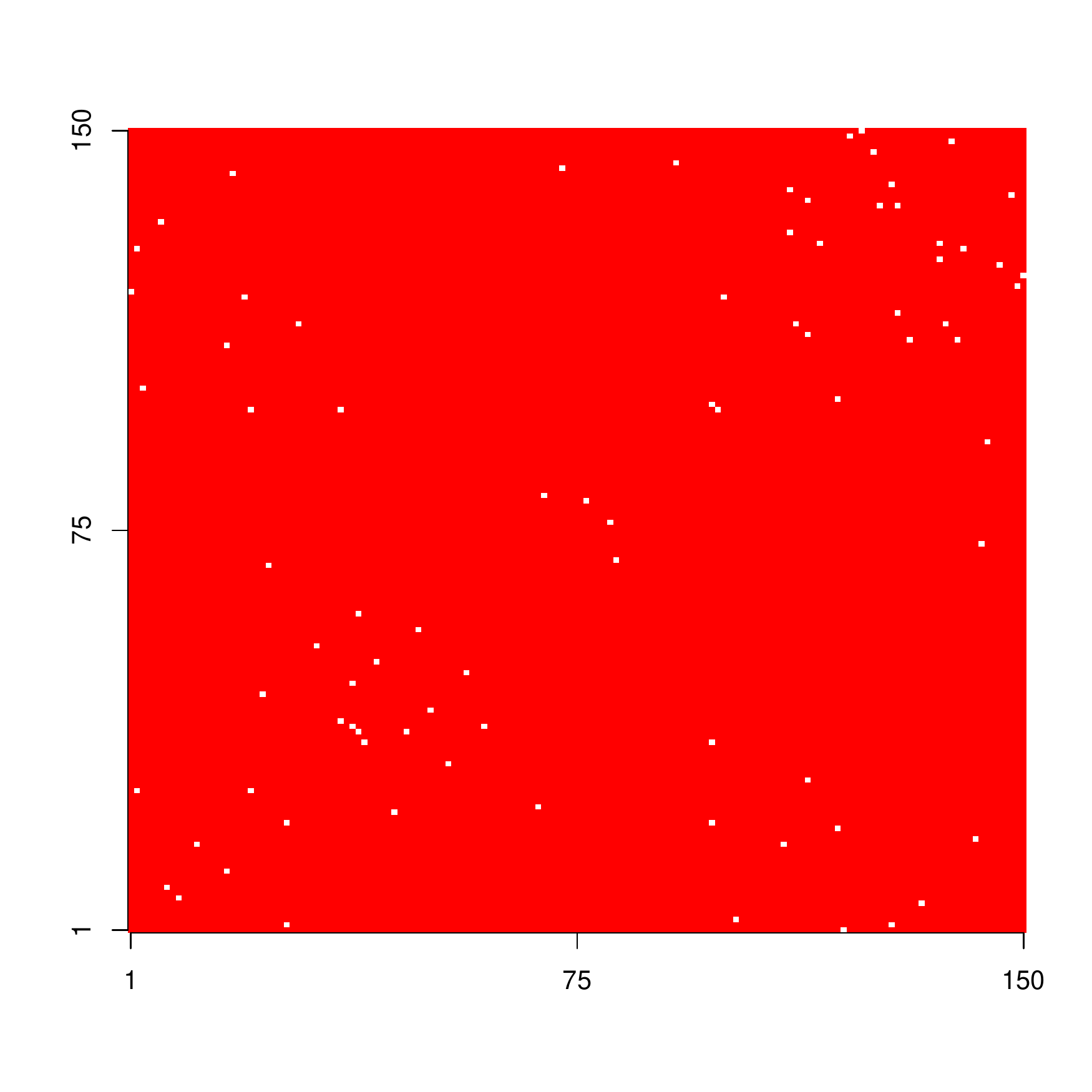}
		\includegraphics[width=0.24\textwidth]{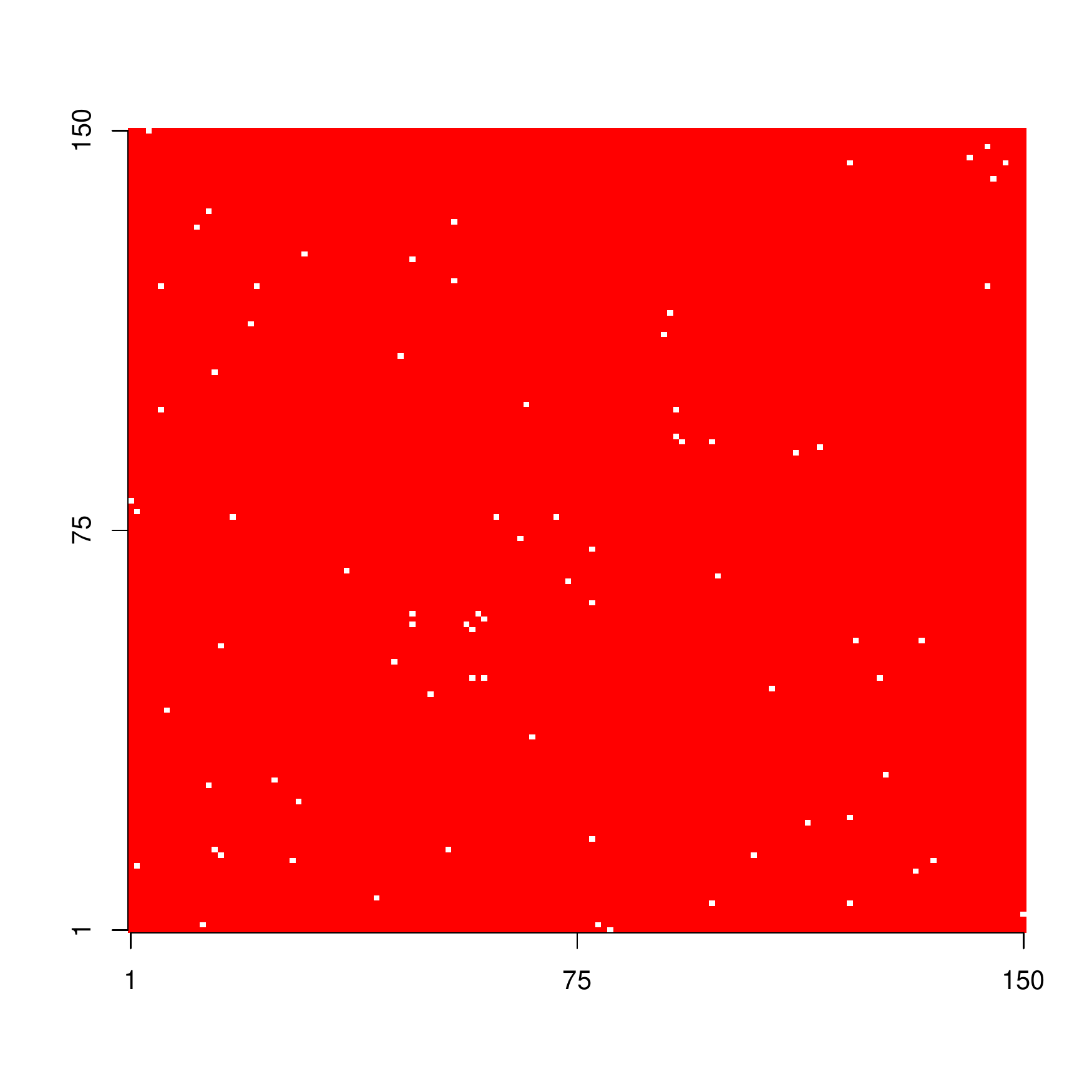} 
		\includegraphics[width=0.24\textwidth]{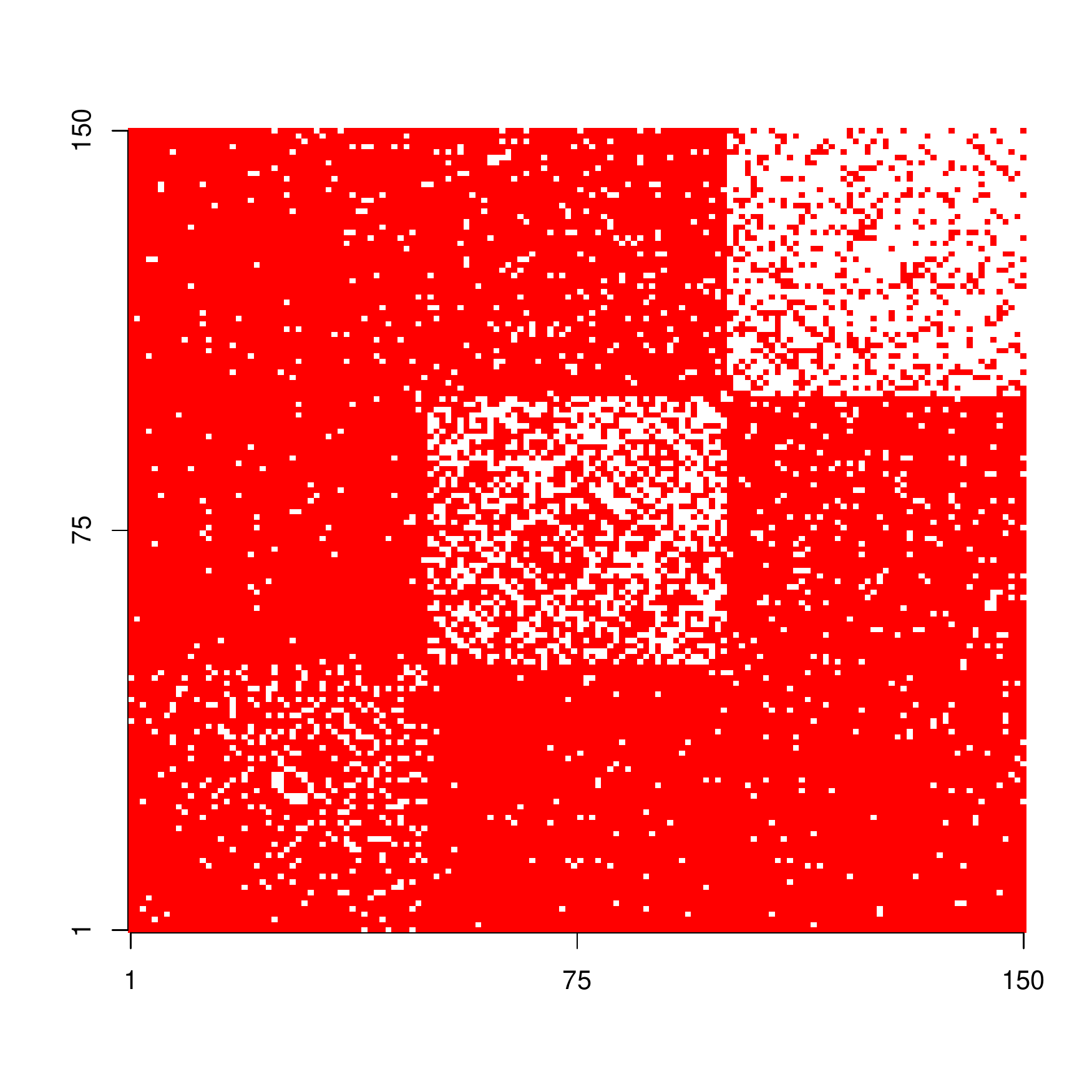} 
		\includegraphics[width=0.24\textwidth]{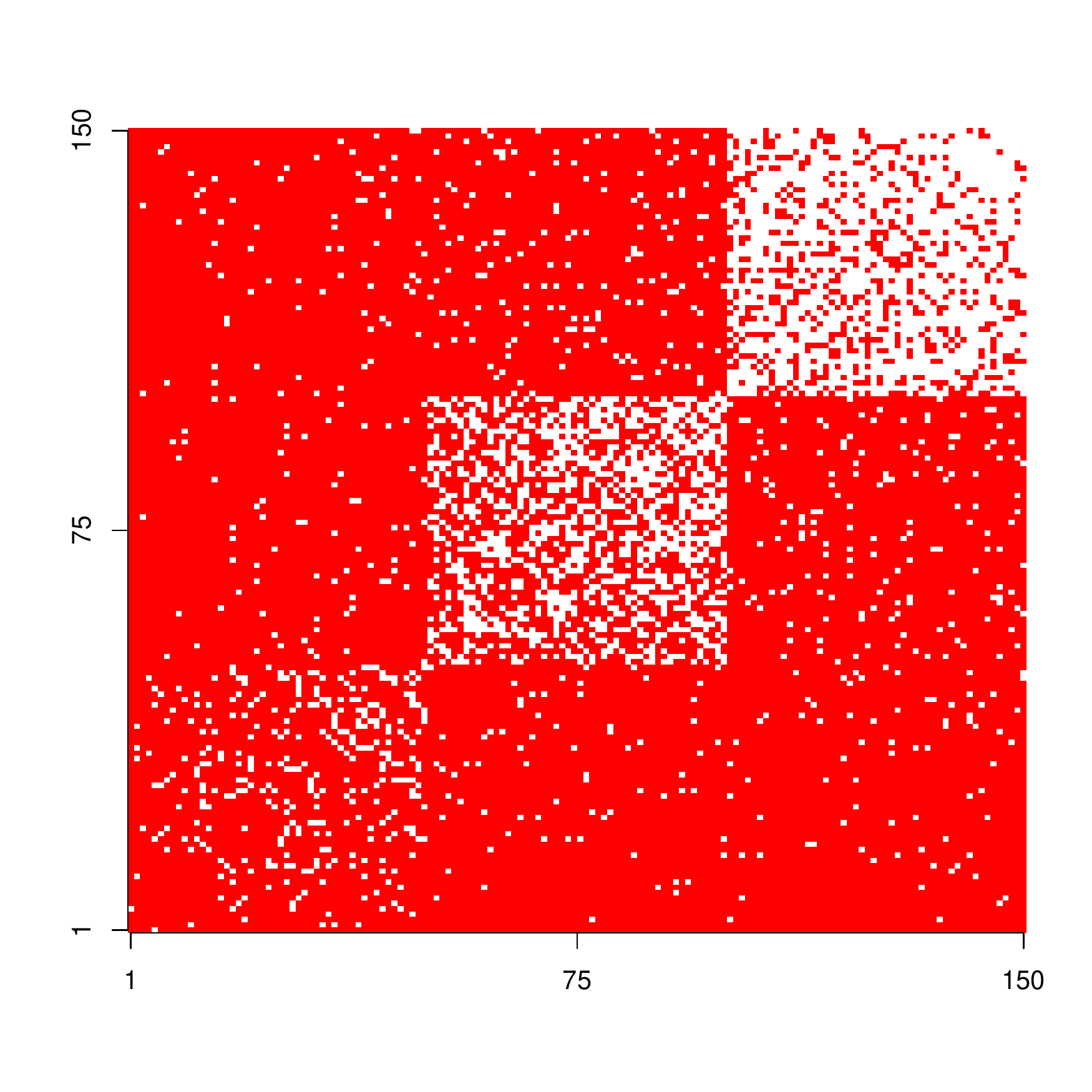}
		\includegraphics[width=0.24\textwidth]{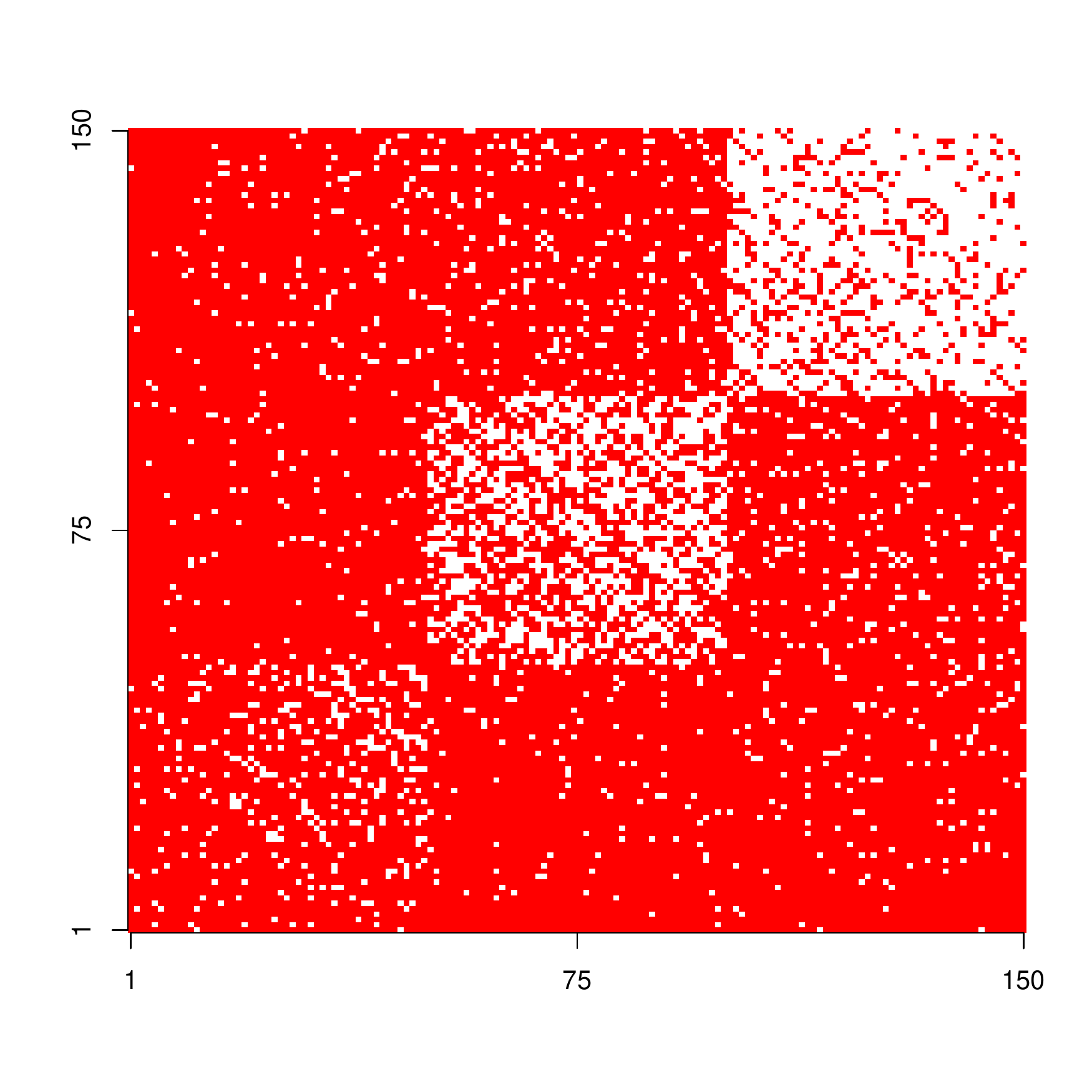}
		\includegraphics[width=0.24\textwidth]{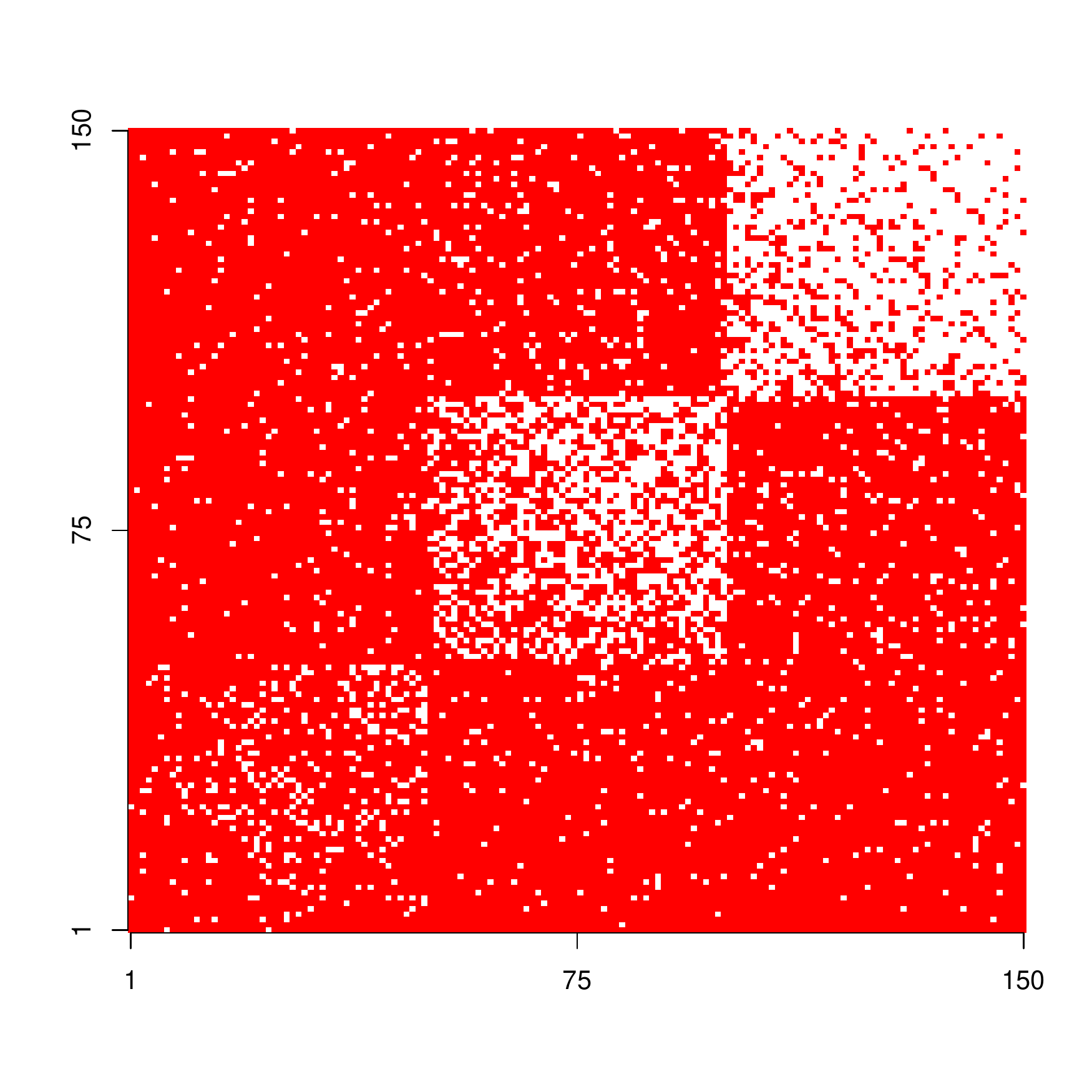} 	
		\includegraphics[width=0.24\textwidth]{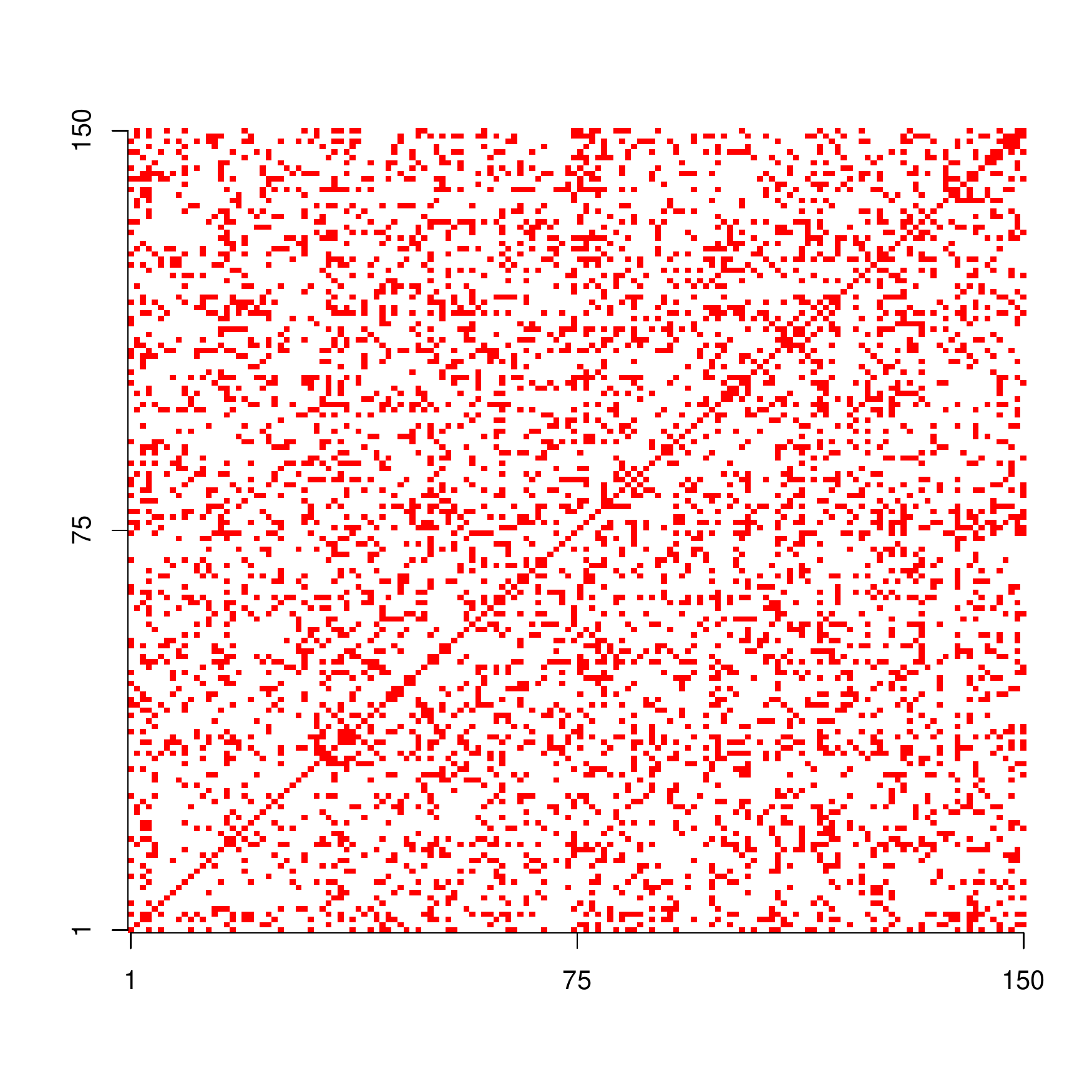} 
		\includegraphics[width=0.24\textwidth]{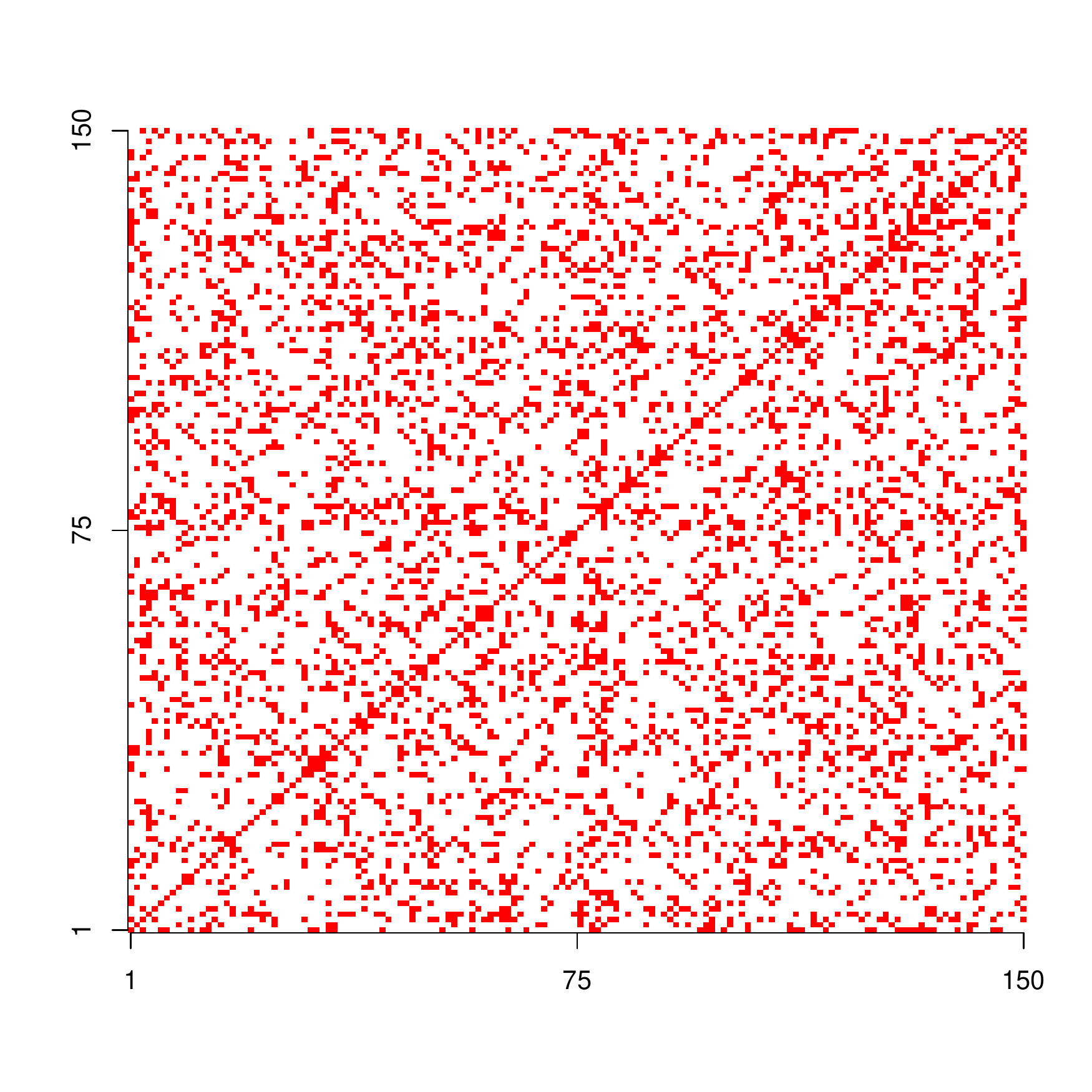}
		\includegraphics[width=0.24\textwidth]{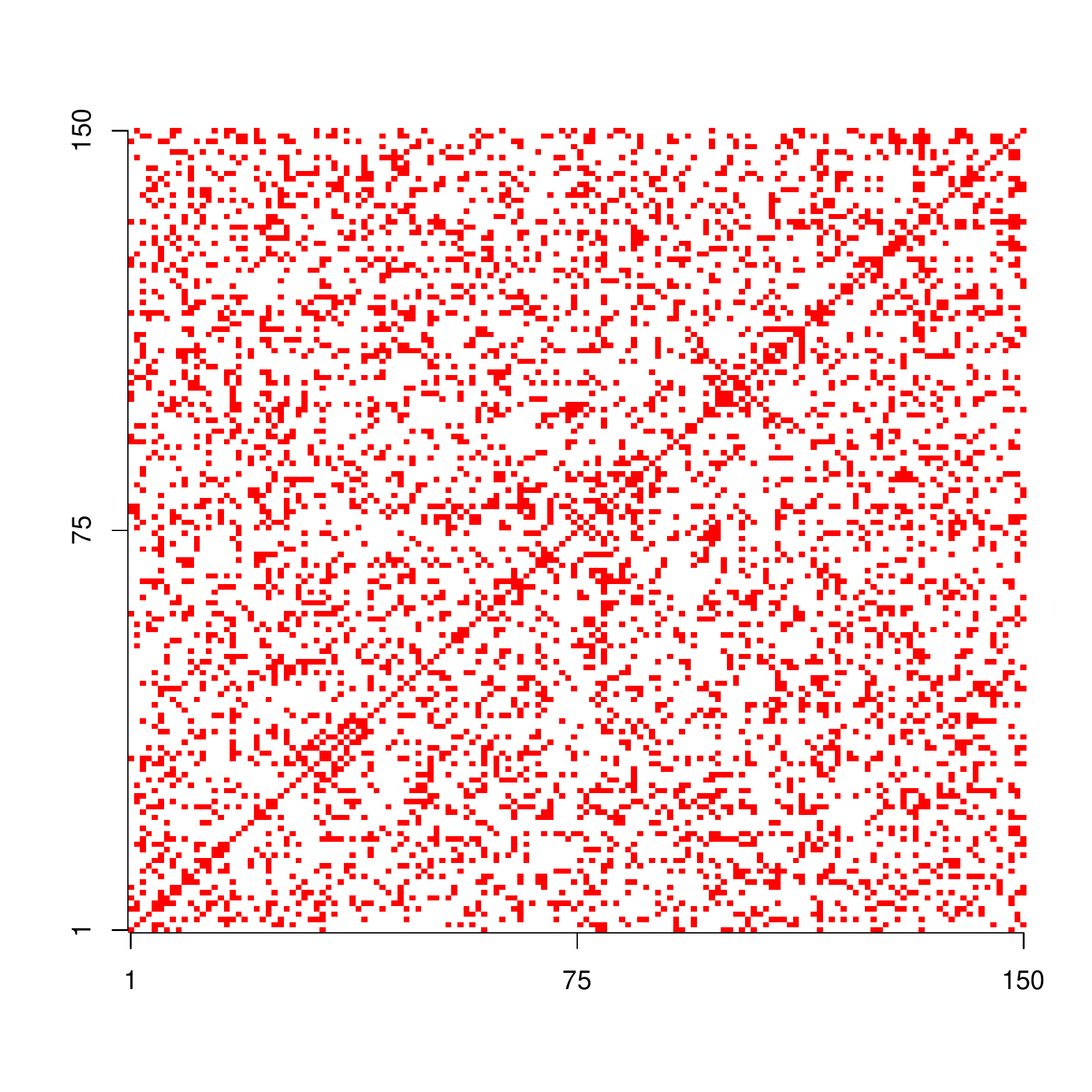}
		\includegraphics[width=0.24\textwidth]{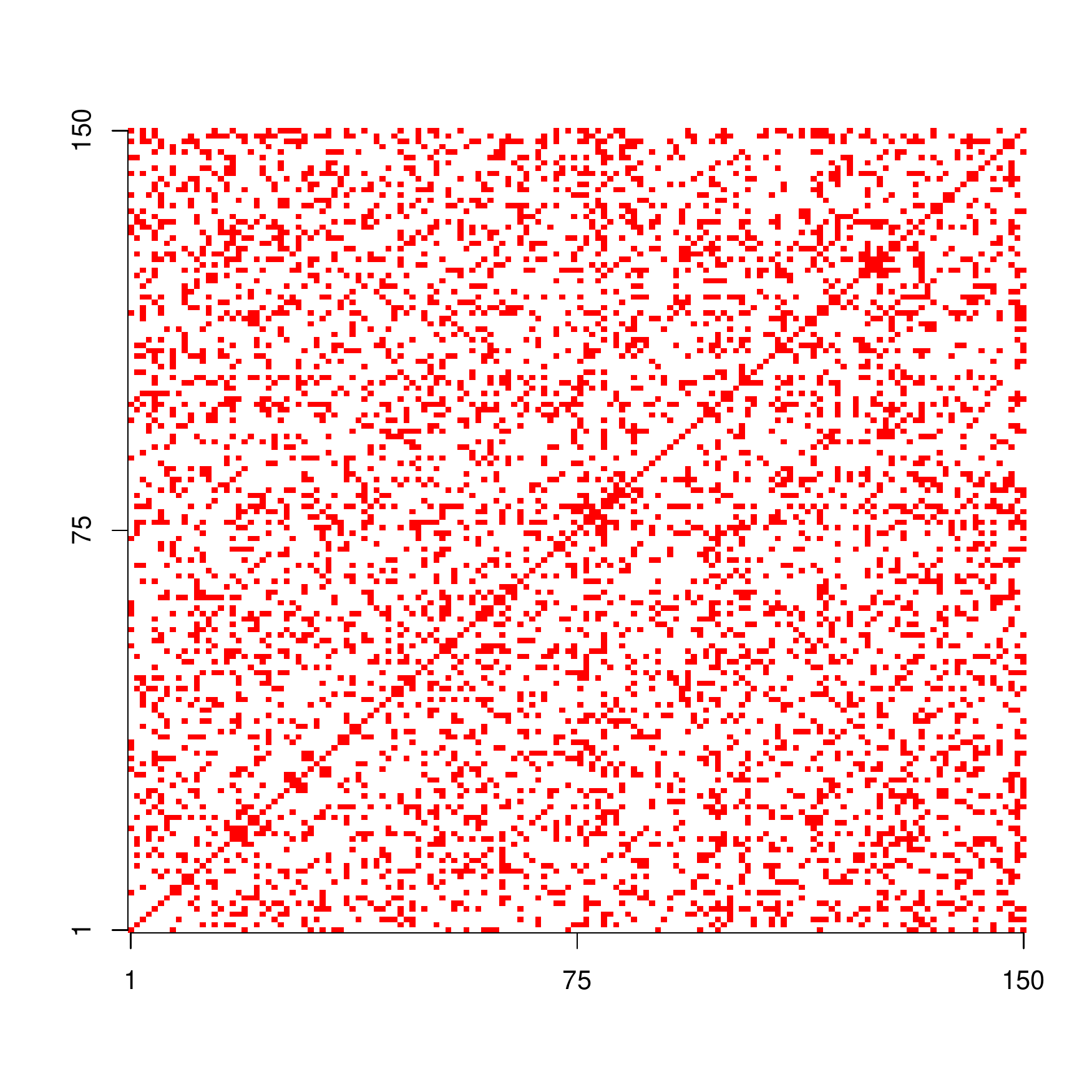} 				 						
		\caption{\label{fig3} Examples  of  adjacency  matrices generated  under different scenarios.  The first to the fourth rows correspond to the first to the fourth scenarios, respectively.  In each row, from left to right, the first two plots correspond to networks generated  before the change point, and the last two plots  to networks generated  after the change point.  In each display, a white dot indicates one and a red dot indicates zero. }
	\end{center}
\end{figure}

\begin{table}[t!]
	\centering
	\caption{\label{tab1}  Upper bounding overall Type-I errors: $n$, the network size; $\alpha$, the Type-I error upper bound; $T_{\mathrm{train}}$, the time length of the training data used for selecting tuning parameters; ORI, \cite{chen2019sequential} using the original edge-count scan statistic; W, \cite{chen2019sequential} using the weighted edge-count scan statistic; G, \cite{chen2019sequential} using the generalised edge-count scan statistic.}
	\medskip
	\setlength{\tabcolsep}{6pt}
	\begin{tabular}{ rrrr|rrrr|rrrrr}
		\hline
		\multicolumn{4}{c|}{Settings} & \multicolumn{4}{c|}{Delay}& \multicolumn{4}{c}{PFA} \\
		\hline
		$n$ & Scenario     &$\alpha$       & $T_{\mathrm{train}}$            & Algo. \ref{alg-online-net}           &  ORI          &        W          &G          & Algo. \ref{alg-online-net}                    &  ORI          &        W          &G    \\  
		\hline  
		150&    1                & 0.01             &200           &      \textbf{35.38}                           &  145.92     &   147.44    & 147.44   &                 \textbf{0.00}                                 &   0.02&   \textbf{0.00}         &  \textbf{0.00}  \\    
		150  & 1                  &    0.05        &   200        &       \textbf{32.93}                           &    135.41      &  137.42 &    137.42       &        \textbf{0.02}                                  &    0.04&      0.03       &     0.03        \\           
		150   &     1            &    0.01         &150          &             \textbf{33.10}                     &   145.92       &       146.08  &    146.08   &          0.02                     &   0.02      & \textbf{0.00}  &        \textbf{0.00}                \\     
		150 &     1             &   0.05          &  150         &           \textbf{30.61}                           &     140.09  &      139.55&     139.55 &                  \textbf{0.02}                                 & 0.03 &          \textbf{0.02}        &       \textbf{0.02}           \\     
		100 &      1             &    0.01         &200          &            \textbf{89.74}                      &  147.66    &       148.97    &     148.97   &            \textbf{0.00}                          & 0.02     &   \textbf{0.00}        & \textbf{0.00}          \\     
		100 &   1                 &      0.05        &  200        &              \textbf{54.90}                     &   135.94   &       142.88  &  142.88    &           \textbf{0.05}                                &      0.07           &     \textbf{0.05}            &   \textbf{0.05}      \\     
		100      &           1        &      0.01        &150   &                          \textbf{73.14}                &        148.95&     149.83  &      149.83   &          \textbf{0.00}                     &        0.02     &   \textbf{0.00}   &   \textbf{0.00}  \\     
		100     &           1         &       0.05     &150           &                   \textbf{50.52}                 &   135.94     &   141.77  &         141.77      &                 0.06                            &    \textbf{0.04}  &           0.05            &    0.05   \\     
		150      &      2           &        0.01      &     200     &                 \textbf{25.00}               &      149.36&        150.00  &   150.00     &                \textbf{0.00}                 &    \textbf{0.00}         &  \textbf{0.00}       &    \textbf{0.00}       \\     
		150    &       2            &      0.05        &    200      &               \textbf{22.84}              &        149.14    &   146.16     &   146.16       &                        0.05                         &     0.05   & \textbf{0.03}    &  \textbf{0.03} \\     
		150   &     2          &        0.01             &   150    &                \textbf{24.64}               &     150.00 &  148.39  &      148.39       &                         \textbf{0.00}                         &  \textbf{0.00} & 0.02 &         0.02   \\     
		150    &        2       &      0.05   &       150             &              \textbf{22.34}               &      149.14   &  144.76       &   144.76          &                        \textbf{0.02}                &      0.06&       0.05       &       0.05      \\     
		100  &       2          &         0.01        &    200    &                  \textbf{92.90}                    &     150.00      &     150.00       &        150.00   &                 0.05                          &        \textbf{0.01}    &      0.02 &   0.02       \\     
		100    &          2        &       0.05        &     200      &                   \textbf{68.48}            &    150.00     &     147.35      &      145.88     &                         0.06                            &         \textbf{0.03}         &              0.04        &   0.04       \\     
		100     &          2        &         0.01     &       150      &              \textbf{92.94}                 &  150.00    &       150.00    &       150.00   &                  0.02                        &    0.02   &       \textbf{0.01}     &    \textbf{0.01}   \\      
		100 &          2         &           0.05& 150           &                \textbf{68.48}              &    149.12   &    150.00      &      150.00    &         0.06                                    &       0.04 &       \textbf{0.03}      &  \textbf{0.03}      \\         
		150  &         3           &            0.01   & 200       &               \textbf{26.62}                        &      150.00  &    148.61   &       73.05    &            \textbf{0.00}                           &          0.02       & 0.03              &        0.03       \\     
		150    &      3       &            0.05      &      200     &               \textbf{17.74}                      &       147.23 &      140.67    &   62.88     &         \textbf{0.02}                               &    0.06 &   0.04             &        0.04     \\     
		150  &        3         &        0.01     &   150     &             \textbf{25.54}                      &      150.00  &          150.00     &          76.54      &                   \textbf{0.00}                   &  0.01 &               0.01        &        0.02    \\     
		150 &        3          &           0.05       &     150    &                 \textbf{17.11}              &       146.80     &     138.90  &        59.85     &                 \textbf{0.03}                      &          0.05      &  0.04            &     0.04        \\                
		100 &       3      &       0.01        & 200    &                       \textbf{38.36}              &       150.00  &      149.85    &     75.98  &               0.01                    &   \textbf{0.00}        &            0.01          &    0.01         \\      
		100 &      3      &    0.05        &    200     &                         \textbf{35.48}               &  149.19    &     148.45     &             64.6    &               \textbf{0.02}                             &    \textbf{0.02}      &   0.05            &           0.05 \\         
		100 &       3        &          0.01          &    150       &                   \textbf{35.23}              &  149.57      &    149.85    &     77.01    &  0.01                          &    \textbf{0.00}      &             0.01             &         0.01               \\     
		100&       3          &      0.05           &       150       &               \textbf{36.06}           &       149.14     &     148.35        &       58.07      &                         0.04                                &   \textbf{0.03}       &      0.06                 &        0.08    \\     
		150 &     4          &            0.01       &       200     &                   \textbf{3.68}       &      13.67            &              10.33     &     10.78      &               \textbf{0.00}                  &            0.03     &        0.04               &        0.03     \\     
		150 &      4    &       0.05      &     200        &               \textbf{3.35}                 &      12.77     &           9.81   &          10.10         &                       \textbf{0.05}                             &           0.07      &    \textbf{0.05}                         &     0.06             \\      
		150&         4      &      0.01    &    150         &                   \textbf{3.70}                   &      14.19    &        11.19       & 12.01       &                          \textbf{0.01}                      &     \textbf{0.01}            &           0.02            &      0.01       \\      
		150  &       4         &         0.05       &  150           &                 \textbf{3.33}                     &   12.60     &           9.68   &   10.01       &                        \textbf{0.03}                  &           0.07      &        0.05               &        0.04     \\         
		100&    4          &       0.01          &    200   &                   \textbf{4.00}                  &         16.51    &        13.12      &    13.12   &                       \textbf{0.00}                  &      \textbf{0.00}       &    \textbf{0.00}        &            \textbf{0.00}    \\     
		100  &     4            &       0.05         &   200    &                \textbf{3.80}                          &    14.28     &  12.05        &      12.03    &                   0.04                             &        0.05         & \textbf{0.03}            &         0.05    \\     
		100  &     4         &       0.01     &          150       &                        \textbf{4.00}            &  16.95             &      12.66     &   12.62       &            \textbf{0.00}                               &   \textbf{0.00}               &  0.02              &    0.03  \\     
		100 &     4         &        0.05       &       150   &                   \textbf{4.00}                 &      14.91           &      12.14       &       12.12     &                   0.05                     &  \textbf{0.03}               &     \textbf{0.03}        & 0.05 \\      
		\hline	         
	\end{tabular}
\end{table}

\begin{table}[t!]
	\centering
	\caption{\label{tab2}  Lower bounding the average run lengths:  $n$, the network size; $\gamma$, the average run length lower bound; ORI, \cite{chen2019sequential} using the original edge-count scan statistic; W, \cite{chen2019sequential} using the weighted edge-count scan statistic; G, \cite{chen2019sequential} using the generalised edge-count scan statistic.}
	\medskip
	\setlength{\tabcolsep}{7pt}
	\begin{tabular}{ rrr|rrrr|rrrrr}
		\hline
		\multicolumn{3}{c|}{Settings} & \multicolumn{4}{c|}{Delay}& \multicolumn{4}{c}{PFA} \\
		\hline
		$n$ & Scenario            &$\gamma$     & Algo.\ref{alg-online-net-variant}         &  ORI          &        W          &G          & Algo. \ref{alg-online-net-variant}           &  ORI          &        W          &G    \\  
		\hline  
		150 &        1                   & 150            &         \textbf{20.77}                             &  71.09         &      80.92         &  70.26 &         \textbf{0.28}                                      &     0.56    &             0.50 & 0.54\\                 
		150  &             1             &      200          &             \textbf{23.56}                     &    93.52         &        101.94       &   103.66  &            \textbf{0.18}                                    &       0.32  &          0.32    &  0.34\\    
		100  &             1             &     150           &             \textbf{29.57}                     &     88.01      &         75.59    & 101.90   &              \textbf{0.34}                                  &       0.52    &      0.56     & 0.56 \\                          
		100 &             1             &       200          &          \textbf{34.71}                     &     97.54         &      104.00       &  105.75&                   \textbf{0.22}                           &     0.31      & 0.28             &   0.28\\                       
		150  &            2             &       150        &            22.08                                &     136.65   &          10.00      &      \textbf{2.22}  &                \textbf{0.11}                   &  0.48       &          0.62               &  0.64\\      
		150 &               2             &           200      &          23.48                                       &    139.94     & 17.3            &      \textbf{12.76}  &                   \textbf{0.06}             &0.30     &     0.41            &  0.41\\   
		100&        2                      &     150         &            \textbf{44.58}                         &   121.55    &     58.78           &   64.15    &                        \textbf{0.52}                       &    0.60        &     0.62       &         0.62   \\    
		100  &         2                 &       200     &                \textbf{51.83}                        &133.44           &     59.92          &    56.46 &                    0.52                                     &    \textbf{0.46}     &         \textbf{0.46}       &  0.49\\   
		150 &         3               &          150       &              \textbf{9.44}                        &    119.12   &     102.64       &     33.78  &                      \textbf{0.00}                         &     0.50       &     0.66              &0.62  \\   
		150 &             3               &      200        &             \textbf{11.08}                           &     128.34     &     112.76   &37.82    &                     \textbf{0.00}                           &      0.32     &    0.40              &  0.42\\           
		100    &         3                &     150          &               \textbf{22.17}                       &   115.76   &          114.00         &   38.66  &                       \textbf{0.08}                          &  0.66   &  0.52           & 0.46 \\
		100   &          3                &      200           &                \textbf{ 24.97}                        &    141.5    &    133.59         &     42.20    &                     \textbf{0.06}                         &     0.44        &      0.37                   &  0.32\\            
		150&          3              &           150     &              \textbf{2.21}                             &  9.00    &    7.53                      &7.82        &                    \textbf{0.45}                            &            0.53   &         0.48              &  0.54 \\    
		150 &          3                 &         200      &                   \textbf{2.37}                        & 10.21       &      8.48             &  8.81         &                       \textbf{0.26}                              & 0.36      &         0.38           & 0.37 \\   
		100  &         3                  &        150      &                  \textbf{3.80}                        &     10.28     &        9.59      &     9.59     &               \textbf{0.18}                                             &        0.50       &        0.47           &  0.58\\                                                                  
		100 &           3                 &         200      &                     \textbf{3.95}                                 & 10.97        &      9.88          &        10.63     &            \textbf{0.13}                                              &         0.30      &         0.30                &0.28  \\             
		\hline	         
	\end{tabular}
\end{table}

\paragraph{Scenario 4.} This is a random dot product graph  \citep{YoungScheinerman2007} with fixed latent positions.  First, we generate the latent positions $X \in \mathbb{R}^{n \times 5}$  as
	\[
		X_{ij}  \stackrel{\mbox{i.i.d.}}{\sim}  \mathrm{Unif}[0,1], \quad i = 1,\ldots, n, \, j = 1,\ldots,5,
	\]
	which are kept fixed throughout our simulations.  We then construct  
	$\widetilde{X} \in \mathbb{R}^{n \times 5}$  as
	\[
		\tilde{X}_{ij}  \stackrel{\mbox{i.i.d.}}{\sim}  \mathrm{Unif}[0,1], \quad i = 1,\ldots, n, \, j = 1,\ldots,5,
	\]
	which are also kept fixed throughout our simulations.   Finally, the data are generated as 
	\[
		A_{ij}(t) \sim \text{Bernoulli}\left(       \frac{ X_i^{\top} X_j   }{  \|X_i\|  \,\|X_j\|  }        \right  ), \quad t \in \{1, \ldots, \Delta\} 
	\]
	and
	\[
		A_{ij}(t) \sim \text{Bernoulli}\left(          \frac{ Y_i^{\top} Y_j   }{  \|Y_i\|  \,\|Y_j\|  }        \right  ), \quad t \in \{\Delta + 1, \Delta + 2, \ldots, T\},
	\]
	where $X_i, \widetilde{X}_i \in \mathbb{R}^5$ are the $i$th rows of the matrices $X$ and $\widetilde{X}$, $\|\cdot\|$ is the $\ell_2$-norm of vectors, and
	\[
		Y_ i  \,=\, \begin{cases}
		\widetilde{X}_i, & i \leq \floor{n/4},\\
		X_i, & \text{otherwise.}  
		\end{cases}  
	\]

We collect the results in \Cref{fig3}, Tables~\ref{tab1} and \ref{tab2}.  In \Cref{fig3}, we exhibit one realisation each for each scenario.  Each row corresponds to each scenario, from the first to the fourth.  In each row, the left two panels are realisations before change points, and the right two panels are the post change points realisations.  It can be seen from \Cref{fig3} that, these four scenarios cover different types of networks, and the change points are hard to spot with the naked eye.

Tables~\ref{tab1} and \ref{tab2} correspond to the two different ways to control Type-I errors.  We reiterate that \Cref{alg-online-net}, \Cref{thm-net-delay-varying-N} and \Cref{tab1} correspond to the strategy of controlling the overall Type-I error $\alpha$.  \Cref{alg-online-net-variant}, \Cref{thm-selection} and \Cref{tab2} correspond to the strategy of lower bounding the average run length $\gamma$. 

We can see that if we choose to upper bound the overall Type-I error, then \Cref{alg-online-net} outperforms all three versions of \cite{chen2019sequential}.  If we choose to lower bound the average run length, then \Cref{alg-online-net-variant} still outperforms all three competitors except in all instances of Scenario 2.  In fact, \Cref{alg-online-net} also performs worst in Scenario 2 out of all four scenarios.  A possible reason why most methods suffer with Scenario 2 is that this is the model that has the largest $r$, the rank of the difference of the graphons.  It is understood that the USVT algorithm \citep[\Cref{alg-usvt}][]{Chatterjee2015} is less effective when the rank is relatively large.  This is also reflected in the detection delay rate, which is linear in $r$.

\subsection{Stock market data}
\label{sec:stock}

We consider  stock market data   from  April 1990 to January 2012. The data  consist of  the weekly log returns for the Dow Jones Industrial Average index and they are available in the R \citep{R} package ecp \citep{Recp}.  To construct networks, we first use a sliding window of window width being 3, and consider the covariance matrix  among 29 companies' log-weekly-returns over a 3 week period.  We then truncate the covariance matrices  by setting those entries which have values above the 0.95-quantile as 1, and the remaining as 0.  This construction leads to sparse networks.  Some examples of these networks are illustrated in the first two  rows of  Figure \ref{fig4}.

As competitors to our estimator, we consider the   same statistics  (ORI), (W) and  (G)  from \cite{chen2019sequential}  that were  used in \Cref{sec:sim}. To evaluate the performances of these methods. we have chosen two periods of the original data, each consisting of a training set and a test set.  We calculate the maximum score of each method using the training data and use the maximum score as the threshold for detecting false alarms.  

In the first period, the data  from 2-Apr-1990  to 4-Jan-1999 are used as the training set, and the data from 25-Jan-1999  to 31-May-2004 are used as the test set.  The algorithm we proposed in \Cref{alg-online-net}  detects a change point corresponding to 25-Mar-2002.  This seems to coincide with the period of financial turbulence after the  11-Sep-2001  terrorist attacks. We can also see in the first row of \Cref{fig4} that there seems to be a change in pattern around such date. In contrast, the competitor methods did not detect the change point with the given choice of threshold.

In the second period,  the data from  31-May-2004  to  15-Jan-2007 are used as the training set, and the data from  5-Feb-2007  to 1-Mar-2010 are used as the test set.  We remark that the training data correspond to the period before the financial crisis of 2007--2008.   Algorithm \ref{alg-online-net} detects a change point  corresponding to the date 10-Mar-2008.  The competing approaches  detect a change point in the same period.  Specifically, ORI detects the date December 17, 2007;  and both W and G  detect the date  19-Mar-2007.  From looking at the second row of \Cref{fig4}, we can see that the 2007-2008 financial crisis seems to affect the network patterns  in the  data.






\begin{figure}[t!]
	\begin{center}
		\includegraphics[width=0.24\textwidth]{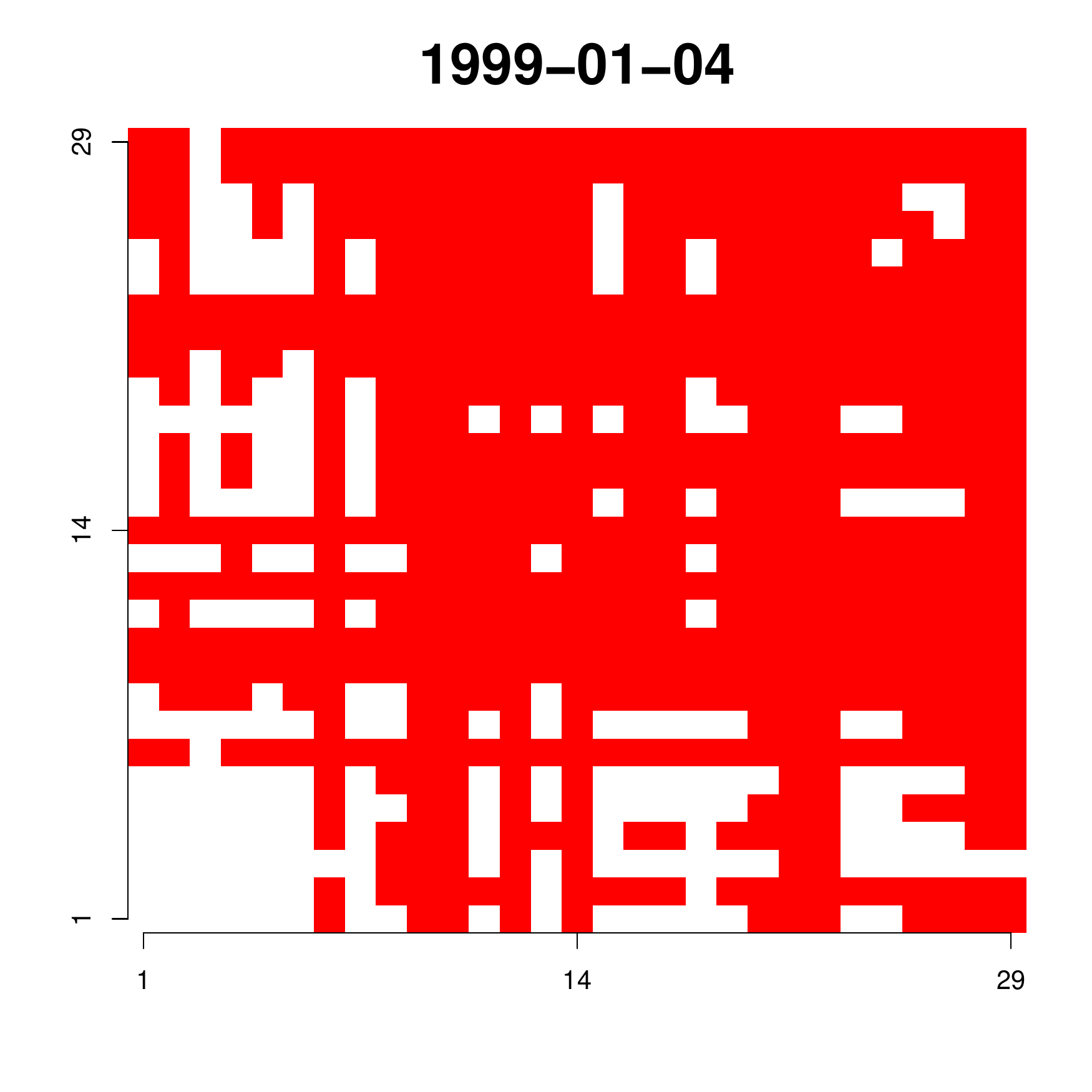} 
		\includegraphics[width=0.24\textwidth]{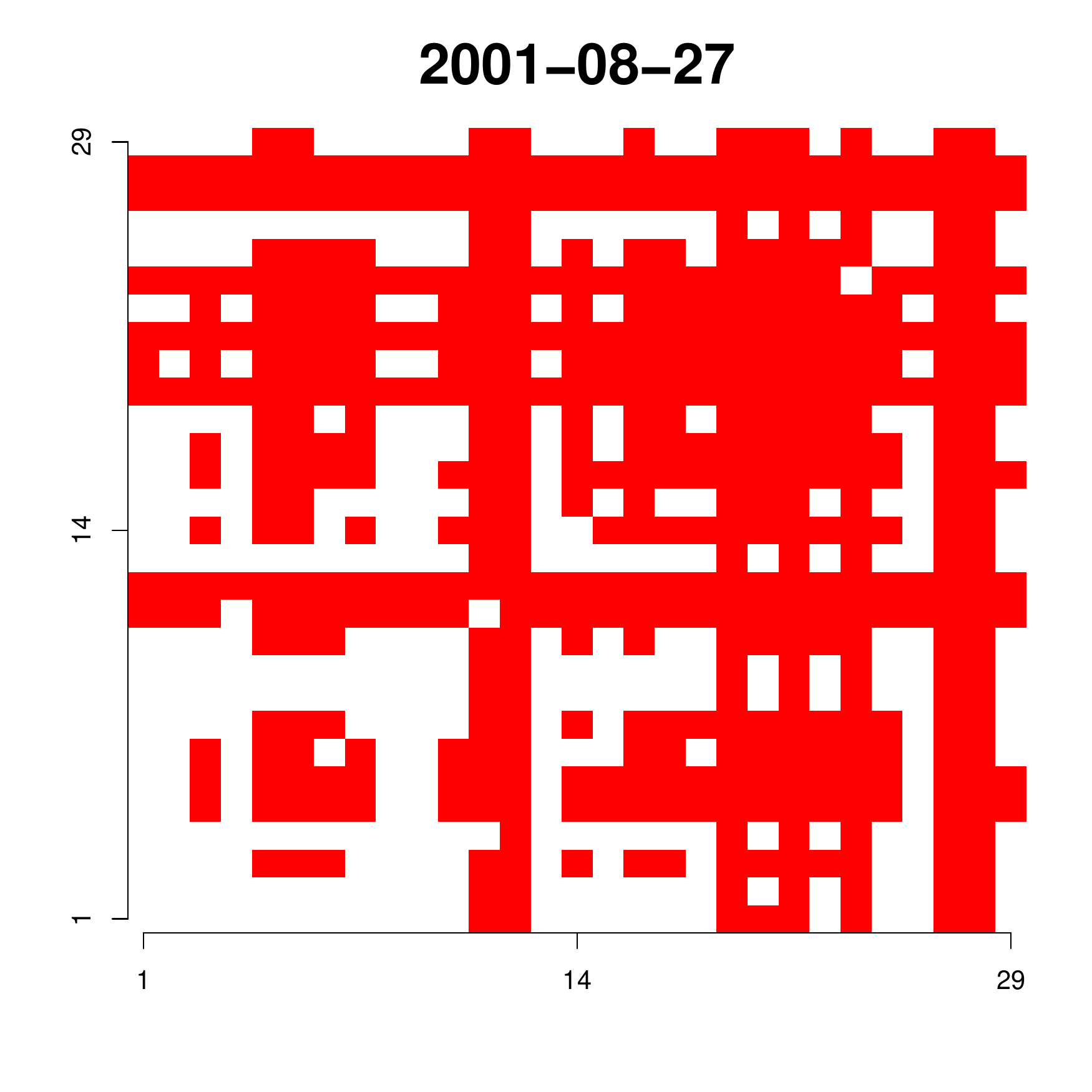} 	
		\includegraphics[width=0.24\textwidth]{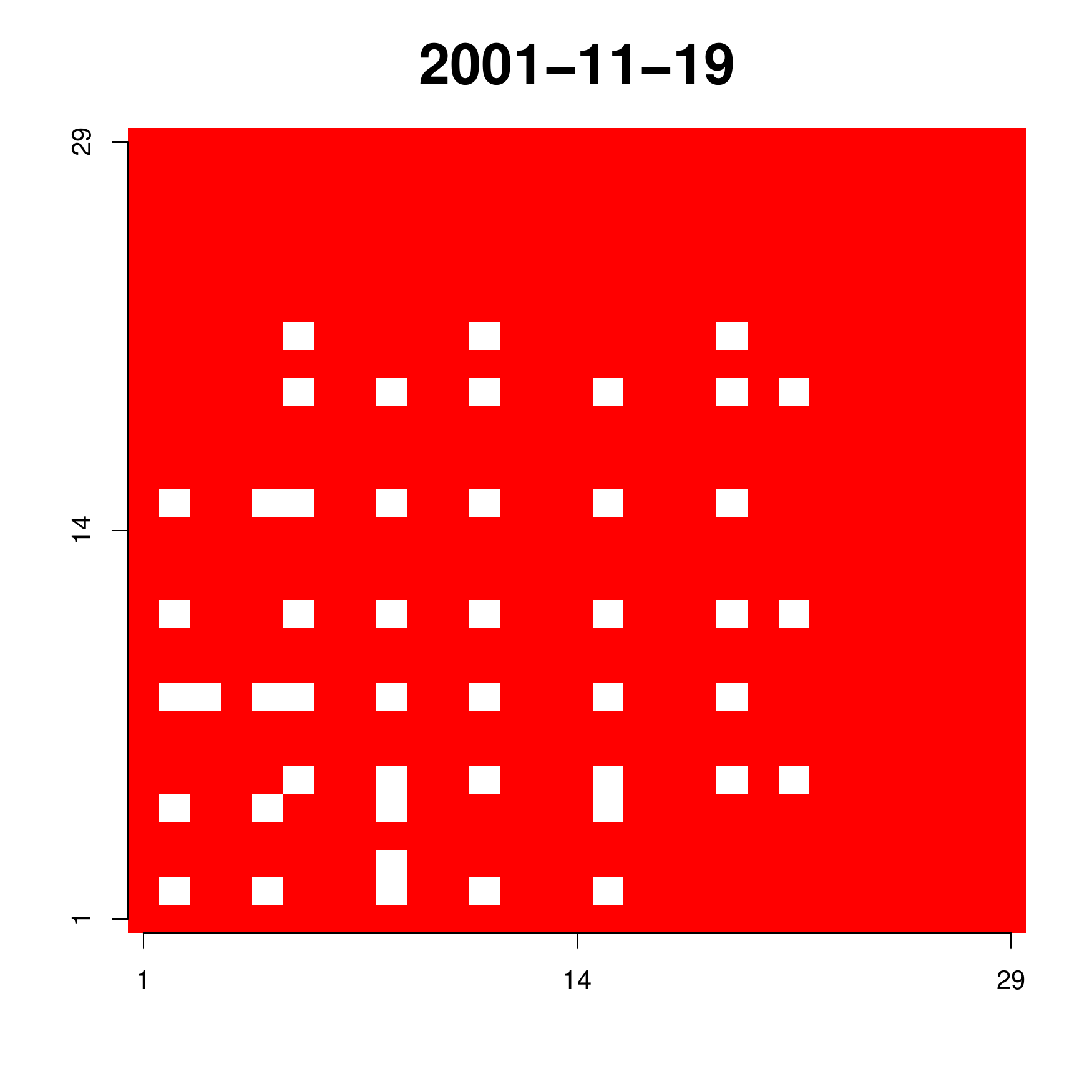} 
		\includegraphics[width=0.24\textwidth]{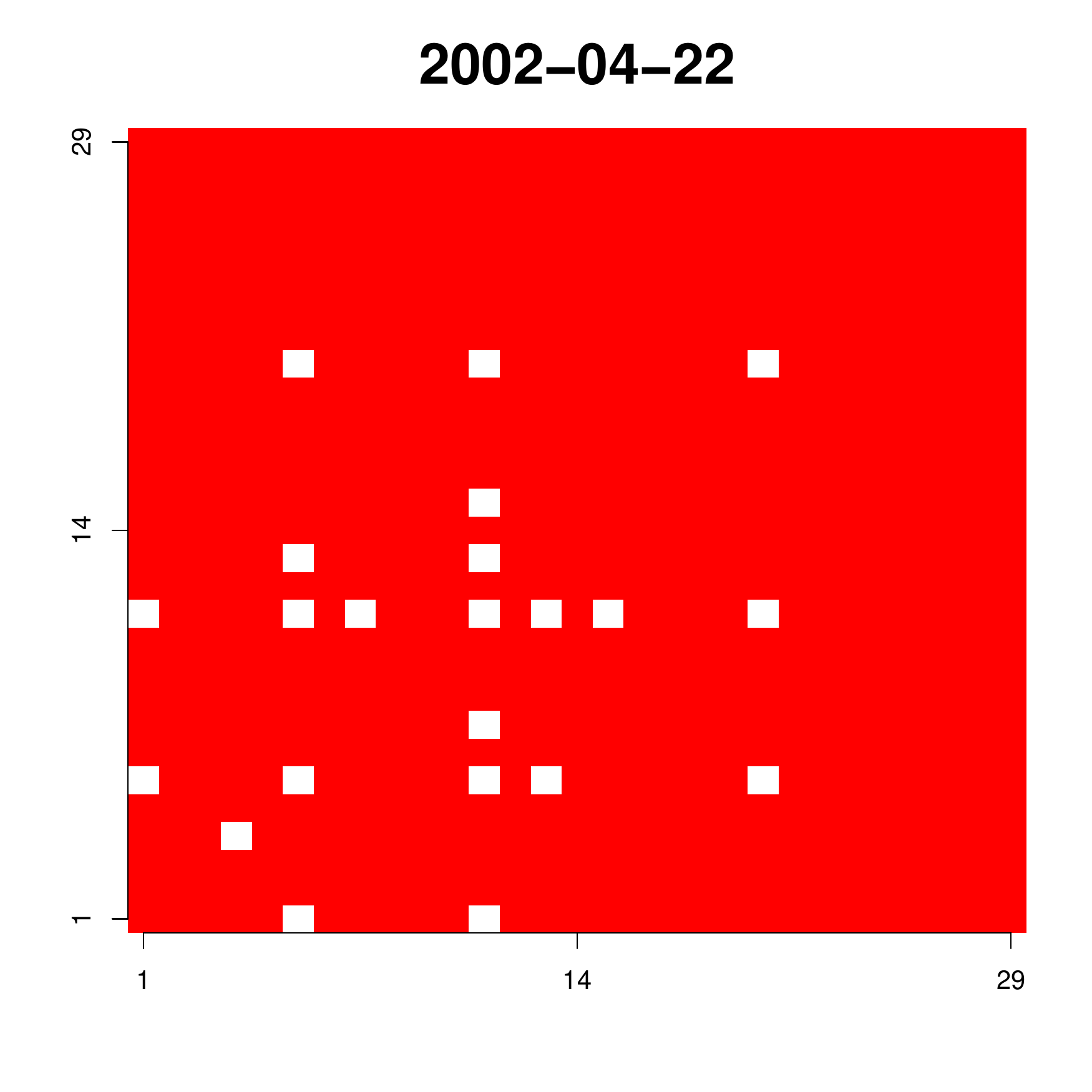} 
		\includegraphics[width=0.24\textwidth]{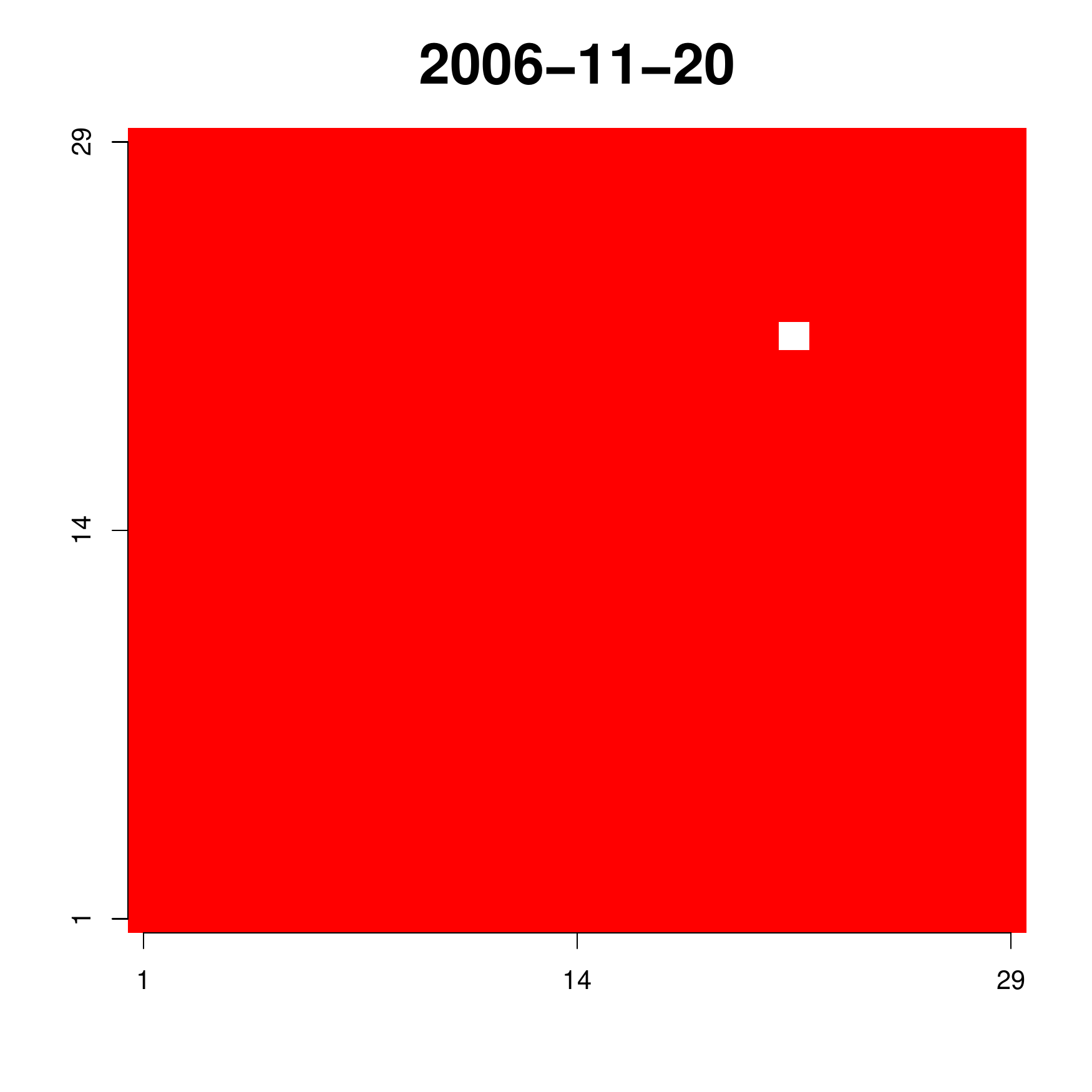} 
		\includegraphics[width=0.24\textwidth]{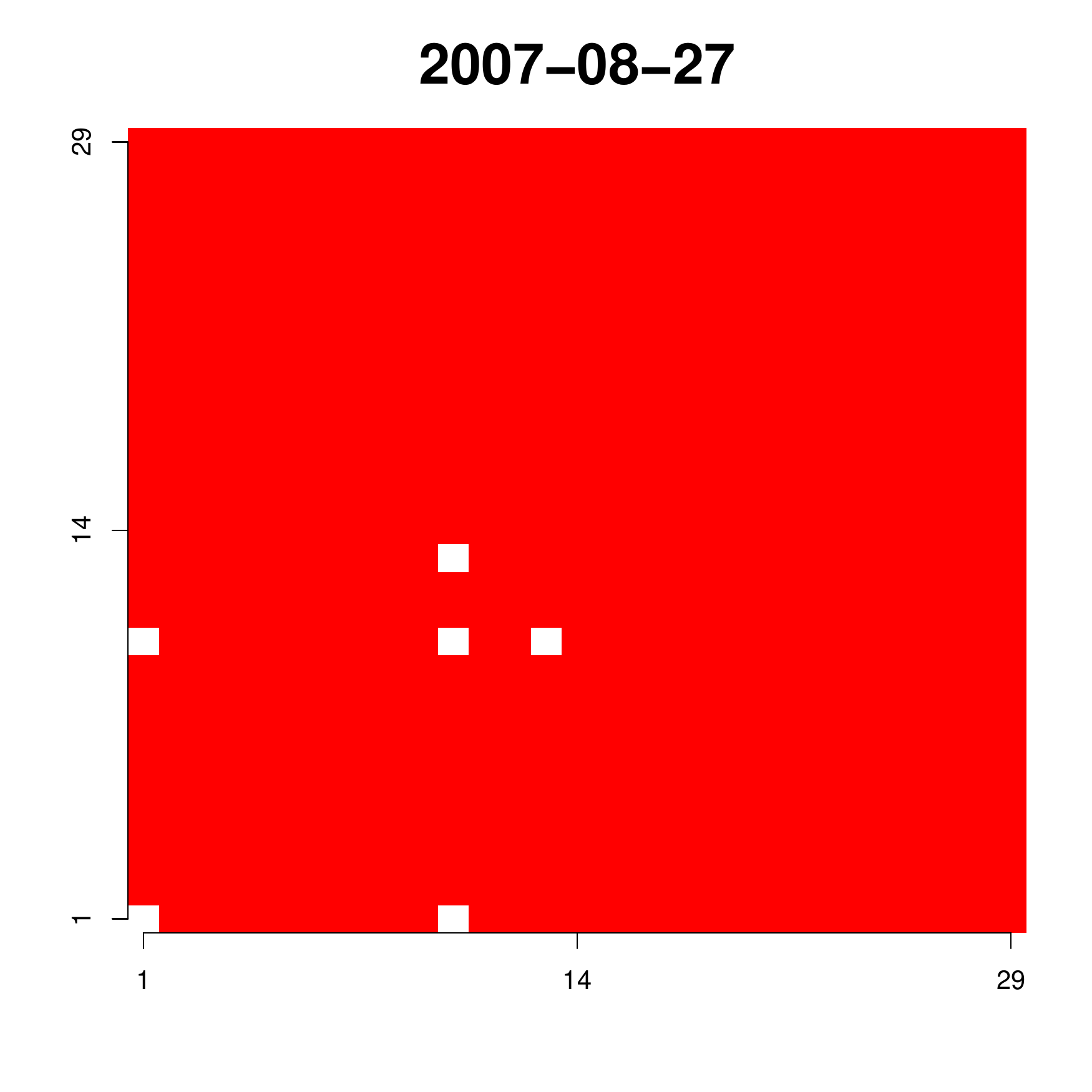} 	
		\includegraphics[width=0.24\textwidth]{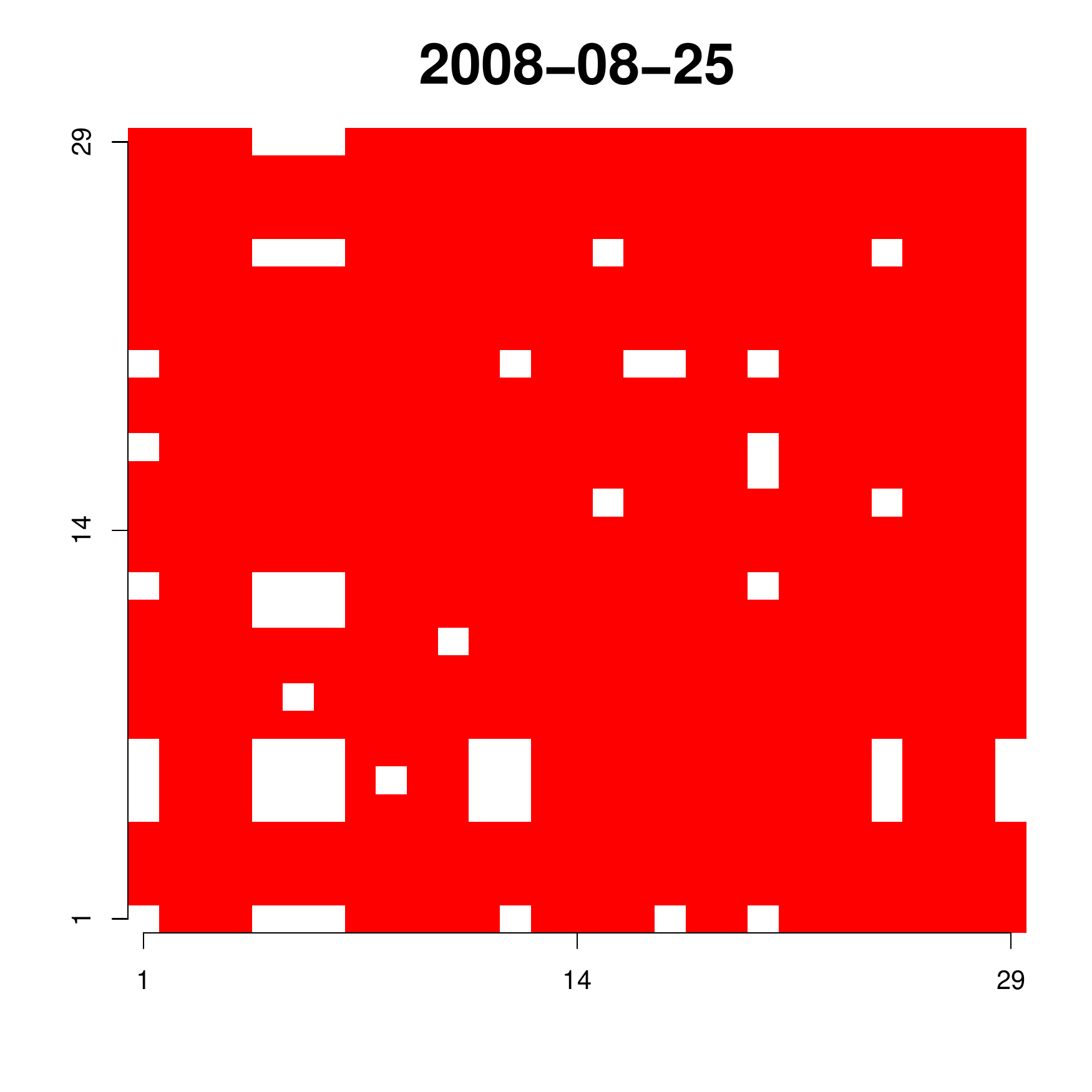} 
		\includegraphics[width=0.24\textwidth]{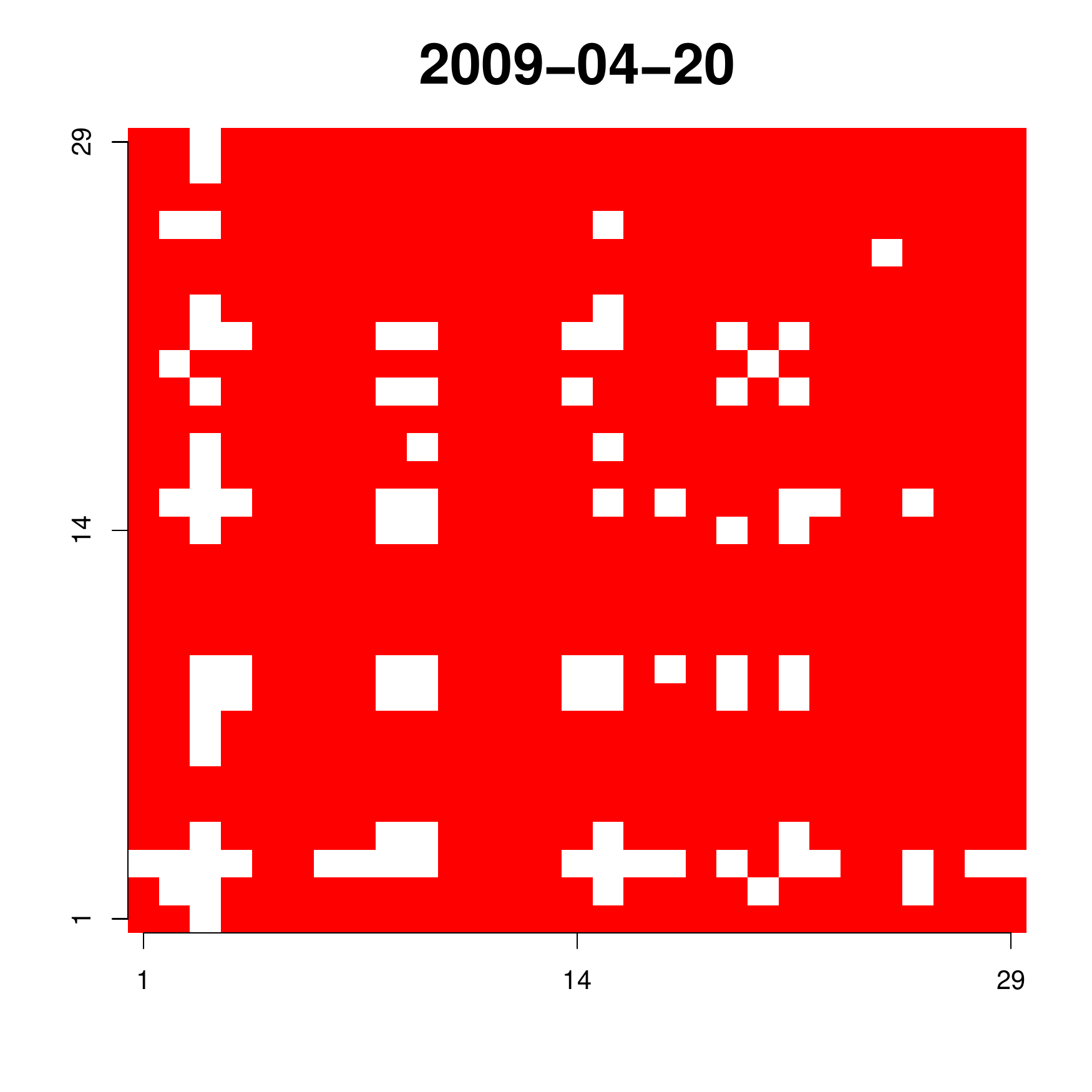} 
		\caption{\label{fig4} Stock market data set described in  Section \ref{sec:stock}.}
	\end{center}
\end{figure}

\subsection{MIT cellphone data}
\label{sec:cellphone_data}

In \Cref{sec-introduction}, we have studied the MIT cellphone data set.  In this subsection, we provide all numerical details.

Originally, the data are in the form of 1392 networks of size $96 \times 96$.  For each day of the experiment, the data include four networks corresponding to six hours interval in the given day.  We sum the four networks during each day resulting in 232 networks, each with 96 nodes.  The networks are transformed to binary networks by setting all strictly positive entries as $1$.  In other words, in each binary network, if the entry $(i,j)$ equals $1$, then it means that participants $i$ and $j$  were within physically close proximity during the corresponding day.

To evaluate the performances of different methods, we construct two experiments. In the first we use the data from 14-Sept-2004 to 1-Dec-2004 as the training data set, and the data from 2-Dec-2004 to 15-Feb-2005 as the test set.  The training set is the period before the MIT winter recess, which that year took place from 22-Dec-2004 to   3-Jan-2005.  The thresholds are chosen in the same way as \Cref{sec:stock}.  \Cref{alg-online-net} detects a change point on 27-Dec-2004;  ORI on 30-Jan-2005;  W and G on 29-Jan-2005.  Clearly, our method is the best at detecting the winter recess period.  

In our second example, we use the data from  1-Jan-2005 to 3-Mar-2005 as the training data set and the data from 4-Mar-2005 to 5-May-2005 as the test set.   Thus the  training data  consist of networks  before the  spring recess  which took place from  26-Mar-2005 to 3-Apr-2005.  \Cref{alg-online-net}  detects a change point on 31-Mar-2005; and  ORI on 6-Apr-2005.  Thus, our method seems to be quicker at detecting the spring recess. 

\section{Discussions}

In this paper, we are concerned with online change point detection in a sequence of inhomogeneous Bernoulli networks.  We established the minimax lower bound on the detection delay, which matches an upper bound, saving for logarithmic factors, based on an NP-hard estimation procedure.  In addition, we proposed a polynomial-time algorithm, the detection delay of which matches the that based on NP-hard estimators in the extreme cases, i.e.~$r \asymp 1$ and $r \asymp n$.  

Our proposed methods consist of two different Type-I error control strategies, with the worst case computational cost of order $O(\log(t) \mathrm{Cost}(n))$ when proceeding to the time point $t$, where $\mathrm{Cost}(n)$ is the computational cost for running the USVT algorithm (\Cref{alg-usvt}) on a size-$n$ network.

In this paper, we only discuss the at most one change point scenario.  In fact, it is straightforward to extend the algorithm and the results to multiple change points scenario.  To be specific, one can restart the algorithm whenever a change point is declared by \Cref{alg-online-net}.  As for the theoretical results, in \Cref{thm-net-delay-varying-N}, we can let $\Delta$ be the minimal spacing between two consecutive change points.  Provided that
	\[
		C_{\mathrm{SNR}} n r\log(\Delta/\alpha) > 2 C_d,
	\]  
	then with probability at least $1 - \alpha$, all change points can be detected with detection delay uniformly upper bounded by the same detection delay upper bound in \Cref{thm-net-delay-varying-N} and without false alarms.  This is a straightforward consequence of \Cref{thm-net-delay-varying-N}, therefore we omit the technical details here.
	
In the existing literature, it is hoped to have an online change point detection with constant cost proceeding to every time point.  In this sense, our cost $O(\log(t))$ is not efficient enough.  However, to the best of our knowledge, when the before and after distributions are not fully specified, this constant computational cost is not achievable even in the univariate case.  Having said this, it is still of vital interest to improve the computational efficiency of network online change point detection methods.  We will leave this for future work.

\appendix
\section*{Appendices}

All necessary lemmas are collected in \Cref{sec-lemmas} and all proofs of the main results are left in \Cref{sec-thms}.

\section{Network change point lemmas}\label{sec-lemmas}

This lemma below is identical to Lemma S.6 in \cite{wang2018optimal}, therefore we skip the proof here.

\begin{lemma}\label{lem-op-net-control}

\noindent {\bf (1)} For any $t \in \mathbb{N}^*$, let $\{A(l)\}_{l=1}^t$ be a collection of independent matrices with independent Bernoulli entries satisfying
	\[
		\max_{l = 1, \ldots, t}\|\mathbb{E}(A(t))\|_{\infty}\le \rho,
	\]	
	with $n\rho \geq \log(n)$.  Let $\{w_l\}_{l=1}^t \subset \mathbb{R}$ be a collection of scalars such that $\sum_{l=1}^t w_l^2 = 1$ and $\sum_{l=1}^t w_l =0$.  Then there exists an absolute constant $C > 32 \times 2^{1/4}e^2$ such that for any $\varepsilon > 0$,
		\begin{align}\label{eq:operator norm bounds} 
			\mathbb{P}\left( \left\| \sum_{l=1}^t w_l A(l) - \mathbb{E} \left( \sum_{l=1}^t w_l A(l)  \right)\right\|_{\mathrm{op}}  \ge C \sqrt{n\rho} + \varepsilon  \right)\le  \exp(-\varepsilon^2/2).
		\end{align}

\vskip 3mm
\noindent {\bf (2)} If $\{ A (l)\}_{l=1}^t$ are symmetric matrices, then \eqref{eq:operator norm bounds} still holds.
\end{lemma}

Lemmas~\ref{lemma:USVT}	and \ref{lemma:USVT 2} are from Lemma~1 in \cite{Xu2017}.
	
\begin{lemma}
	\label{lemma:USVT}
	Let $A, B\in \mathbb{R}^{n \times n}$ be two symmetric matrices with $\|A - B\|_{\mathrm{op}}< \tau/(1+\delta)$, $\tau >0$.  Then for a fixed $\delta < 1$, 
	we have 
	\[
		\| \mathrm{USVT} (A, \tau, \infty) - B\|_\mathrm{F}^2 \le 16\min_{s = 0}^n\left\{s\tau^2+ (1+\delta)^2\delta^{-2} \sum_{i = s+1}^n \lambda_i^2(B)\right\},
	\]
	where $\lambda_n(B) \geq \cdots \lambda_1(B)$ are the eigenvalues of $B$.
\end{lemma}

\begin{lemma}\label{lemma:USVT 2}
	Let $A$ and $B$ be defined as in \Cref{lemma:USVT}, and that $\|B\|_{\infty}\le \tau' $, then
		\[
			\| \mathrm{USVT} (A, \tau,  \tau') - B\|_\mathrm{F}^2 \le 16\min_{s = 0}^n \left\{s\tau^2+ (1+\delta)^2\delta^{-2} \sum_{i = s+1}^n \lambda_i^2(B)\right\},
		\]
	where $\lambda_n(B) \geq \cdots \lambda_1(B)$ are the eigenvalues of $B$		
\end{lemma}

This lemma below is Lemma S.2 in \cite{wang2018optimal}.

\begin{lemma} \label{lem-s2}
	Let $\{X(l)\}_{l=1, 2, \ldots} \in \mathbb{R}^p$ be a sequence of independent random vectors with independent Bernoulli entires. Suppose that $\mathbb{E}(X_i(t)) = \mu_i(t)$ and that 
		\[
			\sup_{l = 1, 2, \ldots} \left\|\mu(l)\right\|_{\infty} \le \rho.
		\] 
		For any $t > 1$, let $v \in \mathbb{R}^p$ and $\{w_l\}_{l=1}^t \subset \mathbb{R}$ satisfy $\sum_{l=1}^t w_l^2 =1$. Then for any $\varepsilon > 0$, we have
		\[
			\mathbb{P}\left (\left|\sum_{i=1}^p v_i \sum_{l=1}^t w_l (X_i(l) - \mu_i(l)) \right| \geq \varepsilon \right) \leq 2\exp\left(-\frac{3/2\varepsilon^2}{3\rho \| v\|_2^2  +\varepsilon  \max_{i=1}^p | v_i|  \max_{l=1}^t |w_l|}\right).
		\]
\end{lemma}

\begin{lemma}\label{lem-usvt-error-control}
Assume that $\{B(u)\}$ is a sequence of adjacency matrices satisfying \Cref{assump-1-net}.  For any integer $t \geq 2$, let
	\begin{align*}
		& \mathcal{S}(t) = \{t - s^{j-1}, \, j = 1, \ldots, \lfloor\log(t)/\log(2) \rfloor\}, \quad C > 32 \times 2^{1/4}e^2, \\
		& \varepsilon_{s, t} = \sqrt{2\log\left\{ \frac{t(t+1) \log(t)}{ \alpha \log(2)}\right\}}, \quad \tau_{1, s, t} = C\sqrt{n\rho} + \varepsilon_{s, t} \quad \mbox{and} \quad \tau_{2, s, t} = \sqrt{\frac{(t-s)s}{t}} \rho.
	\end{align*}
	We have that under \Cref{assump-2-net}, the event 
	\begin{align*}
		\mathcal{F}_1 = \Bigg\{\forall t \geq 2: \, & \max_{s \in \mathcal{S}(t)}\left\|\mathrm{USVT}(\widehat{B}_{s, t}, \tau_{1, s, t}, \tau_{2, s, t})\right\|_{\mathrm{F}}= 0, \, t \leq \Delta, \\
		\mbox{and} \quad & \max_{s \in \mathcal{S}(t)}\left\|\mathrm{USVT}(\widehat{B}_{s, t}, \tau_{1, s, t}, \tau_{2, s, t}) - \widehat{\Theta}_{s, t}\right\|_{\mathrm{F}} \leq \sqrt{r}\left(C\sqrt{n\rho} + \varepsilon_{s, t}\right), \, t > \Delta\Bigg\}
	\end{align*}
	holds with probability at least $1 - \alpha/2$; under \Cref{assump-2-0-net}, the event
	\[
		\mathcal{F}_1 = \left\{\forall t \geq 2: \, \max_{s \in \mathcal{S}(t)}\left\|\mathrm{USVT}(\widehat{B}_{s, t}, \tau_{1, s, t}, \tau_{2, s, t})\right\|_{\mathrm{F}} = 0\right\}
	\]
	holds with probability at least $1 - \alpha/2$.  
	
\end{lemma}

\begin{proof}
\noindent \textbf{Step 1}.   If \Cref{assump-2-net} holds, then for $t \leq \Delta$, it holds that
	\[
		\widehat{\Theta}_{s, t} = 0 \quad \mbox{and} \quad \mathrm{rank}(\widehat{\Theta}_{s, t}) = 0, \quad \forall s \in \mathcal{S}(t);
	\]	
	for $t > \Delta$, it holds that
	\begin{equation}\label{eq-widehat-theta-2-222}
		\widehat{\Theta}_{s, t} = \begin{cases}
 			(t-\Delta)\sqrt{\frac{s}{t(t-s)}} (\Theta_1 - \Theta_2), & s \leq \Delta,\\
	 		\Delta\sqrt{\frac{t-s}{st}}(\Theta_1 - \Theta_2), & s > \Delta.
 		\end{cases} \quad \mbox{and} \quad \mathrm{rank}(\widehat{\Theta}_{s, t}) \leq r.
	\end{equation}

If \Cref{assump-2-0-net} holds, then for any $t$, it holds that
	\[
		\widehat{\Theta}_{s, t} = 0 \quad \mbox{and} \quad \mathrm{rank}(\widehat{\Theta}_{s, t}) = 0, \quad \forall s \in \mathcal{S}(t).
	\]	

\medskip
\noindent \textbf{Step 2}.  Due to \Cref{def-cusum-net}, we have
	\[
		\widehat{B}_{s, t} = \sum_{l = 1}^t w^s_l B(l),
	\]
	where $\sum_{l = 1}^t w^s_l = 0$ and $\sum_{l = 1}^t (w^s_l)^2 = 1$.

Define 
	\[
		\mathcal{E}_1^c = \left\{\exists t \geq 2, \, s \in \mathcal{S}(t): \, \|\widehat{B}_{s, t} - \widehat{\Theta}_{s, t}\|_{\mathrm{op}} > C\sqrt{n\rho} + \varepsilon_{s, t} \right\}.
	\] 
	Then it follows from \Cref{lem-op-net-control} that, 
	\begin{align*}
		& \mathbb{P}(\mathcal{E}_1^c) \leq \sum_{t = 2}^{\infty} \frac{\log(t)}{\log(2)} \max_{s \in \mathcal{S}(t)}\mathbb{P} \left\{\|\widehat{B}_{s, t} - \widehat{\Theta}_{s, t}\|_{\mathrm{op}} > C\sqrt{n\rho} + \varepsilon_{s, t} \right\} \\
		\leq & \sum_{t = 2}^{\infty} \frac{\log(t)}{\log(2)} \frac{\log(2)}{\log(t)} \frac{\alpha}{t(t+1)} \leq \frac{\alpha}{2},
	\end{align*} 
	where 
	\begin{equation}\label{eq-c-epsst-2-222}
		C > 32 \times 2^{1/4}e^2 \quad \mbox{and} \quad \varepsilon_{s, t} = \sqrt{2\log\left\{ \frac{t(t+1) \log(t)}{ \alpha \log(2)}\right\}}.
	\end{equation}
	
Under \Cref{assump-2-net}, it follows from \Cref{lemma:USVT} that $\mathcal{E}_2 \subset \mathcal{E}_1$, where 
	\begin{align*}
		\mathcal{E}_2 = \Bigg\{\forall t \geq 2: \, & \max_{s \in \mathcal{S}(t)}\left\|\mathrm{USVT}(\widehat{B}_{s, t}, \tau_{1, s, t}, \infty) - \widehat{\Theta}_{s, t}\right\|_{\mathrm{F}} \leq \sqrt{r}\left(C\sqrt{n\rho} + \varepsilon_{s, t}\right), \, t > \Delta \\
		\mbox{and} \quad & \max_{s \in \mathcal{S}(t)}\left\|\mathrm{USVT}(\widehat{B}_{s, t}, \tau_{1, s, t}, \infty)\right\|_{\mathrm{F}} = 0, \, t \leq \Delta\Bigg\},
	\end{align*}
	with $C$ and $\varepsilon_{s, t}$ defined in \eqref{eq-c-epsst-2-222} and 
	\[
		\tau_{1, s, t} = C\sqrt{n\rho} + \varepsilon_{s, t}.
	\]
	Due to \eqref{eq-widehat-theta-2-222}, we have that 
	\[
		\left\|\widehat{\Theta}_{s, t}\right\|_{\infty} \leq \sqrt{\frac{s(t-s)}{t}} \rho = \tau_{2, s, t}.
	\]
	Therefore, we have that $\mathcal{F}_1 \subset \mathcal{E}_2$.  
	
Under \Cref{assump-2-0-net}, note that $\widehat{\Theta}_{s, t} = 0$ and $\mathrm{rank}(\widehat{\Theta}_{s, t}) = 0$.  Due to Lemmas~\ref{lemma:USVT} and \ref{lemma:USVT 2}, we have $\mathcal{F}_1 \subset \mathcal{E}_1$, which completes the proof.
\end{proof}

\begin{lemma}\label{lem-usvt-error-control-2}
Assume $\{B(u)\}$ is a sequence of adjacency matrices satisfying \Cref{assump-1-net}.  For any integer $t \geq 2$, let
	\begin{align*}
		& \mathcal{S}(t) = \{t - s^{j-1}, \, j = 1, \ldots, \lfloor\log(t)/\log(2) \rfloor\}, \quad C > 32 \times 2^{1/4}e^2, \\
		& \varepsilon_{s, t} = \sqrt{2\log\left\{\frac{2(\gamma+1)^2 \log(\gamma+1)}{\log(2)}\right\}}, \quad \tau_{1, s, t} = C\sqrt{n\rho} + \varepsilon_t \quad \mbox{and} \quad \tau_{2, s, t} = \sqrt{\frac{(t-s)s}{t}} \rho.
	\end{align*}
	We have that under \Cref{assump-2-net}, the event 
	\begin{align*}
		\mathcal{F}_2 = \Bigg\{\forall t \in \{2, \ldots, \gamma+1\}: \, & \max_{s \in \mathcal{S}(t)}\left\|\mathrm{USVT}(\widehat{B}_{s, t}, \tau_{1, s, t}, \tau_{2, s, t})\right\|_{\mathrm{F}} = 0, \, t \leq \Delta \\
		\mbox{and} \quad & \max_{s \in \mathcal{S}(t)}\left\|\mathrm{USVT}(\widehat{B}_{s, t}, \tau_{1, s, t}, \tau_{2, s, t}) - \widehat{\Theta}_{s, t}\right\|_{\mathrm{F}} \leq \sqrt{r}\left(C\sqrt{n\rho} + \varepsilon_{s, t}\right), \, t > \Delta\Bigg\}
	\end{align*}
	holds with probability at least $1 - (\gamma+1)^{-1}/2$; under \Cref{assump-2-net}, the event 
	\[
		\mathcal{F}_2 = \left\{\forall t \in \{2, \ldots, \gamma+1\}: \, \max_{s \in \mathcal{S}(t)}\left\|\mathrm{USVT}(\widehat{B}_{s, t}, \tau_{1, s, t}, \tau_{2, s, t})\right\|_{\mathrm{F}} = 0\right\}
	\]
	holds with probability at least $1 - (\gamma+1)^{-1}/2$.	
\end{lemma}

\begin{proof}
\noindent \textbf{Step 1}.   If \Cref{assump-2-net} holds, then for $t \leq \Delta$, it holds that
	\[
		\widehat{\Theta}_{s, t} = 0 \quad \mbox{and} \quad \mathrm{rank}(\widehat{\Theta}_{s, t}) = 0,  \quad \forall s \in \mathcal{S}(t);
	\]	
	for $t > \Delta$, it holds that
	\begin{equation}\label{eq-widehat-theta-2}
		\widehat{\Theta}_{s, t} = \begin{cases}
 			(t-\Delta)\sqrt{\frac{s}{t(t-s)}} (\Theta_1 - \Theta_2), & s \leq \Delta,\\
	 		\Delta\sqrt{\frac{t-s}{st}}(\Theta_1 - \Theta_2), & s > \Delta, 
 		\end{cases}  \quad \mbox{and} \quad \mathrm{rank}(\widehat{\Theta}_{s, t}) \leq r.
	\end{equation}
	In addition, it holds that $\mathrm{rank}(\widehat{\Theta}_{s, t}) \leq r$.
	
If \Cref{assump-2-0-net} holds, then for any $t$, it holds that
	\[
		\widehat{\Theta}_{s, t} = 0 \quad \mbox{and} \quad \mathrm{rank}(\widehat{\Theta}_{s, t}) = 0,  \quad \forall s \in \mathcal{S}(t).
	\]	
 
\medskip
\noindent \textbf{Step 2}.  Due to \Cref{def-cusum-net}, we have
	\[
		\widehat{B}_{s, t} = \sum_{l = 1}^t w^s_l B(l),
	\]
	where $\sum_{l = 1}^t w^s_l = 0$ and $\sum_{l = 1}^t (w^s_l)^2 = 1$.

Define 
	\[
		\mathcal{E}_3^c = \left\{\exists t \in \{2, \ldots, \gamma + 1\}, \, s \in \mathcal{S}(t): \, \|\widehat{B}_{s, t} - \widehat{\Theta}_{s, t}\|_{\mathrm{op}} > C\sqrt{n\rho} + \varepsilon_{s, t} \right\}.
	\] 
	Then it follows from \Cref{lem-op-net-control} that, 
	\begin{align*}
		& \mathbb{P}(\mathcal{E}_3^c) < \frac{1}{2(\gamma+1)},
	\end{align*} 
	where 
	\begin{equation}\label{eq-c-epsst-2}
		C > 32 \times 2^{1/4}e^2 \quad \mbox{and} \quad \varepsilon_{s, t} = \sqrt{2\log\left\{\frac{2(\gamma+1)\gamma \log(\gamma+1)}{\log(2)}\right\}}.
	\end{equation}

Under \Cref{assump-2-net}, it follows from \Cref{lemma:USVT} that $\mathcal{E}_4 \subset \mathcal{E}_3$, where 
	\begin{align*}
		\mathcal{E}_4 = \Bigg\{\forall t \in \{2, \ldots, \gamma+1\}: \, & \max_{s \in \mathcal{S}(t)}\left\|\mathrm{USVT}(\widehat{B}_{s, t}, \tau_{1, s, t}, \infty) - \widehat{\Theta}_{s, t}\right\|_{\mathrm{F}} \leq \sqrt{r}\left(C\sqrt{n\rho} + \varepsilon_{s, t}\right), \, t > \Delta \\
		\mbox{and} \quad & \max_{s \in \mathcal{S}(t)}\left\|\mathrm{USVT}(\widehat{B}_{s, t}, \tau_{1, s, t}, \infty)\right\|_{\mathrm{F}} = 0, \, t \leq \Delta\Bigg\},
	\end{align*}
	with $C$ and $\varepsilon_{s, t}$ defined in \eqref{eq-c-epsst-2} and 
	\[
		\tau_{1, s, t} = C\sqrt{n\rho} + \varepsilon_{s, t}.
	\]
	Due to \eqref{eq-widehat-theta-2}, we have that 
	\[
		\left\|\widehat{\Theta}_{s, t}\right\|_{\infty} \leq \sqrt{\frac{s(t-s)}{t}} \rho = \tau_{2, s, t}.
	\]
	Therefore, we have $\mathcal{F}_2 \subset \mathcal{E}_4$. 
	
Under \Cref{assump-2-0-net}, not that $\widehat{\Theta}_{s, t} = 0$ and $\mathrm{rank}(\widehat{\Theta}) = 0$.  Due to Lemmas~\ref{lemma:USVT} and \ref{lemma:USVT 2}, we have $\mathcal{F}_2 \subset \mathcal{E}_3$, which completes the proof.
\end{proof}

\section{Proofs of main results}\label{sec-thms}

\begin{proof}[Proof of \Cref{thm-net-delay-varying-N}]
We let
	\[
		\widetilde{B}_{s, t} = \mathrm{USVT}(\widehat{B}_{s, t}, \tau_{1, s, t}, \tau_{2, s, t}), \quad t \geq 2
	\]
	and
	\[
		\mathcal{S}(t) = \{t - 2^{j-1}, \, j = 1, \ldots, \lfloor\log(t)/\log(2) \rfloor\}.
	\]
	The rest of the proof is conducted on the event $\mathcal{F}_1$, defined in \Cref{lem-usvt-error-control}.  In particular, 
	\begin{equation}\label{eq-f1-ddddd}
	\mathbb{P}\{\mathcal{F}_1\} > 1-\alpha/2
	\end{equation} 
	and the event $\mathcal{F}_1$ is regarding the data $\{B(t)\}$, which are independent of the data $\{A(t)\}$.

\medskip
\noindent \textbf{Step 1}.  	 Due to \Cref{lem-usvt-error-control}, it holds that if \Cref{assump-2-net} holds and $t \leq \Delta$, or \Cref{assump-2-0-net} holds, then $\widetilde{B}_{s, t} = 0$, $s \in \mathcal{S}(t)$.  Therefore the claim (i) is proved and for the claim (ii), we have 
	\[
		\widehat{t} - \Delta > 0.
	\]

Define
	\begin{equation}\label{eq-t-1-proof}
		t_1 = \min\left\{t > \Delta: \max_{s \in \mathcal{S}(t)}\left[\left|\left(\widehat{A}_{s, t}, \frac{\widetilde{B}_{s, t}}{\left\|\widetilde{B}_{s, t}\right\|_{\mathrm{F}}}\right)\right| \mathbbm{1} \left\{\|\widetilde{B}_{s, t}\|_{\mathrm{F}} > C\log^{1/2}(t/\alpha)\right\}\right] > b_t\right\}.
	\end{equation}
	Due to the design of \Cref{alg-online-net}, we can see that $\widehat{t} \leq t_1$ and therefore $d \leq t_1 - \Delta$.  In order to provide an upper bound on $d$, it thus suffices to upper bound $t_1$.

\medskip	
\noindent \textbf{Step 2}.  
Recall the quantity 
	\[
		\varepsilon_{s, t} = \sqrt{2\log\left\{\frac{t(t+1)\log(t)}{\alpha\log(2)}\right\}}
	\]
	defined in \Cref{lem-usvt-error-control}.  With the quantity $\varepsilon_{s, t}$, we define
	\begin{align*}
		t_2 & = \min\Bigg\{t > \Delta: \, \max_{s \in \mathcal{S}(t)}\Bigg[\left|\left(\widehat{A}_{s, t}, \frac{\widetilde{B}_{s, t}}{\left\|\widetilde{B}_{s, t}\right\|_{\mathrm{F}}}\right)\right| \\
		& \hspace{1cm} \times \mathbbm{1} \left\{\|\widehat{\Theta}_{s, t}\|_{\mathrm{F}} > C\log^{1/2}(t/\alpha) +  \sqrt{r}\left(C\sqrt{n\rho} + \varepsilon_{s, t}\right)\right\}\Bigg] > b_t\Bigg\}.
	\end{align*}
	Due to \Cref{lem-usvt-error-control}, we know that if
	\[
		\left\{\|\widehat{\Theta}_{s, t}\|_{\mathrm{F}} > C\log^{1/2}(t/\alpha) +  \sqrt{r}\left(C\sqrt{n\rho} + \varepsilon_{s, t}\right)\right\}
	\]
	considered in $t_2$ holds, then
	\[
		\left\{\|\widetilde{B}_{s, t}\|_{\mathrm{F}} > C\log^{1/2}(t/\alpha)\right\}
	\]
	considered in $t_1$ also holds.  This implies that $t_2 > t_1$.  It now suffices to find $t_2$, which yields an upper bound on $d$ that $d \leq t_2 - \Delta$.  
	
Due to the choices of $s$, we in turn define
	\begin{align*}
		J & = \min\Bigg\{j \in \mathbb{N}: \, \left|\left(\widehat{A}_{\Delta, \Delta + 2^j}, \frac{\widetilde{B}_{\Delta, \Delta + 2^j}}{\left\|\widetilde{B}_{\Delta, \Delta + 2^j}\right\|_{\mathrm{F}}}\right)\right| \\
		& \hspace{1cm} \times \mathbbm{1} \left\{\|\widehat{\Theta}_{\Delta, \Delta + 2^j}\|_{\mathrm{F}} > C\log^{1/2}((\Delta + 2^j)/\alpha) +  \sqrt{r}\left(C\sqrt{n\rho} + \varepsilon_{\Delta, \Delta + 2^j}\right)\right\} > b_{\Delta + 2^j}\Bigg\}
	\end{align*}
	and $t_3 = \Delta + 2^J$.  In the definition of $J$, we essentially choose the integer pair $(s, t)$ to be $(\Delta, \Delta + 2^J)$.  This is to ensure that $s \in \mathcal{S}(t)$ and $s = \Delta$.  Due to this construction, we can see that $t_3$ is an upper bound of $t_2$ and our task is now to find $J$ defined above.

\medskip	
\noindent \textbf{Step 3}.  We are now to show that, with a large enough absolute constant $C_d > 0$, 
	\begin{equation}\label{eq-J-def}
		J = \Bigg\lceil \log\left(\frac{C_d r\log(\Delta/\alpha)}{\kappa_0^2 n \rho}\right)/\log(2) \Bigg\rceil.
	\end{equation}
	For notational simplicity, in the rest of the proof, we assume that 
	\[
		\log\left(\frac{C_d r\log(\Delta/\alpha)}{\kappa_0^2 n \rho}\right)/\log(2)
	\]
	is a positive integer.  If this is violated, then the proof only needs to be modified by keeping the ceiling operator throughout.
	
\medskip	
\noindent \textbf{Step 3.1}.	  With $J$ defined in \eqref{eq-J-def}, we have that
	\begin{align*}
		\|\widehat{\Theta}_{\Delta, \Delta + 2^J}\|_{\mathrm{F}} = \kappa \sqrt{\frac{\Delta \frac{C_d r\log(\Delta/\alpha)}{\kappa_0^2 n \rho}}{\Delta + \frac{C_d r\log(\Delta/\alpha)}{\kappa_0^2 n \rho}}},
	\end{align*}
	which can be derived by plugging in $\Delta$ and $\Delta+2^J$ into \eqref{eq-widehat-theta-2}.  Due to \eqref{eq-snr}, we have that
	\begin{equation}\label{eq-b-f-lower-bound}
		\|\widehat{\Theta}_{\Delta, \Delta + 2^J}\|_{\mathrm{F}} > C\log^{1/2}((\Delta + 2^J)/\alpha) +  \sqrt{r}\left(C\sqrt{n\rho} + \varepsilon_{s, \Delta + 2^J}\right).
	\end{equation}
	This can be seen in the following three steps.
	
\medskip	
\noindent \textbf{Step 3.1.1}.  We first show that 
	\begin{equation}\label{eq-3.1.1}
		3^{-1}\|\widehat{\Theta}_{\Delta, \Delta + 2^J}\|_{\mathrm{F}} > C\log^{1/2}((\Delta + 2^J)/\alpha).
	\end{equation}
	Provided that $C_d < C_{\mathrm{SNR}}$, due to \eqref{eq-snr}, it holds that
	\begin{equation}\label{eq-31111}
		\|\widehat{\Theta}_{\Delta, \Delta + 2^J}\|_{\mathrm{F}} \geq \kappa \sqrt{\Delta \frac{C_d r\log(\Delta/\alpha)}{2\kappa_0^2 n \rho}} = \sqrt{\Delta n\rho \frac{C_d r\log(\Delta/\alpha)}{2}}. 
	\end{equation}
	In addition, provided that $\Delta/\alpha \geq 2$, it holds that
	\begin{equation}\label{eq-31111111}
		\log^{1/2}((\Delta + 2^J)/\alpha) \leq \log^{1/2}(2\Delta/\alpha) \leq \sqrt{2\log(\Delta/\alpha)}.
	\end{equation}
	Therefore, provided that $C_d > 36C^2/\log(2)$ and $n \geq 2$, \eqref{eq-3.1.1} holds, where $n\rho \geq \log(n)$ assumed in \Cref{assump-1-net} is used.
		
\medskip	
\noindent \textbf{Step 3.1.2}.  Provided that $C_d > 18C^2/\log(2)$, we have that $3^{-1}\|\widehat{\Theta}_{\Delta, \Delta + 2^J}\|_{\mathrm{F}} > C\sqrt{rn\rho}$, by using~\eqref{eq-31111}.

\medskip	
\noindent \textbf{Step 3.1.3}.  Lastly, we are to show 
	\begin{equation}\label{eq-3.1.3}
		3^{-1}\|\widehat{\Theta}_{\Delta, \Delta + 2^J}\|_{\mathrm{F}} > C\sqrt{r}\varepsilon_{s, \Delta + 2^J} = C\sqrt{2r\log\left\{\frac{(\Delta + 2^J)(\Delta + 2^J+1)\log(\Delta + 2^J)}{\alpha\log(2)}\right\}}.
	\end{equation}
	Due to \eqref{eq-31111111}, the last term in \eqref{eq-3.1.3} is upper bounded by 
	\begin{align*}
		& C\sqrt{2r\log\left\{\frac{2\Delta(2\Delta + 1)\log(2\Delta)}{\alpha\log(2)}\right\}} \leq C\sqrt{2r\log\left\{\frac{(2\Delta + 1)^3}{\alpha\log(2)}\right\}} \\
		\leq & C\sqrt{8r\log\left\{\frac{2\Delta}{\alpha}\right\}}  \leq C\sqrt{16r\log\left\{\frac{\Delta}{\alpha}\right\}}.
	\end{align*}
	Therefore provided that $C_d > 288C^2/\log(2)$, \eqref{eq-3.1.3} holds.

\medskip	
\noindent \textbf{Step 3.2}.  In addition, we have that 
	\begin{align}
		& \left|\left(\widehat{A}_{\Delta, \Delta + 2^J}, \frac{\widetilde{B}_{\Delta, \Delta + 2^J}}{\left\|\widetilde{B}_{\Delta, \Delta + 2^J}\right\|_{\mathrm{F}}}\right)\right| \nonumber \\
		\geq & \left|\left(\widehat{\Theta}_{\Delta, \Delta + 2^J}, \frac{\widetilde{B}_{\Delta, \Delta + 2^J}}{\left\|\widetilde{B}_{\Delta, \Delta + 2^J}\right\|_{\mathrm{F}}}\right)\right| - \left|\left(\widehat{A}_{\Delta, \Delta + 2^J} - \widehat{\Theta}_{\Delta, \Delta + 2^J}, \frac{\widetilde{B}_{\Delta, \Delta + 2^J}}{\left\|\widetilde{B}_{\Delta, \Delta + 2^J}\right\|_{\mathrm{F}}}\right)\right| = (I) - (II). \label{eq-i-ii-proof}
	\end{align}

\medskip	
\noindent \textbf{Step 3.2.1}. 
As for $(I)$, we have that
	\begin{align}
		& \left|\left(\widehat{\Theta}_{\Delta, \Delta + 2^J}, \frac{\widetilde{B}_{\Delta, \Delta + 2^J}}{\left\|\widetilde{B}_{\Delta, \Delta + 2^J}\right\|_{\mathrm{F}}}\right)\right| = \|\widehat{\Theta}_{\Delta, \Delta + 2^J}\|_{\mathrm{F}} \left|\left(\frac{\widehat{\Theta}_{\Delta, \Delta + 2^J}}{\|\widehat{\Theta}_{\Delta, \Delta + 2^J}\|_{\mathrm{F}}}, \frac{\widetilde{B}_{\Delta, \Delta + 2^J}}{\left\|\widetilde{B}_{\Delta, \Delta + 2^J}\right\|_{\mathrm{F}}}\right)\right| \nonumber \\
		= & \frac{\|\widehat{\Theta}_{\Delta, \Delta + 2^J}\|_{\mathrm{F}}}{2} \left(2 - \left\|\frac{\widehat{\Theta}_{\Delta, \Delta + 2^J}}{\|\widehat{\Theta}_{\Delta, \Delta + 2^J}\|_{\mathrm{F}}} - \frac{\widetilde{B}_{\Delta, \Delta + 2^J}}{\left\|\widetilde{B}_{\Delta, \Delta + 2^J}\right\|_{\mathrm{F}}}\right\|_{\mathrm{F}}^2\right) \nonumber \\ 
		= & \frac{\|\widehat{\Theta}_{\Delta, \Delta + 2^J}\|_{\mathrm{F}}}{2} \left(2 - \left\|\frac{\widehat{\Theta}_{\Delta, \Delta + 2^J}\left\|\widetilde{B}_{\Delta, \Delta + 2^J}\right\|_{\mathrm{F}} - \widetilde{B}_{\Delta, \Delta + 2^J}\|\widehat{\Theta}_{\Delta, \Delta + 2^J}\|_{\mathrm{F}}}{\|\widehat{\Theta}_{\Delta, \Delta + 2^J}\|_{\mathrm{F}}\left\|\widetilde{B}_{\Delta, \Delta + 2^J}\right\|_{\mathrm{F}}  }\right\|_{\mathrm{F}}^2\right) \nonumber \\		
		\geq & \frac{\|\widehat{\Theta}_{\Delta, \Delta + 2^J}\|_{\mathrm{F}}}{2} \left(2 - \left\|  \frac{\left\|\widehat{\Theta}_{\Delta, \Delta + 2^J} - \widetilde{B}_{\Delta, \Delta + 2^J} \right\|_{\mathrm{F}}}{\|\widehat{\Theta}_{\Delta, \Delta + 2^J}\|_{\mathrm{F}}}  +   \frac{\left|\left\|\widehat{\Theta}_{\Delta, \Delta + 2^J}\right\|_{\mathrm{F}} - \left\|\widetilde{B}_{\Delta, \Delta + 2^J} \right\|_{\mathrm{F}}\right|}{\|\widehat{\Theta}_{\Delta, \Delta + 2^J}\|_{\mathrm{F}}} \right\|_{\mathrm{F}}^2\right) \nonumber \\	
		\geq & \frac{\|\widehat{\Theta}_{\Delta, \Delta + 2^J}\|_{\mathrm{F}}}{2}\left\{2 - 4 \left(\frac{\|\widehat{\Theta}_{\Delta, \Delta + 2^J} - \widetilde{B}_{\Delta, \Delta + 2^J}\|_{\mathrm{F}}}{\|\widehat{\Theta}_{\Delta, \Delta + 2^J}\|_{\mathrm{F}}}\right)^2\right\} \nonumber \\
		\geq & \frac{\|\widehat{\Theta}_{\Delta, \Delta + 2^J}\|_{\mathrm{F}}}{2}\left\{2 - 4\left(\frac{C\sqrt{rn\rho} + C\sqrt{r}\varepsilon_{\Delta, \Delta + 2^J}}{\|\widehat{\Theta}_{\Delta, \Delta + 2^J}\|_{\mathrm{F}}}\right)^2\right\} \geq \frac{\|\widehat{\Theta}_{\Delta, \Delta + 2^J}\|_{\mathrm{F}}}{2}, \label{eq-i-bound-proof}
	\end{align}	
	where the third inequality is due to the event $\mathcal{F}_1$ and the last inequality follows from \eqref{eq-b-f-lower-bound} with a sufficiently large $C_{\mathrm{SNR}}$.
	
\medskip	
\noindent \textbf{Step 3.2.2}. 
As for (II), due to the independence between $\{A(t)\}$ and $\{B(t)\}$, it follows from \Cref{lem-s2} and~\eqref{eq-b-f-lower-bound} that 
	\begin{align}
		& \mathbb{P}_A\left\{\left|\left(\widehat{A}_{\Delta, \Delta + 2^J} - \widehat{\Theta}_{\Delta, \Delta + 2^J}, \frac{\widetilde{B}_{\Delta, \Delta + 2^J}}{\left\|\widetilde{B}_{\Delta, \Delta + 2^J}\right\|_{\mathrm{F}}}\right)\right| \geq b_{\Delta + 2^J} \right\} \nonumber \\
		\leq &  \exp\left\{-\frac{3/2 b_t^2}{3 \rho + b_{\Delta + 2^J}\rho C^{-1}\log^{-1/2}((\Delta + 2^J)/\alpha)}\right\} < \frac{\alpha}{2}. \label{eq-ii-term-bound-proof}
	\end{align}
	To be specific, the CUSUM weights are regarded as the $\{w_l\}$ sequence in \Cref{lem-s2} and all the entries in $\widetilde{B}_{\Delta, \Delta + 2^J}\left\|\widetilde{B}_{\Delta, \Delta + 2^J}\right\|_{\mathrm{F}}^{-1}$ are regarded as the $\{v_i\}$ sequence in \Cref{lem-s2}.  Therefore, $\|v\|_2 = 1$, $\max_{l = 1}^t |w_l| \leq 2^{-J/2}$ and
	\[
		\max_{i=1}^p|v_i| \leq \frac{\rho\sqrt{\frac{2^J\Delta}{\Delta + 2^J}}}{C\log^{1/2}((\Delta + 2^J)/\alpha)} \leq \frac{\rho 2^{J/2}}{C\log^{1/2}((\Delta + 2^J)/\alpha)},
	\]
	where the last inequality follows from \eqref{eq-b-f-lower-bound} and the definition of $\mathcal{F}_1$.
	
Due to 	\eqref{eq-b-f-lower-bound}, with a sufficiently large $C_{\mathrm{SNR}}$, it holds that 
	\[
		\frac{\|\widehat{\Theta}_{\Delta, \Delta + 2^J}\|_{\mathrm{F}}}{2} > 2b_{\Delta + 2^J}.
	\]
	Then, combining \eqref{eq-i-ii-proof}, \eqref{eq-i-bound-proof} and \eqref{eq-ii-term-bound-proof}, it holds that
	\begin{align}
		\mathbb{P}_A\left\{\left|\left(\widehat{A}_{\Delta, \Delta + 2^J}, \frac{\widetilde{B}_{\Delta, \Delta + 2^J}}{\left\|\widetilde{B}_{\Delta, \Delta + 2^J}\right\|_{\mathrm{F}}}\right)\right| \geq b_{\Delta + 2^J} \right\} > 1 - \frac{\alpha}{2}. \label{eq-ii-term-bound-proof-2}
	\end{align}
\medskip	
\noindent \textbf{Step 3.3}. Combining \eqref{eq-f1-ddddd} and \eqref{eq-ii-term-bound-proof-2}, we have that 
	\[
		\mathbb{P}\left\{\left|\left(\widehat{A}_{\Delta, \Delta + 2^J}, \frac{\widetilde{B}_{\Delta, \Delta + 2^J}}{\left\|\widetilde{B}_{\Delta, \Delta + 2^J}\right\|_{\mathrm{F}}}\right)\right| \mathbbm{1} \left\{\|\widetilde{B}_{\Delta, \Delta + 2^J}\|_{\mathrm{F}} > C\log^{1/2}((\Delta + 2^J)/\alpha)\right\} > b_{\Delta + 2^J}\right\} > 1 - \alpha,
	\]
	which completes the proof.	

\end{proof}

\begin{proof}[Proof of \Cref{thm-selection}]
We let
	\[
		\widetilde{B}_{s, t} = \mathrm{USVT}(\widehat{B}_{s, t}, \tau_{1, s, t}, \tau_{2, s, t}), \quad t > N.
	\]
	Let 
	\[
		\mathcal{S}(t) = \{t - 2^{j-1}, \, j = 1, \ldots, \lfloor\log(t)/\log(2) \rfloor\}.
	\]

\medskip	
\noindent \textbf{Step 1}.   Due to \Cref{lem-usvt-error-control-2}, it holds that if \Cref{assump-2-0-net} holds, then with probability at least $1 - (\gamma+1)^{-1}$, the event $\mathcal{F}_2$ holds, i.e.~$\widetilde{B}_{s, t} = 0$, for all $t \in \{2, \ldots, \gamma + 1\}$ and $s \in \mathcal{S}(t)$  Then we have
	\begin{align*}
		& \mathbb{E}_{\infty}(\widehat{t}) = \sum_{t = 1}^{\infty}\mathbb{P}(\widehat{t} \geq t) \geq \sum_{t = 1}^{\gamma + 1} \mathbb{P}(\widehat{t} \geq t) \geq (\gamma+1) \mathbb{P}(\widehat{t} \geq \gamma + 1) \geq (\gamma+1) \left(1 - \frac{1}{\gamma+1}\right) = \gamma.
	\end{align*}
	The claim (i) is proved.
	
\medskip	
\noindent \textbf{Step 2}.	As for the claim (ii), recall the event $\mathcal{F}_2$ and associated quantities defined in \Cref{lem-usvt-error-control-2}.  We have that $\mathbb{P}(\mathcal{F}_2) > 1 - 1/\{2(\gamma+1)\}$.  Conditional on the event $\mathcal{F}_2$ instead of $\mathcal{F}_1$, the rest of the proof is identical to that of \Cref{thm-net-delay-varying-N}. 
\end{proof}

\begin{proof}[Proof of \Cref{thm-net-delay-np}]
The proof is almost identical to the proof of \Cref{thm-net-delay-varying-N}, except that the large probability events where USVT estimators are well controlled in the proof of \Cref{thm-net-delay-varying-N} are replaced by Theorem~2.1 in \cite{gao2015rate}.  In fact, Theorem~2.1 in \cite{gao2015rate} is stated and proved by assuming $\rho = 1$.  In order to get an upper bound being a function of $\rho$, we only need to change Lemmas~4.1-4.3 in \cite{gao2015rate} correspondingly.
\end{proof}

\begin{proof}[Proof of \Cref{thm-lb}]

This proof consists of two different cases: a) $r \lesssim \sqrt{n}$ and b) $r \gtrsim \sqrt{n}$.

\medskip
\noindent \textbf{Case 1: $r \lesssim \sqrt{n}$.}

\noindent \textbf{Step 1 - Setup.} We assume the networks are generated as follows.  Prior to the change point, if there exists any, the adjacency matrices are generated independently from the distribution $P_0$, which has the graphon matrix
	\[
		\Theta_1 = (\rho/2)_{i, j = 1}^n.
	\]
	If there exists a change point, then the adjacency matrices after the change point are generated independently from the distribution 
	\[
		P_1 = \frac{1}{2^n} \sum_{u \in \{\pm 1\}^n} P_{1, u},
	\]
	where the graphon of the distribution $P_{1, u}$ is $\rho/2 \mathbbm{1}\mathbbm{1}^{\top} + \kappa_0 \rho u u^{\top}$, $u \in \{\pm 1\}^n$.

For any $M \in \mathbb{N}$, let $P^M$ be the restriction of a distribution $P$ on $\mathcal{F}_M$, i.e.~the $\sigma$-filed generated by the observations $\{A(t)\}_{i = 1}^M$.  For notational simplicity, in this proof, the adjacency matrices $A(t)$'s will be denoted as $A^t$'s.  For any $\nu \geq 1$ and $M \geq \nu$, we have that for any $M \geq \Delta$, let
	\[
		Z_{\nu, M} = \log\left(\frac{P_{\kappa_0, \nu}^M}{P_{\kappa_0, \infty}^M}\right),
	\]	
	where $P_{\kappa, \infty}$ indicates the distribution under which there is no change point.

\medskip
\noindent \textbf{Step 2 - When $Z_{\nu, T}$ is upper bounded.}  For any $\nu \geq 1$, define the event
	\[
		\mathcal{E}_{\nu} = \left\{\nu < T < \nu + \frac{\log(1/\alpha)}{8\kappa^2_0 n \rho}, \quad Z_{\nu, T} < \frac{3}{4}\log(1/\alpha)\right\}.
	\]	
	Then we have
	\begin{align}
		\mathbb{P}_{\kappa, \nu}(\mathcal{E}_{\nu}) = \frac{P_{\kappa_0, \nu}}{P_{\kappa, \infty}}(\mathcal{E}_{\nu}) P_{\kappa, \infty}(\mathcal{E}_{\nu}) \leq \alpha^{-3/4}\alpha = \alpha^{1/4},\label{eq-lb-pf-v1-111}
	\end{align}
	where the inequality follows from the definition of $\mathcal{D}$ and $\mathcal{E}_{\nu}$.

\medskip
\noindent \textbf{Step 3 - When $Z_{\nu, T}$ is lower bounded.} 
For any $\nu \geq 1$ and $T \in \mathcal{D}$, since $\{T \geq \nu\} \in \mathcal{F}_{\nu-1}$, we have that
	\begin{align}
		& \mathbb{P}_{\kappa, \nu}\left\{\nu < T < \nu + \frac{\log(1/\alpha)}{8\kappa^2_0 n \rho}, \quad Z_{\nu, T} \geq \frac{3}{4}\log(1/\alpha) \Bigg | T \geq \nu\right\}	\nonumber \\
		\leq & \esssup \mathbb{P}_{\kappa, \nu} \left\{\max_{1 \leq l \leq \frac{\log(1/\alpha)}{8\kappa^2_0 n \rho}} Z_{\nu, \nu + l} \geq \frac{3}{4}\log(1/\alpha) \Bigg | A^1, \ldots, A^{\nu}\right\} \nonumber \\
		\leq & \frac{\log(1/\alpha)}{8\kappa^2_0 n \rho} \max_{1 \leq l \leq \frac{\log(1/\alpha)}{8\kappa^2_0 n \rho}} \esssup \mathbb{P}_{\kappa, \nu} \left\{Z_{\nu, \nu + l} \geq \frac{3}{4}\log(1/\alpha) \Bigg | A^1, \ldots, A^{\nu}\right\} \nonumber \\
		\leq &  \frac{\frac{\log(1/\alpha)}{8\kappa^2_0 n \rho} } {\exp\left\{\frac{3}{4} \log(1/\alpha)\right\}}\max_{1 \leq l \leq \frac{\log(1/\alpha)}{8\kappa^2_0 n \rho}} \esssup \mathbb{E}_{\kappa, \nu} \left\{ \exp(Z_{\nu, \nu + l}) \Bigg | A^1, \ldots, A^{\nu}\right\}. \label{eq-lb-pf-v1-0}
	\end{align}

\medskip
\noindent \textbf{Step 3.1.}  Note that for any $l \in \{1, \ldots, \log(1/\alpha)(8\kappa^2_0 n \rho)^{-1}\}$, it holds that
	\begin{align}
		& \mathbb{E}_{\kappa, \nu}	\left\{\exp(Z_{\nu, \nu + l}) \Bigg | A^1, \ldots, A^{\nu}\right\} = \mathbb{E}_{\kappa, \nu}	\left\{ \left(\frac{P_{\kappa_0, \nu}^{\nu + l}}{P_{\kappa_0, \infty}^{\nu + l}}\right) \Bigg | A^1, \ldots, A^{\nu}\right\}. \label{eq-lb-pf-v1-1}
	\end{align}
	
In addition, letting $\zeta = \rho/2$ and $v_i$ be the $i$th entry of any vector $v$, we have that
	\begin{align}
		& \mathbb{E}_{P_1} \left(\frac{P_1}{P_0}(A)\right) = \mathbb{E}_u \mathbb{E}_{A|u} \left(\frac{1}{2^n} \sum_{v \in \{\pm 1\}} \frac{P_{1, v}}{P_0}(A)\right) \label{eq-question}\\
		= &	\mathbb{E}_u \mathbb{E}_{A|u} \left\{\frac{1}{2^n} \sum_{v \in \{\pm 1\}} \prod_{1\leq i < j \leq n} \left(\frac{\zeta + \kappa_0\rho v_i v_j}{\zeta}\right)^{A_{ij}}\left(\frac{1-\zeta-\kappa_0\rho v_i v_j}{1 - \zeta}\right)^{1-A_{ij}}\right\} \nonumber \\
		= &	\mathbb{E}_u \left\{\frac{1}{2^n} \sum_{v \in \{\pm 1\}} \prod_{1\leq i < j \leq n} \left\{1 + \frac{\kappa_0^2\rho^2 u_i u_jv_i v_j}{\zeta(1-\zeta)}\right\} \right\} \nonumber \\
		\leq & \mathbb{E}_u \left\{\frac{1}{2^n} \sum_{v \in \{\pm 1\}} \prod_{i, j = 1}^n \left\{1 + \frac{\kappa_0^2\rho^2 u_i u_jv_i v_j}{\zeta(1-\zeta)}\right\} \right\}, \label{eq-aaaaaaaaaa}
	\end{align}	
	where in \eqref{eq-aaaaaaaaaa}, $u$ is a random vector with entries being independent Rademacher random variables.  We further have that 
	\begin{align}
		\eqref{eq-aaaaaaaaaa} \leq  \mathbb{E}_u \mathbb{E}_v\prod_{i, j = 1}^n \exp \left\{\frac{\kappa_0^2\rho^2 u_i u_jv_i v_j}{\zeta(1-\zeta)}\right\} = \mathbb{E}_{u, v}\exp\left\{\frac{4\kappa_0^2 \rho}{2-\rho} (u^{\top}v)^2\right\},	\label{eq-bbbbbbbb}
	\end{align}
	where $u$ and $v$ are independent random vectors with entries being independent Rademacher random variables.  Finally, we have that 
	\begin{align}
		\eqref{eq-bbbbbbbb}	= \mathbb{E}_u\exp\left\{\frac{4\kappa_0^2 \rho}{2-\rho} (u^{\top}1)^2\right\}, \label{eq-lb-pf-v1-2}
	\end{align}
	where $1$ is an all-one $n$-dimensional vector.  \Cref{eq-lb-pf-v1-2} is due to the fact that for each $i \in \{1, \ldots, n\}$, $u_iv_i$ has the same distribution as $u_i$.

\medskip
\noindent \textbf{Step 3.2.}  Let 
	\[
		\varepsilon_n = \left(\frac{\sum_{i = 1}^n u_i}{n}\right)^2.
	\]
	Then we have that for any $x > 0$, due to Hoeffding's inequality that
	\[
		\mathbb{P}\left\{\varepsilon_n > x\right\} \leq 2\exp(-2nx).
	\]
	Therefore
	\begin{align}
		& \mathbb{E}_u\exp\left\{\frac{4\kappa_0^2 n^2 \rho}{2-\rho} (u^{\top}1)^2\right\} = \int_0^{\infty} \mathbb{P}\left\{\exp\left(\frac{4\kappa_0^2n^2 \rho}{2-\rho} \varepsilon_n\right) > x\right\}	\, dx \nonumber\\
		\leq & 1 + \int_1^{\infty} \mathbb{P}\left\{\varepsilon_n > \log(x) \frac{2-\rho}{4\kappa_0^2n^2 \rho}\right\} \, dx \leq 1 + 2\int_1^{\infty} \exp \left\{\log(x) \frac{\rho-2}{2\kappa_0^2n\rho}\right\} \, dx \nonumber \\
		\leq & 1 + \frac{2}{1 + \frac{\rho-2}{2\kappa_0^2n\rho}} x^{1 + \frac{\rho-2}{2\kappa_0^2n\rho}} \Bigg|_1^{\infty} = 1 - \frac{2}{1 + \frac{\rho-2}  {2\kappa_0^2n\rho}} \leq \exp\left(4\kappa_0^2n \rho\right), \label{eq-lb-pf-v1-3}
	\end{align}
	provided that 
	\begin{equation}\label{eq-lb-pf-cond-1}
		\rho + 2\kappa^2_0n\rho < 1.
	\end{equation}

\medskip
\noindent \textbf{Step 3.3.}  Combining \eqref{eq-lb-pf-v1-0}, \eqref{eq-lb-pf-v1-1}, \eqref{eq-lb-pf-v1-2} and \eqref{eq-lb-pf-v1-3}, we have that	
	\begin{align}
		& \mathbb{P}_{\kappa, \nu}\left\{\nu < T < \nu + \frac{\log(1/\alpha)}{8\kappa^2_0 n \rho}, \quad Z_{\nu, T} \geq \frac{3}{4}\log(1/\alpha) \Bigg | T \geq \nu\right\} \nonumber \\
		\leq & \frac{\log(1/\alpha)}{8\kappa^2_0 n \rho}\frac{\exp \left\{\frac{\log(1/\alpha)}{8\kappa^2_0 n \rho} 4\kappa_0^2n \rho\right\}}{\exp\left\{\frac{3}{4} \log(1/\alpha)\right\}} \leq \alpha^{1/8}, \label{eq-lb-pf-v1-222}
	\end{align}
	provided that 
	\begin{equation}\label{eq-lb-pf-cond-2}
		\alpha^{1/8} \log(1/\alpha) < 8 \kappa_0^2 n \rho.
	\end{equation}

\medskip	
\noindent \textbf{Step 4.}  Combining \eqref{eq-lb-pf-v1-111} and \eqref{eq-lb-pf-v1-222}, we then have
	\[
		\sup_{\nu \geq 1} \mathbb{P}_{\kappa, \nu} 	\left\{\nu < T < \nu + \frac{\log(1/\alpha)}{8\kappa^2_0 n \rho}\right\} \leq 2\alpha^{1/8}.
	\]	

Then it holds that
	\begin{align*}
		& \mathbb{E}_{\kappa, \Delta} \{(T - \Delta)_+\} \geq \frac{\log(1/\alpha)}{8\kappa^2_0 n \rho}\mathbb{P}_{\kappa, \nu} 	\left\{T - \Delta \geq \frac{\log(1/\alpha)}{8\kappa^2_0 n \rho}\right\} \\
		= & \frac{\log(1/\alpha)}{8\kappa^2_0 n \rho}\mathbb{P}_{\kappa, \nu} \left[\mathbb{P}_{\kappa, \nu} 	\left\{T > \Delta\right\} - \mathbb{P}_{\kappa, \nu} 	\left\{\Delta < T < \Delta + \frac{\log(1/\alpha)}{8\kappa^2_0 n \rho}\right\} \right] \\
		\geq & \frac{\log(1/\alpha)}{8\kappa^2_0 n \rho} (1 - \alpha - 2\alpha^{1/8}) \geq \frac{\log(1/\alpha)}{16\kappa^2_0 n \rho},
	\end{align*}
	provided that
	\begin{equation}\label{eq-lb-pf-cond-3}
		\alpha + 2\alpha^{1/8} < 1/2.
	\end{equation}
	
\medskip	
\noindent \textbf{Step 5.}  Finally, we are to show the set of parameters satisfying \eqref{eq-lb-pf-cond-1}, \eqref{eq-lb-pf-cond-2} and \eqref{eq-lb-pf-cond-3} is not an empty set.  For instance, we take $\rho = 1/4$ and $\kappa_0^2n = 1/2$, then \eqref{eq-lb-pf-cond-1} holds.  With this choice, \eqref{eq-lb-pf-cond-2} holds if 
	\begin{equation}\label{eq-lb-pf-cond-4}
	\alpha^{1/8} < 1/\log(1/\alpha)
	\end{equation}
	holds.  Since $\alpha < \alpha^{1/8}$, \eqref{eq-lb-pf-cond-3} holds if 
	\begin{equation}\label{eq-lb-pf-cond-5}
	\alpha^{1/8} < 1/6 
	\end{equation}	
	holds.  The choice of $\alpha = 30^{-8}$ satisfies both \eqref{eq-lb-pf-cond-4} and \eqref{eq-lb-pf-cond-5}.
	
\medskip	
\noindent \textbf{Step 6.}  We have now shown that when $r \lesssim \sqrt{n}$, it holds that 
	\begin{equation}\label{eq-lb-final-1}
		\inf_{\widehat{t} \in \mathcal{D}}\sup_{P_{\kappa, \Delta}} \mathbb{E}_{P}\left\{(\widehat{t} - \Delta)_+\right\} \geq \frac{c\log(1/\alpha)}{\kappa^2_0 n \rho},
	\end{equation}
	with an absolute constant $c > 0$.
	
\medskip
\noindent \textbf{Case 2: $r \gtrsim \sqrt{n}$.}

\noindent \textbf{Step 1 - Setup.} We assume that the networks are generated as follows.  Prior to the change point, if there exists any, the adjacency matrices are generated independently from the distribution $P_0$, which has the graphon matrix
	\[
		\Theta_0 = (\rho/2)_{i, j = 1}^n.
	\]
	If there exists a change point, then the adjacency matrices after the change point are generated independently from the distribution 
	\[
		P_1 = \frac{1}{2^{(r^2/2)}} \sum_{Z \in \mathcal{Z}} P_{1, Z},
	\]
	where the graphon of the distribution $P_{1, Z}$ is $\rho/2 \mathbbm{1}\mathbbm{1}^{\top} + \kappa_0 \rho n/r Z$ and the collection $\mathcal{Z}$ is the set for all symmetric matrices satisfying $Z_{ij} = 0$, if $\max\{i, j\} > r$, and all the upper triangular matrix of $Z_{(1:r), (1:r)}$ are independent Radamacher random variables.
	
In order to show that $P_1$ is a probability distribution, it suffices to justify that for any $Z \in \mathcal{Z}$, $P_{1, Z}$ is a suitable probability distribution.  
	\begin{itemize}
	\item Firstly, we have that 
	\begin{align*}
		\left\|\kappa_0 \rho n/r Z\right\|_{\mathrm{F}} = \kappa_0 \rho n.
	\end{align*}
	\item Secondly, we have that	 the entries of the matrix $Z$ are all in the set $\{0, \pm 1\}$.  This means that provided
	\begin{equation}\label{eq-first-condition-lb}
		\kappa_0 n/ r < 1/2,
	\end{equation}
	all the entries of $P_{1, Z, E}$ are in the interval $[0, \rho]$. 
	\item Lastly, the rank of the matrix $Z$ is upper bounded by $r$.
	\end{itemize}

For any $M \in \mathbb{N}$, let $P^M$ be the restriction of a distribution $P$ on $\mathcal{F}_M$, i.e.~the $\sigma$-filed generated by the observations $\{A(t)\}_{i = 1}^M$.  For notational simplicity, in this proof, the adjacency matrices $A(t)$'s will be denoted as $A^t$'s.  For any $\nu \geq 1$ and $M \geq \nu$, we have that for any $M \geq \Delta$, let
	\[
		Z_{\nu, M} = \log\left(\frac{P_{\kappa_0, \nu}^M}{P_{\kappa_0, \infty}^M}\right),
	\]	
	where $P_{\kappa, \infty}$ indicates the distribution under which there is no change point.

\medskip
\noindent \textbf{Step 2 - When $Z_{\nu, T}$ is upper bounded.}  For any $\nu \geq 1$, define the event
	\[
		\mathcal{E}_{\nu} = \left\{\nu < T < \nu + \frac{r^2/n\log(1/\alpha)}{8\kappa^2_0 n \rho}, \quad Z_{\nu, T} < \frac{3}{4}\log(1/\alpha)\right\}.
	\]	
	Then we have
	\begin{align}
		\mathbb{P}_{\kappa, \nu}(\mathcal{E}_{\nu}) = \frac{P_{\kappa_0, \nu}}{P_{\kappa, \infty}}(\mathcal{E}_{\nu}) P_{\kappa, \infty}(\mathcal{E}_{\nu}) \leq \alpha^{-3/4}\alpha = \alpha^{1/4},\label{eq-lb-pf-v1-111}
	\end{align}
	where the inequality follows from the definition of $\mathcal{D}$ and $\mathcal{E}_{\nu}$.

\medskip
\noindent \textbf{Step 3 - When $Z_{\nu, T}$ is lower bounded.} 
For any $\nu \geq 1$ and $T \in \mathcal{D}$, since $\{T \geq \nu\} \in \mathcal{F}_{\nu-1}$, we have that
	\begin{align}
		& \mathbb{P}_{\kappa, \nu}\left\{\nu < T < \nu + \frac{r^2/n\log(1/\alpha)}{8\kappa^2_0 n \rho}, \quad Z_{\nu, T} \geq \frac{3}{4}\log(1/\alpha) \Bigg | T \geq \nu\right\}	\nonumber \\
		\leq & \esssup \mathbb{P}_{\kappa, \nu} \left\{\max_{1 \leq l \leq \frac{r^2/n\log(1/\alpha)}{8\kappa^2_0 n \rho}} Z_{\nu, \nu + l} \geq \frac{3}{4}\log(1/\alpha) \Bigg | A^1, \ldots, A^{\nu}\right\} \nonumber \\
		\leq & \frac{r^2/n\log(1/\alpha)}{8\kappa^2_0 n \rho} \max_{1 \leq l \leq \frac{r^2/n\log(1/\alpha)}{8\kappa^2_0 n \rho}} \esssup \mathbb{P}_{\kappa, \nu} \left\{Z_{\nu, \nu + l} \geq \frac{3}{4}\log(1/\alpha) \Bigg | A^1, \ldots, A^{\nu}\right\} \nonumber \\
		\leq &  \frac{\frac{r^2/n\log(1/\alpha)}{8\kappa^2_0 n \rho} } {\exp\left\{\frac{3}{4} \log(1/\alpha)\right\}}\max_{1 \leq l \leq \frac{r^2/n\log(1/\alpha)}{8\kappa^2_0 n \rho}} \esssup \mathbb{E}_{\kappa, \nu} \left\{ \exp(Z_{\nu, \nu + l}) \Bigg | A^1, \ldots, A^{\nu}\right\}. \label{eq-lb-pf-v1-0-111}
	\end{align}

\medskip
\noindent \textbf{Step 2.1.}  Note that for any $l \in \{1, \ldots, r^2/n\log(1/\alpha)(8\kappa^2_0 n \rho)^{-1}\}$, it holds that
	\begin{align}
		& \mathbb{E}_{\kappa, \nu}	\left\{\exp(Z_{\nu, \nu + l}) \Bigg | A^1, \ldots, A^{\nu}\right\} = \mathbb{E}_{\kappa, \nu}	\left\{ \left(\frac{P_{\kappa_0, \nu}^{\nu + l}}{P_{\kappa_0, \infty}^{\nu + l}}\right) \Bigg | A^1, \ldots, A^{\nu}\right\}. \label{eq-lb-pf-v1-1-1111}
	\end{align}
	
In addition, letting $\zeta = \rho/2$, $U_{ij} = \kappa_0 \rho n/r Z_{ij}$ and $V_{ij} = \kappa_0 \rho n/r W_{ij}$, we have that
	\begin{align}
		& \mathbb{E}_{P_1} \left(\frac{P_1}{P_0}\right) = \mathbb{E}_{Z} \mathbb{E}_{A|Z} \left(\frac{1}{2^{(r^2/2)}} \sum_{W \in \mathcal{Z}} \frac{P_{1, W}}{P_0}\right) \nonumber \\
		= &	\mathbb{E}_{Z} \mathbb{E}_{A|Z}  \left\{\frac{1}{2^{(r^2/2)}}\sum_{W \in \mathcal{Z}}  \prod_{1\leq i < j \leq n} \left(\frac{\zeta + V_{ij}}{\zeta}\right)^{A_{ij}}\left(\frac{1-\zeta-V_{ij}}{1 - \zeta}\right)^{1-A_{ij}}\right\} \nonumber \\
		= & \mathbb{E}_{Z} \left\{\frac{1}{2^{(r^2/2)}} \sum_{W \in \mathcal{Z}} \prod_{1\leq i < j \leq n} \left\{1 + \frac{U_{ij}V_{ij}}{\zeta(1-\zeta)}\right\} \right\} \nonumber \\ 
		\leq & \mathbb{E}_{Z} \left\{\frac{1}{2^{(r^2/2)}} \sum_{W \in \mathcal{Z}} \prod_{i, j = 1}^n \left\{1 + \frac{U_{ij}V_{ij}}{\zeta(1-\zeta)}\right\} \right\}  \nonumber \\
		\leq & \mathbb{E}_{Z} \mathbb{E}_{W} \prod_{i, j = 1}^n \exp \left\{\frac{U_{ij}V_{ij}}{\zeta(1-\zeta)}\right\} = \mathbb{E}_{Z} \mathbb{E}_{W} \exp \left\{\frac{\kappa_0^2 \rho n^2/r^2}{2-\rho} \langle Z, W\rangle\right\} \label{eq-pf-lb-alter-1-1111}.
	\end{align}
	
\medskip
\noindent \textbf{Step 2.2} For any fixed $Z, W \in \mathcal{Z}$, it holds that
	\begin{align}
		\langle Z, W\rangle = 2 z^{\top}w, \label{eq-pf-lb-alter-2}
	\end{align}
	where $z$ and $w$ are vectorised upper triangular parts of $Z$ and $W$, respectively.  The vectors $z$ and $w$ are all $r^2/2$-dimensional vectors, consisting of only $\pm 1$.

\medskip
\noindent \textbf{Step 2.3.}  Due to \eqref{eq-pf-lb-alter-2}, it holds that
	\begin{align*}
		\mathbb{E}_{P_1} \left(\frac{P_1}{P_0}\right)  \leq \mathbb{E}_Z\mathbb{E}_W \exp \left\{\frac{2\kappa_0^2 \rho n^2/r^2}{2-\rho} (z^{\top}w)\right\} = \mathbb{E}_Z \exp \left\{\frac{2\kappa_0^2 \rho n^2/r^2}{2-\rho} (z^{\top}1)\right\}. 
	\end{align*}

Let 
	\[
		\varepsilon = \frac{\sum_{i = 1}^{r^2/2} z_{i}}{r^2/2}.
	\]
	Then we have that for any $x > 1$, due to Hoeffding's inequality that
	\[
		\mathbb{P}\left\{\varepsilon > x\right\} \leq \exp(- 2x^2 r^2) \leq \exp(- 2x r^2).
	\]
	We have that
	\begin{align*}
		& \mathbb{E}_Z \exp \left\{\frac{2\kappa_0^2 \rho n^2/r^2}{2-\rho} (z^{\top}1)\right\} = \int_0^{\infty} \mathbb{P}\left\{\exp\left(\frac{\kappa_0^2 \rho n^2}{2-\rho} \varepsilon_n\right) > x\right\}	\, dx \nonumber\\
		\leq & a + \int_a^{\infty} \mathbb{P}\left\{\varepsilon_n > \log(x) \frac{2-\rho}{\kappa_0^2 \rho n^2}\right\} \, dx \leq a + \int_a^{\infty} \exp \left\{-\log(x) \frac{2r^2(2-\rho)}{\kappa_0^2n^2\rho}\right\} \, dx \nonumber \\
		\leq & a - \frac{1}{1 - \frac{2r^2(2-\rho)}{\kappa_0^2n^2\rho}} = a + \frac{\kappa_0^2n^2\rho}{2r^2 - \kappa_0^2n^2\rho} \leq a + \frac{\kappa_0^2n^2\rho}{r^2}, \
	\end{align*}
	provided that 
	\begin{equation}\label{eq-lb-pf-cond-1-22222}
		\kappa_0^2 n^2 \rho < r^2,
	\end{equation}
	where
	\[
		a = \exp\left\{\frac{2-\rho}{\kappa_0^2\rho n^2}\right\}.
	\]

Then we have
	\begin{align}\label{eq-lb-pf-v1-3-1111}
		\mathbb{E}_{P_1} \left(\frac{P_1}{P_0}\right) \leq \exp\left\{\frac{2-\rho}{\kappa_0^2\rho n^2}\right\} \left\{1 + \exp\left\{\frac{\rho-2}{\kappa_0^2\rho n^2}\right\}\frac{\kappa_0^2n^2\rho}{r^2}\right\} \leq \exp\left\{\frac{2-\rho}{\kappa_0^2\rho n^2}\right\} \exp\left\{\frac{\kappa_0^2n^2\rho}{r^2}\right\}.
	\end{align}

\medskip	
\noindent \textbf{Step 3.}  Combining \eqref{eq-lb-pf-v1-0-111}, \eqref{eq-lb-pf-v1-1-1111} and \eqref{eq-lb-pf-v1-3-1111}, we then have
	\begin{align*}
		& \mathbb{P}_{\kappa, \nu}\left\{\nu < T < \nu + \frac{r^2/n\log(1/\alpha)}{8\kappa^2_0 n \rho}, \quad Z_{\nu, T} \geq \frac{3}{4}\log(1/\alpha) \Bigg | T \geq \nu\right\}	 \\
		\leq & \frac{\frac{r^2/n\log(1/\alpha)}{8\kappa^2_0 n \rho} } {\exp\left\{\frac{3}{4} \log(1/\alpha)\right\}} \exp\left\{\frac{r^2/n\log(1/\alpha)}{8\kappa^2_0 n \rho}\frac{2-\rho}{\kappa_0^2\rho n^2}\right\} \exp\left\{\frac{r^2/n\log(1/\alpha)}{8\kappa^2_0 n \rho}\frac{\kappa_0^2n^2\rho}{r^2}\right\} \\
		\leq & \alpha^{3/4} \alpha^{-1/8} \alpha^{-1/8} \alpha^{-1/8} \leq \alpha^{1/4},
	\end{align*}
	provided that 
	\begin{equation}\label{eq-third-condition-lb}
		\frac{r\log(1/\alpha)}{8\kappa^2_0 n^2 \rho} \leq \alpha^{-1/8} \quad \mbox{and} \quad \frac{r^2}{8 \kappa_0^4n^4\rho^2} \leq 1/8.
	\end{equation}

Then it holds that
	\begin{align*}
		& \mathbb{E}_{\kappa, \Delta} \{(T - \Delta)_+\} \geq \frac{r^2/n \log(1/\alpha)}{8\kappa^2_0 n \rho}\mathbb{P}_{\kappa, \nu} 	\left\{T - \Delta \geq \frac{r^2/n\log(1/\alpha)}{8\kappa^2_0 n \rho}\right\} \\
		= & \frac{r^2/n\log(1/\alpha)}{8\kappa^2_0 n \rho}\mathbb{P}_{\kappa, \nu} \left[\mathbb{P}_{\kappa, \nu} 	\left\{T > \Delta\right\} - \mathbb{P}_{\kappa, \nu} 	\left\{\Delta < T < \Delta + \frac{r^2/n\log(1/\alpha)}{8\kappa^2_0 n \rho}\right\} \right] \\
		\geq & \frac{r^2/n\log(1/\alpha)}{8\kappa^2_0 n \rho} (1 - \alpha - 2\alpha^{1/4}) \geq \frac{r^2/n\log(1/\alpha)}{16\kappa^2_0 n \rho},
	\end{align*}
	provided that
	\begin{equation}\label{eq-lb-pf-cond-3}
		\alpha + 2\alpha^{1/4} < 1/2.
	\end{equation}
	
\medskip	
\noindent \textbf{Step 4.}  Finally, we are to show the set of parameters satisfying \eqref{eq-first-condition-lb}, \eqref{eq-lb-pf-cond-1-22222}, \eqref{eq-third-condition-lb} and \eqref{eq-lb-pf-cond-3} is not an empty set.  For instance, we take $\rho = 1/5$, $r = 30$ and $\kappa_0n = 14$, then \eqref{eq-first-condition-lb}, \eqref{eq-lb-pf-cond-1-22222} and the second condition in \eqref{eq-third-condition-lb} hold.  With this choice, the first half of \eqref{eq-third-condition-lb} and \eqref{eq-lb-pf-cond-3} hold with the choice of $\alpha = 1/2000$.  This shows that the choice is not empty.  In addition,  provided that $n \leq 900$, we have that $r \geq \sqrt{n}$.

\medskip	
\noindent \textbf{Step 5.}  We have now shown that when $r \gtrsim \sqrt{n}$, it holds that 
	\begin{equation}\label{eq-lb-final-2}
		\inf_{\widehat{t} \in \mathcal{D}}\sup_{P_{\kappa, \Delta}} \mathbb{E}_{P}\left\{(\widehat{t} - \Delta)_+\right\} \geq \frac{c\log(1/\alpha) r^2/n}{\kappa^2_0 n \rho},
	\end{equation}
	with an absolute constant $c > 0$.
	
\medskip
Finally, combining \eqref{eq-lb-final-1} and \eqref{eq-lb-final-2}, we conclude the proof.	
\end{proof}

\bibliographystyle{ims}
\bibliography{ref}

\end{document}